\documentclass[english]{article}
\usepackage{lmodern}

\usepackage[T1]{fontenc}
\usepackage[utf8]{inputenc}
\usepackage{dsfont}
\usepackage[a4paper]{geometry}
\usepackage{babel}
\usepackage{amsmath}
\usepackage{mathtools}
\usepackage{mleftright}
\mleftright
\usepackage{amsthm}
\usepackage{amssymb}
\usepackage{setspace}
\usepackage{etoolbox}
\usepackage{needspace}
\BeforeBeginEnvironment{prop}{\Needspace{4\baselineskip}}
\BeforeBeginEnvironment{cor}{\Needspace{4\baselineskip}}
\BeforeBeginEnvironment{lem}{\Needspace{4\baselineskip}}
\usepackage[protrusion=true,expansion=true,tracking=false,kerning=true]{microtype}
\usepackage{enumitem}
\setlist{topsep=3pt plus 1pt,itemsep=2pt plus 1pt,parsep=0pt}
\usepackage{natbib}
\usepackage{caption}
\usepackage{booktabs}
\usepackage{xcolor}
\definecolor{linkblue}{RGB}{15,60,165}
\definecolor{thmnavy}{RGB}{15,50,100}
\definecolor{boxframe}{HTML}{1F6FA2}
\definecolor{boxfill}{HTML}{F2F7FB}
\definecolor{remframe}{HTML}{2F6F4F}
\definecolor{remfill}{HTML}{F1F7F3}
\colorlet{inkblue}{boxframe}
\newcommand{\Qref}[1]{\textcolor{inkblue}{(Q#1)}}
\usepackage{tikz}
\usepackage{graphicx}
\allowdisplaybreaks
\AtBeginDocument{\setlength{\abovedisplayskip}{4.5pt plus 2pt minus 3pt}\setlength{\belowdisplayskip}{4.5pt plus 2pt minus 3pt}\setlength{\abovedisplayshortskip}{3pt plus 2pt}\setlength{\belowdisplayshortskip}{5pt plus 2pt minus 2pt}}
\newcommand{\air}{\setlength{\abovedisplayskip}{8pt plus 2pt minus 2pt}\setlength{\belowdisplayskip}{8pt plus 2pt minus 2pt}\setlength{\abovedisplayshortskip}{6pt plus 2pt minus 2pt}\setlength{\belowdisplayshortskip}{6pt plus 2pt minus 2pt}}
\usepackage[unicode=true,bookmarks=true,breaklinks=false,pdfborder={0 0 0},colorlinks=true,pagebackref=true]{hyperref}
\hypersetup{pdftitle={Convex-Order Comparison under Bernstein Moment Bounds: Sharp Constants and Dimensional Obstructions},pdfauthor={El Mahdi Khribch and Badr-Eddine Ch\'erief-Abdellatif},linkcolor=linkblue,citecolor=linkblue,urlcolor=linkblue}
\renewcommand*{\backref}[1]{}
\renewcommand*{\backrefalt}[4]{\ifcase #1 \or \small (cited on page~#2)\else \small (cited on pages~#2)\fi}

\makeatletter
\AtBeginEnvironment{proof}{\setlength{\parskip}{0.3\baselineskip plus 2pt minus 1pt}}
\renewenvironment{proof}[1][\proofname]{\par
  \ifdim\lastskip<\medskipamount\removelastskip\medskip\fi
  \pushQED{\qed}%
  \normalfont\topsep\z@\partopsep\z@\relax
  \trivlist
  \item[\hskip\labelsep\bfseries\color{thmnavy}#1\@addpunct{.}]\ignorespaces
}{%
  \popQED\endtrivlist\@endpefalse
}
\makeatother

\newcommand{\remheadcolor}{thmnavy}
\newtheoremstyle{navyplain}{\topsep}{\topsep}{\itshape}{0pt}{\bfseries\color{thmnavy}}{.}{5pt plus 1pt minus 1pt}{\thmname{#1}\thmnumber{ #2}\thmnote{ {\normalfont\color{black}(#3)}}}
\newtheoremstyle{remstyle}{\topsep}{\topsep}{\normalfont}{0pt}{\bfseries\color{\remheadcolor}}{.}{5pt plus 1pt minus 1pt}{\thmname{#1}\thmnumber{ #2}\thmnote{ {\normalfont\color{black}(#3)}}}
\theoremstyle{navyplain}
\newtheorem{thm}{Theorem}
\newtheorem{prop}{Proposition}
\newtheorem{lem}{Lemma}
\newtheorem{cor}{Corollary}
\theoremstyle{remstyle}
\newtheorem{remark}{Remark}
\newtheorem{defn}{Definition}

\newcommand{\diff}{\mathrm{d}}
\newcommand{\E}{\mathbb{E}}
\newcommand{\Prob}{\mathbb{P}}
\newcommand{\R}{\mathbb{R}}
\newcommand{\one}{\mathds{1}}
\newcommand{\lecx}{\preceq_{\mathrm{cx}}}
\newcommand{\SL}{\Pi}
\newcommand{\Lap}{\mathcal{L}}
\newcommand{\SE}{\mathrm{SE}}
\newcommand{\SG}{\mathrm{S\Gamma}}
\newcommand{\eL}{e_{\Lap}}
\newcommand{\cE}{c^{\SE}}
\newcommand{\cG}{c^{\SG}}
\newcommand{\nuE}{\nu_{\SE}}
\newcommand{\nuG}{\nu_{\SG}}
\newcommand{\Cl}{\mathcal{C}}

\usepackage[most]{tcolorbox}
\tcbset{subgammabox/.style={enhanced,sharp corners,frame hidden,boxrule=0pt,boxsep=0pt,
  grow to left by=11pt,grow to right by=11pt,left=11pt,right=11pt,top=7pt,bottom=7pt,
  before skip=9pt plus 2pt,after skip=9pt plus 2pt}}
\newtcolorbox{mainbox}{subgammabox,breakable=false,colback=boxfill,borderline west={1pt}{0pt}{boxframe}}
\newtcolorbox{defbox}{subgammabox,breakable=false,colback=boxfill,borderline west={1pt}{0pt}{boxframe}}
\newtcolorbox{questionbox}{subgammabox,breakable=false,colback=boxfill,borderline west={1pt}{0pt}{boxframe}}
\newtcolorbox{remarkbox}{subgammabox,breakable=false,colback=remfill,borderline west={1pt}{0pt}{remframe},before upper={\renewcommand{\remheadcolor}{remframe}}}
\newtcolorbox{roadbox}{subgammabox,breakable=false,colback=boxfill,borderline west={1pt}{0pt}{boxframe}}

\title{\textbf{Convex-Order Comparison under Bernstein Moment Bounds: Sharp Constants and Dimensional Obstructions}}
\author{%
  El Mahdi Khribch\textsuperscript{1,2,$*$}\qquad
  Badr-Eddine Ch\'erief-Abdellatif\textsuperscript{2,3,$\dagger$}
  \\[6pt]
  {\small\textsuperscript{1}Laboratoire de Math\'ematiques d'Orsay, Universit\'e Paris-Saclay}
  \\
  {\small\textsuperscript{2}LPSM, Sorbonne Universit\'e, Universit\'e Paris Cit\'e}
  \\
  {\small\textsuperscript{3}CNRS}%
}
\date{November 2026}

\begin{document}
\maketitle
\begingroup
\renewcommand{\thefootnote}{\fnsymbol{footnote}}
\footnotetext[1]{\href{mailto:el-mahdi.khribch@universite-paris-saclay.fr}{\texttt{el-mahdi.khribch@universite-paris-saclay.fr}}}
\footnotetext[2]{\href{mailto:badr-eddine.cherief-abdellatif@cnrs.fr}{\texttt{badr-eddine.cherief-abdellatif@cnrs.fr}}}
\endgroup

\begin{abstract}
Recent comparison theorems show that sub-Gaussian control of the one-dimensional projections of a random vector gives a comparison in the convex order with a multiple of a Gaussian vector, with a constant that does not depend on the dimension. We study the same question under Bernstein-type bounds on the moment generating function, with the Laplace law as the benchmark. In dimension one, we determine the least multiple of the Laplace law for two classes with variance factor $\sigma^2$ and scale parameter $\alpha$: the sub-exponential class, which is a sub-Gaussian bound on a bounded window of frequencies, and the sub-Gamma class. For the sub-exponential class, the least multiple is $\sigma \vee \alpha$, and the comparison has an equality case, with a convex function that is not affine and finite expectations, if and only if $\alpha \le \sigma$. For the sub-Gamma class, the least multiple is strictly larger than $\sigma \vee \alpha$, has a variational characterization through two-point laws, and has an asymmetric two-point equality case. The proofs use martingale couplings, as in the easy direction of Strassen's theorem, and the Kearns--Saul inequality for Bernoulli variables. In dimension $d$, the least multiple of the rotation-invariant Laplace vector is finite, but it grows at least like $d^{1/4}$ for both classes, and the same lower bound holds for every benchmark normalized by its mean absolute projection, averaged over directions. Bernstein-type bounds thus give sharp comparisons in dimension one, but no comparison with a single normalized benchmark that does not depend on the dimension.
\end{abstract}

{\small\noindent\emph{Keywords:} convex order, sub-Gamma variables, sub-exponential variables, Bernstein's inequality, Laplace law, Kearns--Saul inequality, two-point laws, Strassen's theorem.}

\section{Introduction}
\label{sec:intro}

Bernstein's inequality is a basic tool to describe the fluctuations of sums of independent variables, in concentration of measure \citep[Section 2.8]{boucheron2013concentration}, empirical process theory \citep{massart2000about} and statistical learning \citep{bartlett2006empirical}. It improves on Hoeffding's inequality when the variance is small compared with the range, and the gain is already visible for an average. Let $X_1, \dots, X_n$ be independent and centered, with $|X_i| \le b$ and $\E[X_i^2] = v$. For $t > 0$, with probability at least $1 - e^{-t}$, Hoeffding's and Bernstein's inequalities give, respectively,
\begin{align}
    \frac{1}{n}\sum_{i=1}^n X_i \le b\sqrt{\frac{2t}{n}},
    \qquad\qquad
    \frac{1}{n}\sum_{i=1}^n X_i \le \sqrt{\frac{2vt}{n}} + \frac{bt}{3n},
    \label{eq:HB}
\end{align}
by \citet[Theorem 2]{hoeffding1963probability} and \citet[Theorem 2.10 and p.~38]{boucheron2013concentration}. Hoeffding's bound uses the largest possible variance $b^2$, that of the variable equal to $\pm b$ with probability $1/2$. Bernstein's bound replaces $b^2$ by $v$ and pays for it with a second term of order $b/n$. It has two regimes: the first term, which is Gaussian, dominates for small $t$, and the second, which is exponential, dominates for large $t$. For $b$ of order one and $v$ of order $1/n$, both terms are of order $1/n$, and in PAC-Bayes theory, for instance, this turns a rate $n^{-1/2}$ into the faster rate $n^{-1}$ \citep[Sections 3.2.5 and 4.2]{alquier2024user}.

Hoeffding's proof starts with a comparison of laws. Let $X$ be centered with $|X| \le b$, and let $\varepsilon$ be a Rademacher variable, which takes the values $\pm 1$ with probability $1/2$. A convex function $f$ lies below its chord on $[-b, b]$, so that
\begin{align}
    \E[f(X)] \le \frac{f(-b) + f(b)}{2} = \E[f(b\varepsilon)]
    \qquad \text{for every convex } f
    \label{eq:chord}
\end{align}
\citep[(4.1) and Lemma 1]{hoeffding1963probability}, which Hoeffding writes for the exponential function only. For integrable random variables $X$ and $Z$, we write $X \lecx Z$, and we say that $X$ is smaller than $Z$ in the convex order, when
\begin{align}
    \E[f(X)] \le \E[f(Z)]
    \qquad \text{for every convex } f \colon \R \to \R,
    \label{eq:cx}
\end{align}
both sides being well defined in $(-\infty, \infty]$, since a convex function is bounded from below by an affine function. Thus \eqref{eq:chord} reads $X \lecx b\varepsilon$ \citep[Lemma 4.3(i)]{bentkus2004hoeffding}. With $f(x) = e^{\lambda x}$, it gives the sub-Gaussian bound $\E[e^{\lambda X}] \le \cosh(b\lambda) \le e^{b^2\lambda^2/2}$, and the first bound in \eqref{eq:HB} is the Chernoff bound for the average of $n$ independent variables.

A comparison in the convex order controls much more than this exponential moment. If $X \lecx Z$, every upper bound on $\E[f(Z)]$ with $f$ convex holds for $X$: the absolute moments $\E[|X|^q]$ with $q \ge 1$, the exponential moments, and the stop-loss transform $\SL_X(t) \coloneqq \E[(X - t)_+]$. The comparison also passes to sums of independent variables \citep[Theorem 3.A.12(d)]{shaked2007stochastic}, so that every convex function of the average in \eqref{eq:HB} is bounded by its value at the average of independent copies of $b\varepsilon$. Strassen's theorem \citep[Theorem 8]{strassen1965existence}, see also \citet[Theorem 3.A.4]{shaked2007stochastic}, sheds light on this stability: $X \lecx Z$ if and only if there is a martingale coupling of $X$ and $Z$, that is, a coupling $(X', Z')$ with
\begin{align}
    \E[Z' \mid X'] = X' .
    \label{eq:strassen}
\end{align}
Then $X$ is an average of $Z$, and the easy direction of the equivalence is Jensen's inequality, $\E[f(X')] = \E[f(\E[Z' \mid X'])] \le \E[f(Z')]$. Couplings made independently for each summand add up, which gives the stability under sums.

Hoeffding's comparison needs a bounded variable. A remarkable result of \citet[Theorem 1.1]{vanhandel2025subgaussian}, deduced there by duality from a comparison principle of \citet[Corollary 2]{liu2025lifting}, shows that sub-Gaussian tails alone give a comparison in the convex order, in every dimension. Let $X$ be a centered random vector in $\R^d$ whose one-dimensional projections satisfy $\Prob(|\langle v, X\rangle| > t) \le 2e^{-t^2/2}$ for every unit vector $v$ and every $t \ge 0$. Then, with $G_d$ a standard Gaussian vector in $\R^d$,
\begin{align}
    \E[f(X)] \le \E[f(c^*\,G_d)]
    \qquad \text{for every convex } f \colon \R^d \to \R,
    \label{eq:vanhandel}
\end{align}
where the constant $c^*$ depends neither on $X$ nor on $d$. The assumption bears only on the tails of one-dimensional projections, while the conclusion bears on every convex function of the whole vector, which is stronger than the convex order of all its projections when $d \ge 2$ \citep{pinelis2025does}. For the suprema $x \mapsto \sup_{t \in T}\langle t, x\rangle$, the conclusion was known as a consequence of the majorizing measure theorem of \citet{talagrand1987regularity}; that it holds for every convex function appears to have been overlooked \citep[Section 1]{vanhandel2025subgaussian}. \citet[Theorem 0.1]{song2026short} gave a very short proof under a bound on the moment generating function. Their proof applies the comparison of expected semi-norms of \citet{talagrand1987regularity} to independent copies of the vector, following a sketch of \citet[Subsection 2.4.2]{song2026sum}. The structure of the laws dominated by a Gaussian vector, and its link with Talagrand's convexity problem, is studied by \citet{hua2026talagrand}.

The sharp constants followed in dimension one, where the least multiple of a standard Gaussian variable $G$ depends on the definition of a sub-Gaussian variable. Under the tail bound $2e^{-t^2/2}$, it is approximately $2.310$ \citep[Theorem 1 and Remark 2]{davispower2026sharp}. Under the bound on the moment generating function, which \citet[Remark 3]{davispower2026sharp} raise as a question, \citet[Theorem 1]{zhang2026sharpmgf} proved that
\begin{align}
    \E\left[e^{\lambda X}\right] \le e^{\sigma^2\lambda^2/2} \ \text{ for every } \lambda \in \R
    \quad\Longrightarrow\quad
    X \lecx \sigma\sqrt{\frac{\pi}{2}}\,G .
    \label{eq:zhang}
\end{align}
The constant $\sqrt{\pi/2} = 1/\E[|G|]$ cannot be lowered. The variable $\varepsilon$ of \eqref{eq:chord} satisfies the bound with $\sigma = 1$, since $\cosh\lambda \le e^{\lambda^2/2}$, and $\E[|\varepsilon|] = 1 = \E[\sqrt{\pi/2}\,|G|]$, so that the convex function $|x|$ gives an equality case of \eqref{eq:zhang}.

A natural question is thus what happens under Bernstein's hypothesis. For a bounded variable, Bernstein's inequality rests on a bound on the moment generating function. Its exponent is quadratic in the frequency $\lambda$ near zero and diverges at the endpoints of a bounded interval of frequencies, which we call the window. The tails are then Gaussian for small deviations and exponential for large ones. Does this hypothesis alone give a comparison in the convex order, as the sub-Gaussian bound does? Such a comparison would control every convex function of the variables and of their sums, such as the norms of sums $\sum_i X_i x_i$ with vector coefficients $x_i$ (Section~\ref{sec:sums}), and not only their tails. We call a benchmark a fixed law whose multiples we compare with the elements of a class, as $\varepsilon$ in \eqref{eq:chord} and the Gaussian law in \eqref{eq:vanhandel} and \eqref{eq:zhang}. The benchmark cannot be Gaussian here, because Bernstein's hypothesis allows exponential tails, and the simplest symmetric law with exponential tails is the Laplace law $\Lap$, with density $e^{-|x|}/2$.

We study this question in dimension one and in dimension $d$. In dimension one, we determine the sharp comparison with the Laplace law for two classes: the sub-exponential class, a sub-Gaussian bound on a bounded window of frequencies, and the sub-Gamma class. Both satisfy Bernstein's tail bound, but their least multiples differ. In dimension $d$, the least multiple of the rotation-invariant Laplace vector grows at least like $d^{1/4}$ for both classes, and the same holds for every benchmark normalized by the mean absolute value of its projections: the growth comes from the hypothesis, not from the choice of the benchmark. Bernstein's hypothesis has two standard forms, and each defines one of the two classes.

\begin{defbox}
\begin{defn}[Sub-Gamma class]
\label{def:SG}
Let $\sigma, \alpha > 0$. A random variable $X$ belongs to the sub-Gamma class $\SG(\sigma^2, \alpha)$ when
\begin{align}
    \E\left[e^{\lambda X}\right] \le \exp\left(\frac{\sigma^2\lambda^2}{2\left(1 - \alpha|\lambda|\right)}\right)
    \qquad \text{for every } |\lambda| < \frac{1}{\alpha} .
    \label{eq:SG}
\end{align}
\end{defn}
\end{defbox}

Definition~\ref{def:SG} is the two-sided version of the sub-Gamma condition of \citet[Section 2.4]{boucheron2013concentration}, with variance factor $\sigma^2$ and scale parameter $\alpha$; we call \eqref{eq:SG} the sub-Gamma bound. As above, the window is the interval of frequencies on which the bound is imposed. The exponent on the right side bounds the cumulant generating function $\log\E[e^{\lambda X}]$, and we call it the envelope. A centered variable with $|X| \le b$ and $\E[X^2] = v$ belongs to $\SG(v, b/3)$ \citep[Theorem 2.10 and p.~38]{boucheron2013concentration}, so that the variance gives the first parameter and the range the second. More generally, every centered variable with exponential moments near zero belongs to a sub-Gamma class whose variance factor is its variance \citep[Proposition 1]{power2025subgamma}, after \citet{pinelis2024existence}. By the Chernoff bound \citep[p.~29]{boucheron2013concentration}, every $X \in \SG(\sigma^2,\alpha)$ satisfies Bernstein's tail bound: for every $t \ge 0$,
\begin{align}
    \Prob\left(X \ge \sqrt{2\sigma^2 t} + \alpha t\right) \le e^{-t}
    \qquad\text{and}\qquad
    \Prob\left(X \le -\sqrt{2\sigma^2 t} - \alpha t\right) \le e^{-t} .
    \label{eq:bernstein}
\end{align}
The bound is Gaussian for small deviations and exponential for large ones, with the two regimes separated at $t = 2\sigma^2/\alpha^2$. The average in \eqref{eq:HB} belongs to $\SG(v/n, b/(3n))$ by Lemma~\ref{lem:sumclass}, and \eqref{eq:bernstein} gives the second bound in \eqref{eq:HB}. A second common hypothesis keeps the sub-Gaussian bound but imposes it only on a bounded window \citep[Definition 2.7]{wainwright2019high}, see also \citet[Proposition 2.7.1]{vershynin2018high}.

\begin{defbox}
\begin{defn}[Sub-exponential class]
\label{def:SE}
Let $\sigma, \alpha > 0$. A random variable $X$ belongs to the sub-exponential class $\SE(\sigma^2,\alpha)$ when
\begin{align}
    \E\left[e^{\lambda X}\right] \le \exp\left(\frac{\sigma^2\lambda^2}{2}\right)
    \qquad \text{for every } |\lambda| \le \frac{1}{\alpha} .
    \label{eq:SE}
\end{align}
\end{defn}
\end{defbox}

By Fatou's lemma, it makes no difference whether the endpoints of the window are included. Each class contains the other up to constants, see also \citet{pinelis2022improved} for sub-exponential variables. The Gaussian envelope lies below the sub-Gamma one, and on the half window $|\lambda| \le 1/(2\alpha)$ the sub-Gamma envelope is at most $\sigma^2\lambda^2$, so that
\begin{align}
    \SE(\sigma^2,\alpha) \subseteq \SG(\sigma^2,\alpha) \subseteq \SE(2\sigma^2, 2\alpha) .
    \label{eq:inclusions}
\end{align}
By \eqref{eq:inclusions}, both classes satisfy Bernstein's tail bound \eqref{eq:bernstein}, so this bound does not separate them. We call either of the bounds \eqref{eq:SG} and \eqref{eq:SE} Bernstein's hypothesis. The Rademacher variable $\varepsilon$ belongs to $\SE(1,\alpha)$ for every $\alpha > 0$, as noted after \eqref{eq:zhang}. A centered Gamma variable with shape $k$ and scale $\gamma$ belongs to $\SG(k\gamma^2, \gamma)$ \citep[pp.~28--29]{boucheron2013concentration} but to no sub-exponential class with scale $\gamma$, since its moment generating function is infinite at $1/\gamma$.

The benchmark rests on a representation of the sub-Gamma envelope. As observed by \citet[Proposition 3]{power2025subgamma}, the envelope is, on positive frequencies, the cumulant generating function of a centered compound Poisson variable: if $N$ has the Poisson law with mean $\sigma^2/(2\alpha^2)$ and, given $N$, the $(J_i)_{i \in [N]}$ are independent exponential variables with mean $\alpha$, where $[N] \coloneqq \{1, \dots, N\}$, then
\begin{align}
    \mathcal{P} \coloneqq \sum_{i \in [N]} J_i - \frac{\sigma^2}{2\alpha}
    \qquad\text{satisfies}\qquad
    \log\E\left[e^{\lambda \mathcal{P}}\right] = \frac{\sigma^2\lambda^2}{2(1 - \alpha\lambda)}
    \quad \text{for } \lambda < \frac{1}{\alpha} .
    \label{eq:poisson}
\end{align}
The sub-Gamma bound thus compares $X$ with $\mathcal{P}$ at positive frequencies, as the sub-Gaussian bound compares $X$ with $\sigma G$. Since the right side of \eqref{eq:poisson} lies below the envelope for $\lambda < 0$, both $\mathcal{P}$ and $-\mathcal{P}$ belong to $\SG(\sigma^2,\alpha)$. But $\mathcal{P}$ is bounded from below and $-\mathcal{P}$ is not, so $\mathcal{P}$ does not dominate $-\mathcal{P}$, nor the class. Since the class contains $-X$ with every $X$, a law $Z$ that dominates it satisfies $X \lecx Z$ and $X \lecx -Z$ for every element $X$, hence $X \lecx \varepsilon Z$, with $\varepsilon$ independent of $Z$, by averaging the two comparisons. We may therefore look for a symmetric benchmark. Its moment generating function must be infinite at $\pm 1/\alpha$, as those of $\mathcal{P}$ at $1/\alpha$ and of $-\mathcal{P}$ at $-1/\alpha$ are, so its tails are no lighter than exponential tails at the scale $\alpha$, and no multiple of $G$ dominates the class.

We take the standard Laplace variable $\Lap$, which has $\E[|\Lap|] = 1$. We show that a multiple of $\Lap$ dominates each of the two classes, and we find the least one. In dimension one, the Laplace law thus plays for Bernstein's hypothesis the role that the Gaussian law plays for the sub-Gaussian one. Like $G$ in the sub-Gaussian class, $\Lap$ belongs to a sub-Gamma class, since
\begin{align}
    \log\E\left[e^{\lambda\Lap}\right] = -\log\left(1 - \lambda^2\right) = \sum_{k \ge 1}\frac{\lambda^{2k}}{k} \le \sum_{j \ge 2}|\lambda|^j = \frac{\lambda^2}{1 - |\lambda|}
    \qquad \text{for every } |\lambda| < 1,
    \label{eq:Linclass}
\end{align}
so that $\Lap \in \SG(2,1)$. It is also a Gaussian law with a random variance: with $W$ a standard exponential variable independent of $G$,
\begin{align}
    \Lap \overset{d}{=} \sqrt{2W}\,G ,
    \label{eq:mixture}
\end{align}
since both sides have the moment generating function $\E[e^{\lambda^2 W}] = 1/(1 - \lambda^2)$; Sections~\ref{sec:sums} and~\ref{sec:vectors} use this representation. Its tail is exponential for all deviations, so it fits only the exponential regime of \eqref{eq:bernstein}, and Section~\ref{sec:discussion} discusses what this choice costs.

We therefore look for the least multiples
\begin{equation}
\begin{aligned}
    \cG(\sigma, \alpha) &\coloneqq \inf\bigl\{ c > 0 \colon X \lecx c\Lap \ \text{ for every } X \in \SG(\sigma^2, \alpha) \bigr\},
    \\
    \cE(\sigma, \alpha) &\coloneqq \inf\bigl\{ c > 0 \colon X \lecx c\Lap \ \text{ for every } X \in \SE(\sigma^2, \alpha) \bigr\},
\end{aligned}
\label{eq:question}
\end{equation}
and for the elements of the classes that attain them. Each infimum in \eqref{eq:question} is a minimum when it is finite (Proposition~\ref{prop:reduction}). We say that the least multiple $c$ is attained when the comparison has an equality case, that is, an element $X$ of the class and a convex function $f$ that is not affine with $\E[f(X)] = \E[f(c\Lap)] < \infty$. This is stronger than asking that some element be dominated by no smaller multiple. For $\alpha > \sigma$, the class contains an element that no multiple below $\alpha$ dominates (Lemma~\ref{lem:edge}), because its moment generating function is infinite beyond the window. Yet the constant $\alpha$ has no equality case (Theorem~\ref{thm:SE}(iii)).

Two elementary arguments give the lower bound $c \ge \sigma \vee \alpha$, one for each parameter. The scale parameter $\alpha$ is forced by the exponential moments. Since $x \mapsto e^{\lambda x}$ is convex, a variable that dominates the class has an infinite moment generating function at every frequency where some element of the class does. For every $|\lambda| > 1/\alpha$, the class has such an element: $\mathcal{P}$ or $-\mathcal{P}$ of \eqref{eq:poisson} in the sub-Gamma class, and the variable of Lemma~\ref{lem:edge} in the sub-exponential class. The moment generating function of $c\Lap$ is finite for $|\lambda| < 1/c$, so $1/c \le 1/\alpha$, that is, $c \ge \alpha$. The parameter $\sigma$ is forced by the first absolute moment. The variable $\sigma\varepsilon$ belongs to both classes, and the convex function $|x|$ gives $\sigma = \E[|\sigma\varepsilon|] \le \E[|c\Lap|] = c$. It therefore remains to find upper bounds and to decide whether the constants are attained.

The sub-Gamma bound departs from the sub-Gaussian bound $e^{\sigma^2\lambda^2/2}$ in two ways: its window is bounded, and its envelope is larger inside the window. We measure the cost of a change by how far it brings the least multiple above $\sigma \vee \alpha$. The sub-exponential class has only the first change, so comparing the two classes separates the two costs. We therefore ask three questions.

\begin{questionbox}
\begin{enumerate}
\item[\Qref{1}] What is the least multiple of $\Lap$ that dominates each class, and is it attained?
\item[\Qref{2}] What does each change to the sub-Gaussian bound cost in the convex order?
\item[\Qref{3}] What remains of the comparison for sums of independent variables, the object of Bernstein's inequality, and for random vectors, the object of the theorem of \citet{vanhandel2025subgaussian}?
\end{enumerate}
\end{questionbox}

\paragraph{Our results.} For \Qref{1}, both classes are dominated in the convex order by a multiple of $\Lap$, and the least multiples satisfy
\begin{align}
    \cE(\sigma,\alpha) = \sigma \vee \alpha,
    \qquad\qquad
    \sigma \vee \alpha < \cG(\sigma,\alpha) \le \frac{\alpha + \sqrt{\alpha^2 + 4\sigma^2}}{2} \le (\sigma + \alpha) \wedge \Bigl(\alpha + \frac{\sigma^2}{\alpha}\Bigr)
    \label{eq:mainresults}
\end{align}
(Theorems~\ref{thm:SE} and~\ref{thm:SG}). The sub-exponential constant is attained if and only if $\alpha \le \sigma$. In that case it equals $\sigma$, and the element $\sigma\varepsilon$ attains it with the convex function $|x|$, since $\E[|\sigma\varepsilon|] = \sigma = \E[|\sigma\Lap|]$. The sub-Gaussian constant $\sigma\sqrt{\pi/2} = \sigma/\E[|G|]$ of \citet{zhang2026sharpmgf} is attained in the same way, and $\sigma = \sigma/\E[|\Lap|]$ is its Laplace analogue. When $\alpha > \sigma$, the constant $\alpha$ has no equality case. The sub-Gamma constant is always attained, by an asymmetric two-point law. Theorem~\ref{thm:SG}(i) gives it by a variational formula over two-point laws (Proposition~\ref{prop:reduction} and Section~\ref{sec:SG}). Table~\ref{tab:SG} evaluates this formula numerically; for instance, $\cG(1,1) \approx 1.415$. Each comparison gives bounds for every convex function. For instance, with $c \coloneqq \cG(\sigma,\alpha)$, every $X \in \SG(\sigma^2,\alpha)$ satisfies
\begin{align*}
    \E\left[|X|^q\right] \le \Gamma(q+1)\,c^q \quad (q \ge 1),
    \qquad
    \SL_X(t) \le \frac{c}{2}\,e^{-t/c} \quad (t \ge 0),
\end{align*}
and the same holds in the sub-exponential class with $c = \sigma \vee \alpha$.

For \Qref{2}, the identity $\cE(\sigma,\alpha) = \sigma \vee \alpha$ shows that the bounded window costs nothing beyond the elementary bound. The larger envelope leads to a strictly larger constant, although both classes satisfy Bernstein's tail bound \eqref{eq:bernstein}. This cost vanishes at both ends of the range of the ratio $\alpha/\sigma$. This ratio determines the deviation $t = 2\sigma^2/\alpha^2$ at which Bernstein's tail bound \eqref{eq:bernstein} changes from its Gaussian regime to its exponential regime. As $\alpha/\sigma \to 0$, the Gaussian regime covers every deviation, the class decreases to the sub-Gaussian class with variance factor $\sigma^2$, and $\cG(\sigma,\alpha) \to \sigma$. As $\alpha/\sigma \to \infty$, the exponential regime covers every deviation, and $\cG(\sigma,\alpha)/\alpha \to 1$. Proposition~\ref{prop:ends} gives the cost at each end, to first order: it is $(9/16)\,\alpha^3/\sigma^2$ as $\alpha/\sigma \to 0$, and $\tau_0\,\sigma^2/\alpha$ as $\alpha/\sigma \to \infty$, with $\tau_0 \coloneqq 1/(2\log(1 + e/2)) \approx 0.583$.

For \Qref{3}, the comparison passes to sums of independent variables, and to sums $\sum_i X_i x_i$ with vector coefficients $x_i$, with an explicit coupling (Proposition~\ref{prop:sums}). With the two-level concentration of the Laplace product, it also bounds the upper tail of convex Lipschitz functions of independent elements, whatever their number (Corollary~\ref{cor:concentration}). For weighted sums of centered elements, which need be neither symmetric nor bounded, it gives a comparison with Laplace averages in the spirit of \citet[(4.8)]{ledoux1991probability}, without symmetrization. Comparing each summand with a Laplace variable gives a sum of independent Laplace variables, which is not a Laplace variable, and bringing it back to a single Laplace variable loses a constant factor (Proposition~\ref{prop:averages}). It is better to add the parameters of the summands first, by Lemma~\ref{lem:sumclass}, and to compare the sum once.

For random vectors, the regular simplex of \citet[Theorem 2]{zhang2026sharpmgf}, scaled up as far as the class allows, and the product $\varepsilon\varepsilon'^\top$ of two independent Rademacher vectors, viewed as a square matrix, give lower bounds on the least multiple of the rotation-invariant Laplace vector, the second of order $d^{1/4}$ in dimension $d$ (Proposition~\ref{prop:dim}). Unlike the Gaussian vector in the theorem of \citet{vanhandel2025subgaussian}, the rotation-invariant Laplace vector therefore gives no comparison with a constant that does not depend on the dimension, although its least multiple is at most $d$ times the one-dimensional constant (Theorem~\ref{thm:bench}(i)). The lower bound holds for every benchmark normalized by its mean absolute projection, because, on norms and maxima of linear forms, every benchmark averaged over rotations acts as a multiple of the Laplace vector (Theorem~\ref{thm:bench}).

\paragraph{Idea of the proof.} We follow the proof of \citet{zhang2026sharpmgf} for the sub-Gaussian bound, and we read it through martingale couplings \eqref{eq:strassen}. The comparison with $c\Lap$ reduces to one number for each probability $p \in (0,1/2]$: the largest spread of a centered two-point law of the class whose upper atom has probability $p$, the spread being the distance between its two atoms.

Up to an affine function, a convex function is a mixture of the hinges $x \mapsto (x - t)_+$, so we compare $X$ with $c\Lap$ threshold by threshold \eqref{eq:slcrit}: for every element $X$ of the class and every $t \ge 0$,
\begin{align*}
    \SL_X(t) \le \SL_{c\Lap}(t) = \frac{c}{2}\,e^{-t/c} ,
\end{align*}
and the thresholds $t < 0$ follow by symmetry, since the class contains $-X$. Fix $t$ and let $p_t \coloneqq \Prob(X > t)$, which we call the level. The hinge at $t$ is affine on each side of $t$, so that
\begin{align*}
    \SL_X(t) = p_t\bigl(\E[X \mid X > t] - t\bigr)
\end{align*}
depends on $X$ only through the level and the mean of $X$ above $t$. We replace $X$ by the centered two-point variable $Y \coloneqq \E[X \mid \one(X > t)]$, which we call the compression of $X$ at $t$. It has the same stop-loss transform as $X$ at $t$, it is smaller than $X$ in the convex order, since $(Y, X)$ is a martingale coupling \eqref{eq:strassen}, and it stays in the class. The compression changes with $t$, but it is always a centered two-point law $Y_{p,\nu}$ of the class, with level $p = p_t$, spread $\nu$, upper atom $(1-p)\nu$ and lower atom $-p\nu$. So we only need to compare these laws with $c\Lap$. The same averaging applied to $c\Lap$ gives, for $p \le 1/2$,
\begin{align*}
    Y_{p,\nu} \lecx c\Lap
    \quad\Longleftrightarrow\quad
    (1-p)\,\nu \le c\,\eL(p),
    \qquad
    \eL(p) \coloneqq 1 + \log\frac{1}{2p},
\end{align*}
where $\eL(p)$ is the tail mean of $\Lap$ at level $p$, that is, the mean of $\Lap$ over its upper tail of probability $p$. With a larger upper atom, the stop-loss transform of $Y_{p,\nu}$ exceeds that of $c\Lap$ at the threshold $c\log(1/(2p))$ (Figure~\ref{fig:stoploss}). The case $p_t > 1/2$ follows by symmetry, applying the same argument to $-X$ at the threshold $-t$. Hence $c\Lap$ dominates the class as soon as $(1-p)\,\nu \le c\,\eL(p)$ for every two-point element of the class of level $p \le 1/2$ (Proposition~\ref{prop:reduction}).

What remains is the largest spread at each level, and this is where the class matters. A two-point law belongs to the class when its cumulant generating function stays below the envelope of the class, and the inequality of \citet{kearns1998large} describes this function: for a Bernoulli variable $\xi$ with parameter $p$,
\begin{align*}
    \E\left[e^{\lambda(\xi - p)}\right] \le e^{\kappa_p\lambda^2}
    \qquad \text{for every } \lambda \in \R,
    \qquad
    \kappa_p = \frac{1 - 2p}{4\beta_p}, \quad \beta_p \coloneqq \log\frac{1-p}{p} \quad (p \ne 1/2),
\end{align*}
where $\kappa_p$ is the least such coefficient (Section~\ref{sec:spread}). For $p < 1/2$, the ratio $\log\E[e^{\lambda(\xi - p)}]/\lambda^2$ is largest at the single frequency $2\beta_p$, where the exponential tilt exchanges the probabilities of the two atoms. On the whole line, one frequency therefore decides, as in \citet{zhang2026sharpmgf} (Remark~\ref{rem:zhang}), and this is the first branch. On a bounded window, the frequency $2\beta_p/\nu$ may fall outside the window. The ratio is unimodal \citep{schlemm2016kearns}, so the endpoint of the window then decides, and this gives a second branch (Proposition~\ref{prop:numax}). In the sub-exponential class with $\alpha > \sigma$, the largest upper atoms stay below the tail mean of $\alpha\Lap$ at every level and come close to it at small levels, so the constant is $\alpha$. In the sub-Gamma class, where the envelope diverges at the endpoint of the window, they exceed it at small levels, and the constant is strictly larger than $\alpha$ (Section~\ref{sec:SG}).

\paragraph{Tools and earlier work.} \citet{davispower2026sharp} determine sharp one-dimensional comparisons under the sub-Gaussian and the exponential tail bounds \citep[Theorems 1 and 16]{davispower2026sharp}, with a least Laplace multiple of approximately $1.894$ for the second \citep[Section 5]{davispower2026sharp} and dimension-free consequences for convex ridge functions \citep[Theorem 22]{davispower2026sharp}, and \citet{zhang2026sharpmgf} the sharp Gaussian comparison under the sub-Gaussian bound on the whole line. The exponential tail bound has one parameter and only the exponential regime. The classes of Definitions~\ref{def:SG} and~\ref{def:SE} are those of Bernstein's inequality, with two parameters and both regimes, for which the variance factor and the scale parameter produce distinct extremal mechanisms. Their least multiples are $\sigma$ times functions of $\alpha/\sigma$, which our results describe over its whole range, with extremal laws that are two-point laws, reached or approached. Compared with \citet{zhang2026sharpmgf}, the new ingredients are the second branch of Proposition~\ref{prop:numax}, created by the bounded window, and two features of the larger envelope: the positive third cumulant that it allows, and its divergence at the endpoints of the window.

The proof combines four classical tools, adapted to the classes of Bernstein's inequality.

The convex order extends the majorization of \citet{hardy1929some}, to which it reduces for uniform laws on $n$ points, and \citet{rothschild1970increasing} studied it as a notion of increasing risk. We use its stop-loss criterion \eqref{eq:slcrit} \citep[Section 3.A]{shaked2007stochastic}. Strassen's martingale coupling \citep{strassen1965existence} is our second tool. Its easy direction is Jensen's inequality, which we use all along: to average an element of the class and to compare a two-point law with the benchmark. For sums, Section~\ref{sec:sums} uses the existence of the coupling, to add independent summands. Each comparison then comes with an explicit coupling.

Two-point laws are the usual extremal laws of the classical inequalities. \citet{hoeffding1955extrema} reduced extremal problems under moment constraints to laws with few atoms (Remark~\ref{rem:history}), and \citet{hoeffding1963probability} compared sums of bounded variables with sums of two-point variables. \citet{eaton1970note, eaton1974probability} and \citet{pinelis1998optimal} developed these comparisons for moments and tail probabilities, and \citet{bentkus2004hoeffding} proved them for every convex function (Lemma 4.3) and, up to an absolute constant, for the tail probabilities of martingales with bounded differences (Theorems 1.1 to 1.3). Under an upper bound on the variables and a bound on their variance, the best bound of the Chernoff method comes from an asymmetric two-point law whose upper atom is rare when the variance is small \citep[Theorem 3 and (2.10)]{hoeffding1963probability}, which is the regime of the inequality of \citet{bennett1962probability} \citep[proof of Theorem 2.9]{boucheron2013concentration}. In these results, the two-point laws dominate the variable. We use them the other way around. The compression at a threshold is dominated by the element of the class and has the same stop-loss transform there, so it is enough to dominate each compression at its own threshold. We compare it with the benchmark by a tangent line, the first step of the Brownian embeddings of \citet{chacon1976one} and \citet{dubins1968theorem} (Remark~\ref{rem:brownian}).

The inequality of Kearns and Saul gives the least Gaussian coefficient of a centered Bernoulli variable. \citet{kearns1998large} stated it, with a sketch of proof, for large-deviation bounds in approximate probabilistic inference, so it comes from the learning literature. \citet{berend2013concentration} proved it in full in a study of the concentration of the missing mass, \citet{buldygin2013subgaussian} gave another proof, and \citet{schlemm2016kearns} proved the unimodality behind it. In tail form, the coefficient $1/(4\kappa_p)$ is the coefficient of $nt^2$ in the exponent of the Gaussian tail bound (2.2) in \citet[Theorem 1]{hoeffding1963probability}, for the deviation $t$ of an average of $n$ independent variables with values in $[0,1]$ and mean $p < 1/2$. We combine these tools to obtain the sharp comparison with the Laplace law for the two classes of Bernstein's inequality.

\paragraph{Outline of the paper.} Section~\ref{sec:method} reduces the question to two-point laws and computes the largest spreads in the sub-exponential class. Section~\ref{sec:results} gives the two main theorems, Section~\ref{sec:sums} treats sums of independent variables, Section~\ref{sec:vectors} random vectors, and Section~\ref{sec:discussion} discusses what the results show and the questions that remain. The proofs are in the appendices.

\paragraph{Notation.} For $x, y \in \R$, $x \vee y$ and $x \wedge y$ are their maximum and minimum, $x_+ \coloneqq x \vee 0$, and $\one(A)$ is the indicator of an event $A$. We write $X' \sim X$ when $X'$ has the law of $X$. The standard Gaussian variable $G$ has density $\varphi$ and distribution function $\Phi$, and $\varepsilon$ denotes a Rademacher variable, uniform on $\{-1,1\}$.
Numbers are rounded to three decimals.

\section{Reduction to two-point laws through martingale couplings}
\label{sec:method}

The convex order tests every convex function, and each class contains infinitely many laws. In this section, we bring the comparison with $c\Lap$ down to one real number for each probability $p$, in three reductions that lose nothing. By Lemma~\ref{lem:classes}(iii), $\cE(\sigma,\alpha) = \sigma\,\cE(1,a)$ and $\cG(\sigma,\alpha) = \sigma\,\cG(1,a)$ with $a \coloneqq \alpha/\sigma$, so we may take $\sigma = 1$, and $\Cl$ denotes $\SE(1,a)$ or $\SG(1,a)$. The following statements describe the three reductions.

\begin{roadbox}
Let $c > 0$. Let $Y_{p,\nu}$ be the centered two-point variable whose upper atom has probability $p \in (0,1)$, called its level, and whose atoms are at distance $\nu > 0$, called its spread \eqref{eq:Y}. For $p \le 1/2$, let $\eL(p)$ be the tail mean of $\Lap$ at level $p$ \eqref{eq:laplace}. The following four statements are equivalent.
\begin{enumerate}
\item[(a)] $X \lecx c\Lap$ for every $X \in \Cl$.
\item[(b)] $\SL_X(t) \le \SL_{c\Lap}(t)$ for every $X \in \Cl$ and every threshold $t \in \R$.
\item[(c)] $(1-p)\,\nu \le c\,\eL(p)$ for every $p \in (0,1/2]$ and $\nu > 0$ such that $Y_{p,\nu} \in \Cl$.
\item[(d)] $c \ge \sup_{0 < p \le 1/2}\, (1-p)\,\nu_{\Cl}(p,a)/\eL(p)$, where $\nu_{\Cl}(p,a) \coloneqq \sup\{\nu > 0 \colon Y_{p,\nu} \in \Cl\}$ is the largest spread of a two-point element of $\Cl$ of level $p$, written $\nuE$ and $\nuG$ in \eqref{eq:numax}.
\end{enumerate}
\end{roadbox}

In words, $c\Lap$ dominates the class exactly when no two-point element of the class, of level $p \le 1/2$, has its upper atom above $c\,\eL(p)$, the tail mean of $c\Lap$ at level $p$.

Statement (b) is (a) tested on hinges, and the converse holds because, up to an affine function, a convex function is an integral of hinges, one hinge for each threshold $t$ (Section~\ref{sec:hinges}). Statement (c) is (b) for two-point laws, and a tangent line turns the comparison with $c\Lap$ into a comparison of upper atoms. The converse holds because an average on each side of the threshold replaces every element of $\Cl$ by a two-point law (Section~\ref{sec:compression}). Statement (d) restates (c) through the largest spread, which Section~\ref{sec:spread} computes, and this is where the class matters. By symmetry, it is enough to consider two-point laws whose upper atom has probability at most $1/2$, and, for such a law, to check the bound on its moment generating function at positive frequencies only (Lemma~\ref{lem:sign} in Appendix~\ref{app:method}).

\subsection{Stop-loss transforms}
\label{sec:hinges}

The convex order asks for $\E[f(X)] \le \E[f(c\Lap)]$ for every convex function $f$. It is enough to test the hinges $x \mapsto (x-t)_+$, one for each threshold $t$. The stop-loss transform of an integrable variable $X$ is
\begin{align*}
    \SL_X(t) \coloneqq \E\left[(X - t)_+\right],
    \qquad t \in \R .
\end{align*}
Up to an affine function, every convex function is a mixture of the functions $x \mapsto (x - t)_+$ over $t > 0$ and $x \mapsto (t - x)_+$ over $t \le 0$, as in the proof of Theorem~\ref{thm:SE}(iii). This is why two integrable variables with the same mean are ordered in the convex order exactly when their stop-loss transforms are ordered at every threshold \citep[Theorem 3.A.1]{shaked2007stochastic}:
\begin{align}
    X \lecx Z
    \quad \Longleftrightarrow \quad
    \E[X] = \E[Z] \ \text{ and } \ \SL_X(t) \le \SL_Z(t) \ \text{ for every } t \in \R .
    \label{eq:slcrit}
\end{align}
Every $X \in \Cl$ is centered (Lemma~\ref{lem:classes}(i)), and so is $c\Lap$, so that \eqref{eq:slcrit} gives (a) $\Leftrightarrow$ (b): the convex order is the hinge-by-hinge order. A hinge is the simplest convex function, and it is affine on each side of its threshold (Section~\ref{sec:compression}).

For the benchmark, the exponential tail of $\Lap$ gives closed forms. For $c > 0$ and $t \ge 0$, and for $p \in (0,1/2]$, with $z_p$ the quantile of order $1 - p$ of $\Lap$,
\begin{align}
    \SL_{c\Lap}(t) = \frac{c}{2}\,e^{-t/c},
    \qquad
    z_p = \log\frac{1}{2p},
    \qquad
    \eL(p) \coloneqq \E\left[\Lap \mid \Lap > z_p\right] = 1 + \log\frac{1}{2p} .
    \label{eq:laplace}
\end{align}
The last quantity is the tail mean of $\Lap$ at level $p$, and the identity $\E[\Lap - z_p \mid \Lap > z_p] = 1$ is the memoryless property of the exponential tail. For $t < 0$, symmetry gives $\SL_{c\Lap}(t) = -t + (c/2)\,e^{t/c}$. The right side of (b) is therefore explicit, and what remains is to control the left side for every element of the class, which Section~\ref{sec:compression} does through two-point laws.

\subsection{Two-point laws}
\label{sec:compression}

Fix an element $X$ of the class and a threshold $t$ with $p_t \coloneqq \Prob(X > t) > 0$. The hinge vanishes on $\{X \le t\}$ and equals $X - t$ on $\{X > t\}$, so that
\begin{align*}
    \SL_X(t) = p_t\left(\E[X \mid X > t] - t\right) .
\end{align*}
The stop-loss transform at $t$ depends on $X$ only through the probability $p_t$ of exceeding $t$ and the mean of $X$ above $t$, and not on how $X$ is spread on each side of $t$. We may therefore replace $X$ on each side of $t$ by its conditional mean, which leaves a law with two atoms. The centered two-point laws form a family with two parameters.

For $p \in (0,1)$ and $\nu > 0$, let $Y_{p,\nu}$ denote the centered two-point variable with
\begin{align}
    \Prob\left(Y_{p,\nu} = (1-p)\nu\right) = p,
    \qquad
    \Prob\left(Y_{p,\nu} = -p\nu\right) = 1 - p .
    \label{eq:Y}
\end{align}
We call $p$ its level, the probability of its upper atom, and $\nu$ its spread, the distance between its atoms. Taking $t$ in each of the three intervals cut by the atoms gives its stop-loss transform:
\begin{align}
    \SL_{Y_{p,\nu}}(t) = \max\bigl\{ 0,\ -t,\ p\left((1-p)\nu - t\right) \bigr\},
    \qquad t \in \R .
    \label{eq:twopointsl}
\end{align}
The variable $-Y_{p,\nu}$ has the law of $Y_{1-p,\nu}$, and $\Lap$ is symmetric, so it is enough to compare two-point laws of level at most $1/2$ with $c\Lap$.

An average on each side of a threshold brings an element of the class down to a two-point law.

\begin{lem}[{after \citealp[Lemma 6]{zhang2026sharpmgf}}]
\label{lem:compression}
Let $X$ be integrable and centered, and let $t \in \R$ with $p_t \coloneqq \Prob(X > t) \in (0,1)$. Then the variable $Y \coloneqq \E[X \mid \one(X > t)]$ has the two-point law \eqref{eq:Y} with level $p_t$ and some spread $\nu > 0$, and
\begin{align*}
    \E[f(Y)] \le \E[f(X)] \ \text{ for every convex } f,
    \qquad
    \SL_Y(t) = \SL_X(t) .
\end{align*}
\end{lem}

We call $Y$ the compression of $X$ at $t$. Lemma~\ref{lem:compression} shows that $Y$ keeps the stop-loss transform of $X$ at $t$, because $x \mapsto (x - t)_+$ is affine on each side of $t$, and that $Y$ is smaller than $X$ in the convex order, by the conditional Jensen inequality \citep[Theorem 3.A.20]{shaked2007stochastic}: since $Y$ takes two distinct values, according to whether $X > t$, we have $\E[X \mid Y] = Y$, so that $(Y, X)$ is a martingale coupling \eqref{eq:strassen}, and this is the easy direction of Strassen's theorem. The level of $Y$ is the probability $p_t$ that $X$ exceeds the threshold, so it changes with the threshold, and we need the comparison with $c\Lap$ below at every level.

Both classes are stable under the average of Lemma~\ref{lem:compression}: a bound on $\E[e^{\lambda X}]$ for every frequency $\lambda$ of the window defines each class, and an average lowers these expectations. Lemma~\ref{lem:classes} states this in (ii), and gives in (i) and (iii) two other properties of the classes used above.

\begin{lem}
\label{lem:classes}
Let $\sigma, \alpha > 0$, and let $\Cl$ be $\SE(\sigma^2,\alpha)$ or $\SG(\sigma^2,\alpha)$.
\begin{enumerate}
\item[(i)] Every $X \in \Cl$ is integrable and centered.
\item[(ii)] If $X \in \Cl$, then $-X \in \Cl$, $rX \in \Cl$ for every $r \in [0,1]$, and $\E[X \mid \mathcal{A}] \in \Cl$ for every $\sigma$-algebra $\mathcal{A}$.
\item[(iii)] $X$ belongs to $\Cl$ if and only if $X/\sigma$ belongs to the class with parameters $1$ and $a \coloneqq \alpha/\sigma$.
\end{enumerate}
\end{lem}

Together, Lemmas~\ref{lem:compression} and~\ref{lem:classes} show that, if $c\Lap$ fails to dominate an element of the class at a threshold $t$, then the compression of this element is a two-point element of the class that fails at the same threshold. It is therefore enough to look at the two-point elements of the class, and what the class constrains is how far they can spread. For $p \in (0,1)$ and $a > 0$, the largest spreads of the two-point elements at level $p$ are
\begin{align}
    \nuE(p,a) \coloneqq \sup\bigl\{\nu > 0 \colon Y_{p,\nu} \in \SE(1,a)\bigr\},
    \qquad
    \nuG(p,a) \coloneqq \sup\bigl\{\nu > 0 \colon Y_{p,\nu} \in \SG(1,a)\bigr\} .
    \label{eq:numax}
\end{align}

It remains to decide when $c\Lap$ dominates a two-point law $Y_{p,\nu}$ of level $p \le 1/2$. By \eqref{eq:slcrit}, this is a comparison of stop-loss transforms, and it has a simple answer: the upper atom of $Y_{p,\nu}$ must not exceed the tail mean of $c\Lap$ at level $p$.

\begin{lem}
\label{lem:criterion}
Let $p \in (0,1/2]$, $\nu > 0$ and $c > 0$. Then
\begin{align*}
    Y_{p,\nu} \lecx c\Lap
    \quad \Longleftrightarrow \quad
    (1 - p)\,\nu \le c\,\eL(p) .
\end{align*}
In the case of equality, $Y_{p,\nu}$ has the law of $\E[c\Lap \mid \one(\Lap > z_p)]$.
\end{lem}

Lemma~\ref{lem:criterion} turns the comparison with $c\Lap$ into a comparison of upper atoms. It follows the tangent-line argument of \citet[Lemma 6]{davispower2026sharp} and of \citet{zhang2026sharpmgf}, which we carry out at every level for two-point laws. Its proof, in Appendix~\ref{app:method}, gives the martingale coupling \eqref{eq:strassen} in closed form. When the condition holds, let $\tau \coloneqq (1-p)\nu/(c\,\eL(p)) \in (0,1]$, let $\tilde\Lap$ be a copy of $\Lap$, and let $\eta$ be a Bernoulli variable with parameter $\tau$, the variables $\Lap$, $\tilde\Lap$ and $\eta$ being independent. Then
\begin{equation}
\begin{aligned}
    Z' &\coloneqq \eta\,c\Lap + (1 - \eta)\,c\tilde\Lap \sim c\Lap,
    \\
    X' &\coloneqq \E\left[Z' \mid \one\left(\Lap > z_p\right)\right] = \tau\,\E\left[c\Lap \mid \one\left(\Lap > z_p\right)\right] \sim Y_{p,\nu},
\end{aligned}
\label{eq:coupling}
\end{equation}
and $\E[Z' \mid X'] = X'$, since $X'$ takes two distinct values according to whether $\Lap > z_p$. Replacing $c\Lap$ by an independent copy with probability $1 - \tau$ before averaging multiplies both atoms by $\tau$, because the copy averages to zero on each side. This is a coupling form of the fact that a contraction of a centered variable is dominated by the variable itself \citep[Theorem 3.A.18]{shaked2007stochastic}. Conversely, the comparison of stop-loss transforms at $c z_p$ rules out any larger upper atom.

Figure~\ref{fig:stoploss} describes the geometry. Of the three lines in \eqref{eq:twopointsl}, the lines $0$ and $-t$ lie below the stop-loss transform of every centered variable. The two-point law is therefore dominated exactly when its third line, of slope $-p$, stays below the stop-loss transform of $c\Lap$. Since $\SL_{c\Lap}$ is convex, the extreme position of this line is the tangent to $\SL_{c\Lap}$ of slope $-p$, which meets it at $c z_p$, where $\SL_{c\Lap}'(c z_p) = -\Prob(c\Lap > c z_p) = -p$.

\begin{figure}[!ht]
\centering
\includegraphics[width=267.1pt]{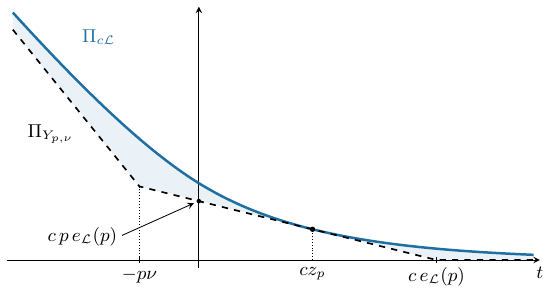}
\caption{Stop-loss transforms of $c\Lap$ (solid) and of the extremal two-point law of level $p$ (dashed), with $c = 1$ and $p = 1/5$. The dashed line of slope $-p$ runs from the lower atom $-p\nu$ to the upper atom $c\,\eL(p)$ and meets the vertical axis at $c\,p\,\eL(p)$; the shaded gap between the two transforms closes only at $c z_p$, where this line is tangent to the solid curve.}
\label{fig:stoploss}
\end{figure}

The extreme case is the compression of $c\Lap$ itself: averaging $c\Lap$ on each side of $c z_p$ gives the two-point law of level $p$ whose stop-loss transform meets that of $c\Lap$ at $c z_p$. Every two-point law of the same level with a smaller upper atom is a contraction of this law, and the coupling \eqref{eq:coupling} shows that $c\Lap$ dominates it as well.

\begin{remarkbox}
\begin{remark}[Earlier two-point reductions]
\label{rem:history}
Replacing a variable by a law with few atoms is a classical way to solve extremal problems. Under finitely many moment constraints, laws with finitely many atoms suffice \citep{hoeffding1955extrema}. The best bound on a stop-loss transform in terms of the mean and the variance has an extremal law with two atoms \citep{bowers1969upper}. The classes $\SE(\sigma^2,\alpha)$ and $\SG(\sigma^2,\alpha)$ have one constraint for each frequency, which is more than these results allow. We use instead that every constraint bounds the expectation of a convex function. Averaging $X$ on each side of a threshold keeps the stop-loss transform at the threshold and lowers the expectation of every convex function, so the average stays in the class. The average and $X$ form a martingale coupling, as in the easy direction of Strassen's theorem \citep{strassen1965existence}.
\end{remark}
\end{remarkbox}

We now put the two lemmas together. Let $X$ be an element of the class and $t$ a threshold with $p_t \coloneqq \Prob(X > t) \in (0, 1)$; the cases $p_t \in \{0, 1\}$ are treated in the proof of Proposition~\ref{prop:reduction}. Let $Y$ be the compression of $X$ at $t$. By Lemma~\ref{lem:compression}, $Y$ has the law of a two-point variable $Y_{p_t,\nu}$ with $\SL_Y(t) = \SL_X(t)$, and $Y \lecx X$. By Lemma~\ref{lem:classes}(ii), $Y$ belongs to the class. Suppose first that $p_t \le 1/2$. If $(1-p_t)\,\nu \le c\,\eL(p_t)$, then $Y \lecx c\Lap$ by Lemma~\ref{lem:criterion}, so that
\begin{align*}
    \SL_X(t) = \SL_Y(t) \le \SL_{c\Lap}(t) .
\end{align*}
If $(1-p_t)\,\nu > c\,\eL(p_t)$, the stop-loss transform of $Y$ exceeds that of $c\Lap$ at $c z_{p_t}$ (Lemma~\ref{lem:criterion} and Figure~\ref{fig:stoploss}), and the same holds for $X$, since $Y \lecx X$. In words, the upper atom $c\,\eL(p_t)$ is the largest that $c\Lap$ allows at level $p_t$: among the events of probability $p_t$, the upper tail $\{\Lap > z_{p_t}\}$ is the one on which $c\Lap$ has the largest mean. The case $p_t > 1/2$ follows by symmetry, applying the same argument to $-X$ at the threshold $-t$. A compression depends on the threshold, so we repeat the argument at every threshold. By \eqref{eq:slcrit}, $c\Lap$ therefore dominates the class as soon as $(1-p)\,\nu \le c\,\eL(p)$ for every two-point element $Y_{p,\nu}$ of the class with $p \le 1/2$. This turns the least multiple into a supremum over the level, as the next proposition states.

\begin{prop}
\label{prop:reduction}
Let $a > 0$. The least multiples of $\Lap$ that dominate $\SE(1,a)$ and $\SG(1,a)$ are
\begin{align}
    \cE(1,a) = \sup_{0 < p \le 1/2} \frac{(1-p)\,\nuE(p,a)}{\eL(p)},
    \qquad
    \cG(1,a) = \sup_{0 < p \le 1/2} \frac{(1-p)\,\nuG(p,a)}{\eL(p)} .
    \label{eq:reduction}
\end{align}
When they are finite, $\cE(1,a)\Lap$ dominates every element of $\SE(1,a)$, and $\cG(1,a)\Lap$ every element of $\SG(1,a)$.
\end{prop}

Proposition~\ref{prop:reduction} turns the comparison into a supremum over the level $p$. It makes explicit the reduction in the proof of \citet{zhang2026sharpmgf}, and its proof, in Appendix~\ref{app:method}, uses only Lemmas~\ref{lem:compression}, \ref{lem:classes} and~\ref{lem:criterion} and \eqref{eq:slcrit}. What remains for each class is the largest spread at a given level, a problem on one two-point law and its moment generating function.

\begin{remarkbox}
\begin{remark}[Two-point laws as exit laws of Brownian motion]
\label{rem:brownian}
The two-point law $Y_{p,\nu}$ is the law of a Brownian motion started at $0$ and stopped when it leaves the interval $[-p\nu, (1-p)\nu]$: the spread $\nu$ is the length of the interval, and the level $p$ is the probability of leaving it through the upper end. For a centered $X$, the potential function of \citet{chacon1976one} is $-\E\left[|X - t|\right] = -2\SL_X(t) - t$, so the convex order is the order of potentials. Strassen's theorem has a Brownian form \citep[Proposition 8.1]{obloj2004skorokhod}: $X \lecx Z$ if and only if some stopping time of a Brownian motion started with the law of $X$ gives the law of $Z$ and keeps the stopped motion uniformly integrable. The compression of Lemma~\ref{lem:compression} is the first step of an embedding of $X$, and the tangent line of Lemma~\ref{lem:criterion} is a step of an embedding of $c\Lap$. The law $\sigma\varepsilon$, which attains the constant of the sub-exponential class when $\alpha \le \sigma$, exits the symmetric interval $[-\sigma, \sigma]$, whereas every extremal two-point law of the sub-Gamma class exits an asymmetric interval (Section~\ref{sec:SG}).
\end{remark}
\end{remarkbox}

\subsection{The Kearns--Saul inequality}
\label{sec:spread}

The class appears only through the largest spread $\nu_{\Cl}(p,a)$, and we compute it in two parts: first for a bound on the whole line, then for a bounded window. A two-point law belongs to the class when its moment generating function stays below the bound of the class, and this function is explicit. For the sub-Gaussian bound $\E[e^{\lambda Y}] \le e^{\sigma^2\lambda^2/2}$, the question for a centered two-point law $Y$ is how large its cumulant generating function can be compared with $\lambda^2$. \citet{kearns1998large} answered this question for centered Bernoulli variables. The classes studied here replace the Gaussian envelope by the envelope of the class on a window, and the shape of the ratio $g_p(s)/s^2$, defined below, decides the largest spread at each level.

For $p \in (0,1)$, the cumulant generating function of $Y_{p,1}$ is
\begin{align}
    g_p(s) \coloneqq \log\left(p\,e^{(1-p)s} + (1-p)\,e^{-ps}\right),
    \qquad\text{so that}\qquad
    \log\E\left[e^{\lambda Y_{p,\nu}}\right] = g_p(\lambda\nu) ;
    \label{eq:gp}
\end{align}
it is strictly convex, with $g_p(0) = g_p'(0) = 0$, so it increases strictly from $0$ to $\infty$ on $[0,\infty)$, with inverse $g_p^{-1}$ there. The exponential tilt of $Y_{p,1}$ at $s$, the law whose density with respect to that of $Y_{p,1}$ is proportional to $e^{sy}$, gives to the upper atom the probability
{\air%
\begin{align*}
    q_s \coloneqq \frac{p\,e^s}{1 - p + p\,e^s} ,
\end{align*}}%
and $g_p'(s) = q_s - p$ is the mean of this tilted law. Let
{\air%
\begin{align*}
    \beta_p \coloneqq \log\frac{1-p}{p},
    \qquad
    \kappa_p \coloneqq \frac{1 - 2p}{4\beta_p} \ \ (p \ne 1/2),
    \qquad
    \kappa_{1/2} \coloneqq \frac{1}{8} .
\end{align*}}%
For a Bernoulli variable $\xi$ with parameter $p$, the inequality of \citet[Lemma 1]{kearns1998large} states that $\E[e^{\lambda(\xi - p)}] \le e^{\kappa_p\lambda^2}$ for every $\lambda \in \R$, and $\kappa_p$ is the least such coefficient \citep{berend2013concentration, buldygin2013subgaussian}. Since $\xi - p$ has the law of $Y_{p,1}$, the inequality reads $g_p(s) \le \kappa_p s^2$ for every $s \in \R$. For $p < 1/2$, the ratio $g_p(s)/s^2$ increases on $(0, 2\beta_p]$ and decreases on $[2\beta_p, \infty)$ \citep[Theorem 1.1]{schlemm2016kearns}, and negative values of $s$ give smaller values (Lemma~\ref{lem:sign}). We call $2\beta_p$ the Kearns--Saul frequency of $Y_{p,1}$, and $2\beta_p/\nu$ the one of $Y_{p,\nu}$. For $p < 1/2$, Lemma~\ref{lem:unimodal} gives
\begin{align}
    q_{2\beta_p} = 1 - p,
    \qquad
    \max_{s > 0}\frac{g_p(s)}{s^2} = \frac{g_p(2\beta_p)}{(2\beta_p)^2} = \kappa_p :
    \label{eq:ksfreq}
\end{align}
at the Kearns--Saul frequency, the tilt exchanges the probabilities of the two atoms, and the ratio is largest. The exchange makes this frequency critical: the derivative of $g_p(s)/s^2$ vanishes where $s\,g_p'(s) = 2g_p(s)$. At $s = 2\beta_p$, the tilted mean is $g_p'(2\beta_p) = q_{2\beta_p} - p = 1 - 2p$ and $g_p(2\beta_p) = (1-2p)\beta_p$, so both sides equal $2(1-2p)\beta_p$.

\begin{figure}[t]
\centering
\includegraphics[width=313.1pt]{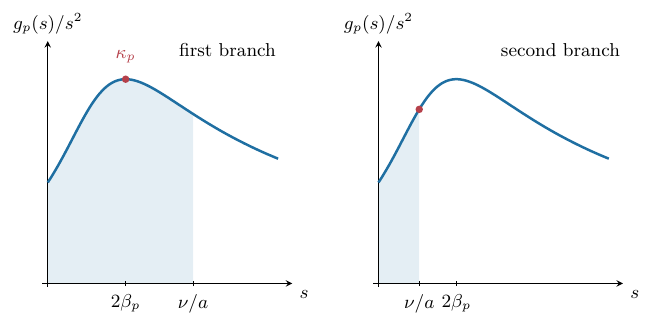}
\caption{The ratio $g_p(s)/s^2$ for $p = 0.1$, where $g_p$ is the cumulant generating function of $Y_{p,1}$. It increases up to the Kearns--Saul frequency $2\beta_p$, where it equals $\kappa_p$, and decreases afterwards. The shaded part is the window $|s| \le \nu/a$ of a two-point law of spread $\nu$, and the dot is the largest value of the ratio over the window. Left: the window contains $2\beta_p$, and the largest value is $\kappa_p$ (first branch of Proposition~\ref{prop:numax}). Right: the window ends before $2\beta_p$, and the largest value is reached at its endpoint (second branch).}
\label{fig:ksratio}
\end{figure}

\begin{remarkbox}
\begin{remark}[The sub-Gaussian case]
\label{rem:zhang}
For the class of the variables with $\E[e^{\lambda X}] \le e^{\lambda^2/2}$ for every $\lambda \in \R$, the proof of \citet[Theorem 1]{zhang2026sharpmgf} follows the same lines, with the benchmark $G$. It uses the compression \citep[Lemma 6]{zhang2026sharpmgf}, the tangent line of Lemma~\ref{lem:criterion} \citep[inequality (15)]{zhang2026sharpmgf}, and the largest spread on the whole line \citep[inequality (10)]{zhang2026sharpmgf}, which is the first branch of Proposition~\ref{prop:numax} below. To find this spread, the proof tests the moment generating function of the two-point law at its Kearns--Saul frequency $2\beta/\nu$, which loses nothing: by \eqref{eq:ksfreq}, the ratio $g_p(s)/s^2$ is largest there, so the bound holds at every frequency once it holds there. The proof then bounds the ratio of \eqref{eq:reduction} for this class and benchmark by $\sqrt{\pi/2}$ at every level, with equality at $p = 1/2$, through a lower bound on the Gaussian isoperimetric function \citep[Lemma 3, inequality (7)]{zhang2026sharpmgf}. Lemma~\ref{lem:elementary}(a) is the counterpart of this bound for the benchmark $\Lap$. A bounded window, as in the classes studied here, adds the second branch of Proposition~\ref{prop:numax}.
\end{remark}
\end{remarkbox}

The unimodality of $g_p(s)/s^2$ also gives the least coefficient on a bounded window.

\begin{prop}[Finite-window Kearns--Saul inequality]
\label{prop:ks}
Let $p \in (0,1/2]$ and $a > 0$. The least coefficient $A_p(a)$ such that $\E[e^{\lambda(\xi - p)}] \le e^{A_p(a)\lambda^2}$ for every $|\lambda| \le 1/a$ is
\begin{align*}
    A_p(a) = \kappa_p \ \text{ if } 2\beta_p \le \frac{1}{a},
    \qquad
    A_p(a) = a^2 g_p\Bigl(\frac1a\Bigr) \ \text{ if } 2\beta_p > \frac{1}{a} .
\end{align*}
For $p < 1/2$, equality holds at a single nonzero frequency of the window: $2\beta_p$ in the first case, $1/a$ in the second case.
\end{prop}

Proposition~\ref{prop:ks} says that the Kearns--Saul frequency decides when the window contains it, and the endpoint of the window decides otherwise. It is the simplest instance of the second branch of Proposition~\ref{prop:numax} below.

By \eqref{eq:gp}, the bound \eqref{eq:SE} for $Y_{p,\nu}$ at the frequency $\lambda$ is a bound on $g_p$ at $s = \lambda\nu$, and the window $|\lambda| \le 1/a$ becomes $|s| \le \nu/a$, which moves with $\nu$. The largest spread is therefore not obtained from Proposition~\ref{prop:ks} at a fixed window, but the same unimodality gives it, on two branches.

\begin{prop}
\label{prop:numax}
Let $p \in (0,1/2)$ and $a > 0$, and let $\bar\lambda(p) \coloneqq \sqrt{2\beta_p(1 - 2p)}$.
\begin{enumerate}
\item[(i)] If $\bar\lambda(p) \le 1/a$ (first branch), then $\nuE(p,a) = \sqrt{2\beta_p/(1 - 2p)}$, and the bound \eqref{eq:SE} for $Y_{p,\nuE(p,a)}$ is an equality at the single nonzero frequency $\bar\lambda(p)$.
\item[(ii)] If $\bar\lambda(p) > 1/a$ (second branch), then $\nuE(p,a) = a\,m$ with $m \coloneqq g_p^{-1}(1/(2a^2))$, which lies in $(0, 2\beta_p)$, and the bound is an equality at the single nonzero frequency $1/a$.
\end{enumerate}
Moreover, $\nuE(1/2, a) = 2$.
\end{prop}

Proposition~\ref{prop:numax} shows that the largest spread has two branches. On the first, it is the one allowed on the whole line, as in the sub-Gaussian case, and $\bar\lambda(p) = 2\beta_p/\nuE(p,a)$ is the Kearns--Saul frequency of $Y_{p,\nuE(p,a)}$. On the second, the endpoint of the window fixes it. Since $\bar\lambda(p)$ decreases from $\infty$ to $0$ as $p$ increases from $0$ to $1/2$, the second branch consists of the small levels, and there the largest spread grows like $a\log(1/p)$ (Lemma~\ref{lem:endpoint}(i)), while on the whole line it grows only like $\sqrt{2\log(1/p)}$. In Brownian terms (Remark~\ref{rem:brownian}), the largest spread is the largest length of an exit interval whose exit law satisfies the bound of the class: on the first branch, a single frequency of the whole line decides, and on the second, the endpoint of the window. The proofs are in Appendix~\ref{app:method}. Propositions~\ref{prop:reduction} and~\ref{prop:numax} together reduce the constant of the sub-exponential class to a supremum over the level of an explicit ratio, which Section~\ref{sec:SE} evaluates.

\section{Main results}
\label{sec:results}

Section~\ref{sec:SE} gives the sharp constant for the sub-exponential class, and Section~\ref{sec:SG} the one for the sub-Gamma class. All the constants come from the same reduction to two-point laws: each is a supremum over the level $p$ of a ratio for two-point laws, and the class appears only through the largest spread that it allows at each level. These largest spreads are given by Proposition~\ref{prop:numax} for the sub-exponential class, and by Theorem~\ref{thm:SG}(i) for the sub-Gamma class.

\subsection{The sub-exponential class}
\label{sec:SE}

The next theorem gives the least multiple of $\Lap$ for the sub-exponential class.

\begin{mainbox}
\begin{thm}
\label{thm:SE}
Let $\sigma, \alpha > 0$.
\begin{enumerate}
\item[(i)] Every $X \in \SE(\sigma^2,\alpha)$ satisfies $X \lecx (\sigma \vee \alpha)\,\Lap$, and no smaller multiple of $\Lap$ dominates every element of the class, that is,
\begin{align*}
    \cE(\sigma,\alpha) = \sigma \vee \alpha .
\end{align*}
\item[(ii)] If $\alpha \le \sigma$, the constant is attained by the variable $\sigma\varepsilon$, which belongs to $\SE(\sigma^2,\alpha)$, and the convex function $f(x) = |x|$: $\E[f(\sigma\varepsilon)] = \sigma = \E[f(\sigma\Lap)]$.
\item[(iii)] If $\alpha > \sigma$, then $\E[f(X)] < \E[f(\alpha\Lap)]$ for every $X \in \SE(\sigma^2,\alpha)$ and every convex function $f$ that is not affine and satisfies $\E[f(\alpha\Lap)] < \infty$.
\end{enumerate}
\end{thm}
\end{mainbox}

Theorem~\ref{thm:SE} shows that the two parameters act through two different mechanisms. The parameter $\sigma$ is forced by the symmetric two-point law $\sigma\varepsilon$ of (ii), as in the sub-Gaussian case. Multiples of the benchmark do not force it: $b\Lap$ lies in the class only if $b < \sigma/\sqrt2$. Indeed, with $u = \lambda^2$, the function $u \mapsto -\log(1 - b^2u) - \sigma^2u/2$ is convex and vanishes at $0$, so the bound \eqref{eq:SE} for $b\Lap$ holds on the window if and only if it holds at the endpoint $u = 1/\alpha^2$:
\begin{align*}
    b\Lap \in \SE(\sigma^2,\alpha)
    \quad\Longleftrightarrow\quad
    b^2 \le \alpha^2\bigl(1 - e^{-\sigma^2/(2\alpha^2)}\bigr),
\end{align*}
and the right side is less than $\sigma^2/2$, since $1 - e^{-y} < y$ for $y > 0$, so $b < \sigma/\sqrt2$.

The parameter $\alpha$ is forced by the window, through rare upper atoms. At the endpoint $1/\alpha$ of the window, the moment generating function of a two-point element with an upper atom $x$ of probability $p$ is at least $p\,e^{x/\alpha}$, and it must stay below $e^{\sigma^2/(2\alpha^2)}$, so that
\begin{align*}
    x \le \alpha\log\frac{1}{p} + \frac{\sigma^2}{2\alpha},
\end{align*}
while the tail mean of $\alpha\Lap$ at level $p$ is $\alpha(\log(1/p) + 1 - \log 2)$. The ratio of the two tends to one as $p \to 0$, and the largest upper atoms of Proposition~\ref{prop:numax} have the same behavior (Lemma~\ref{lem:endpoint}(i)). When $\alpha > \sigma$, the ratio of Proposition~\ref{prop:reduction}, with $a = \alpha/\sigma$, stays below $a$ at every level (Lemmas~\ref{lem:elementary}(a) and~\ref{lem:endpoint}(ii)). No two-point element of the class reaches the constant, and (iii) shows that no element attains it. In the sub-exponential class, the window therefore does not raise the constant above the elementary lower bound $\sigma \vee \alpha$, and it only decides whether this constant is attained.

The proof of Theorem~\ref{thm:SE} has three parts: the upper bound, the lower bound and the strict comparison. For the upper bound, we bound the ratio of Proposition~\ref{prop:reduction} level by level: by Lemma~\ref{lem:elementary}(a) on the first branch of Proposition~\ref{prop:numax}, where the whole line decides, and by Lemma~\ref{lem:endpoint} on the second branch, where the window decides. The second branch is the most delicate part. The lower bound is the pair of elementary arguments of the introduction; the bound $c \ge \alpha$ also follows from the limit of the ratio as $p \to 0$ (Lemma~\ref{lem:endpoint}(i)). The strict comparison for $\alpha > \sigma$ uses the reduction once more: a strict inequality between stop-loss transforms at every threshold, integrated over the hinges that make up a convex function (Section~\ref{sec:hinges}), gives a strict inequality between expectations.

\subsection{The sub-Gamma class}
\label{sec:SG}

The next theorem gives the least multiple for the sub-Gamma class.

\begin{mainbox}
\begin{thm}
\label{thm:SG}
Let $\sigma, \alpha > 0$ and $a \coloneqq \alpha/\sigma$.
\begin{enumerate}
\item[(i)] Every $X \in \SG(\sigma^2,\alpha)$ satisfies $X \lecx \cG(\sigma,\alpha)\,\Lap$, and
\begin{align*}
    \cG(\sigma,\alpha) = \sigma \sup_{0 < p \le 1/2} \frac{(1-p)\,\nuG(p,a)}{\eL(p)},
    \qquad
    \nuG(p,a) = \inf_{s > 0}\, \frac{s}{2}\biggl(a + \sqrt{a^2 + \frac{2}{g_p(s)}}\,\biggr) .
\end{align*}
\item[(ii)] The constant is attained by a two-point law: there exist a level $p_a \in (0,1/2)$, a spread $\nu > 0$ with $Y_{p_a,\nu} \in \SG(\sigma^2,\alpha)$, and a threshold $t \ge 0$ such that $\SL_{Y_{p_a,\nu}}(t) = \SL_{\cG(\sigma,\alpha)\Lap}(t)$.
\item[(iii)] The constant satisfies
\begin{align*}
    \sigma \vee \alpha < \cG(\sigma,\alpha) \le \frac{\alpha + \sqrt{\alpha^2 + 4\sigma^2}}{2} \le (\sigma + \alpha) \wedge \Bigl(\alpha + \frac{\sigma^2}{\alpha}\Bigr),
\end{align*}
and $\cG(\sigma,\alpha)/(\sigma \vee \alpha) \to 1$ as $a \to 0$ and as $a \to \infty$.
\end{enumerate}
\end{thm}
\end{mainbox}

Theorem~\ref{thm:SG} shows that the larger envelope always has a cost: the least multiple is strictly above $\sigma \vee \alpha$, and it is attained by an asymmetric two-point law. Figure~\ref{fig:scalar} shows this cost over the whole range of $a$: it vanishes at both ends, and it is largest at $a = 1$. Indeed, the class $\SG(\sigma^2,\alpha)$ increases with each parameter, so that $\cG(1,a)$ is nondecreasing in $a$, which settles $a \le 1$, and $\cG(1,a)/a = \cG(1/a, 1)$ is nonincreasing in $a$, which settles $a \ge 1$.

In (i), the largest spread is an infimum over the frequency (Lemma~\ref{lem:fixedp} in Appendix~\ref{app:SG}). The variable $s = \lambda\nu$ in (i) is the frequency in units of the spread $\nu$. The bound \eqref{eq:SG}, with $\sigma = 1$, for $Y_{p,\nu}$ at the frequency $s/\nu$ is a quadratic inequality in $\nu$, and the quantity inside the infimum is its positive root, the largest spread allowed at this frequency. The infimum over $s$ takes every frequency into account. Each $s$ gives an upper bound on the largest spread: the limit $s \to 0$ gives the variance bound $p(1-p)\,\nu^2 \le 1$, and the Kearns--Saul frequency $s = 2\beta_p$ gives
\begin{align*}
    \nuG(p,a) \le a\beta_p + \sqrt{a^2\beta_p^2 + \frac{2\beta_p}{1 - 2p}},
    \qquad 0 < p < \frac12,
\end{align*}
the largest spread $\sqrt{2\beta_p/(1-2p)}$ on the whole line, corrected by the scale parameter. Table~\ref{tab:SG} gives the constant for a few ratios $a$, computed from (i) by minimizing over $s$ and maximizing over $\log(1/p)$, a parametrization that reaches the very small levels involved. The level $p_a$ of the extremal law tends to $1/2$ as $a \to 0$, outside the range of the table, and to $0$ as $a \to \infty$ (Proposition~\ref{prop:ends} and its proof); in the table, it decreases from $10^{-3}$ to $10^{-113}$. At each ratio of the table, the extremal law thus has a rare upper atom, as in the inequality of \citet{bennett1962probability}.

An explicit upper bound comes from the inclusions \eqref{eq:inclusions}, with the half window replaced by the best window. For $b > \alpha$, the sub-Gamma envelope is at most $\sigma^2\lambda^2/(2(1 - \alpha/b))$ on the window $|\lambda| \le 1/b$, so that
\begin{align}
    \SG(\sigma^2,\alpha) \subseteq \SE\Bigl(\frac{\sigma^2 b}{b - \alpha},\, b\Bigr),
    \qquad b > \alpha,
    \label{eq:window}
\end{align}
which is the second inclusion of \eqref{eq:inclusions} when $b = 2\alpha$, and Theorem~\ref{thm:SE} bounds the least multiple by $\sigma\sqrt{b/(b - \alpha)} \vee b$. The first term decreases in $b$ and the second increases; they are equal when $b^2 = \alpha b + \sigma^2$, which gives the bound of Theorem~\ref{thm:SG}(iii). The first upper bound in (iii) is the best that the inclusion \eqref{eq:window} and Theorem~\ref{thm:SE} give: by the compound Poisson variable of \eqref{eq:poisson}, an inclusion of $\SG(\sigma^2,\alpha)$ in a sub-exponential class $\SE(s^2, b)$ needs $b > \alpha$ and $s^2 \ge \sigma^2 b/(b - \alpha)$. The other two bounds follow from it and are simpler.

\begin{table}[ht]
\centering
\renewcommand{\arraystretch}{1.25}
\caption{The least multiple $\cG(1,a)$, the upper bound $(a + \sqrt{a^2 + 4})/2$ of Theorem~\ref{thm:SG}(iii), the quantity $(a + \sqrt{a^2 + 4\tau_0})/2$ of Proposition~\ref{prop:ends}(iii), which is not an upper bound, the ratio $\cG(1,a)/a$, and the level $p_a$ at which the supremum of Theorem~\ref{thm:SG}(i) is reached numerically. The values are rounded to three decimal places, and to three significant figures for $p_a$; two independent implementations give the same values up to $10^{-6}$.}
\label{tab:SG}
\begin{tabular}{lcccc}
\toprule
$a$ & $1$ & $2$ & $5$ & $10$ \\
\midrule
$\cG(1,a)$ & $1.415$ & $2.258$ & $5.114$ & $10.058$ \\
$(a + \sqrt{a^2 + 4})/2$ & $1.618$ & $2.414$ & $5.193$ & $10.099$ \\
$(a + \sqrt{a^2 + 4\tau_0})/2$ & $1.412$ & $2.258$ & $5.114$ & $10.058$ \\
$\cG(1,a)/a$ & $1.415$ & $1.129$ & $1.023$ & $1.006$ \\
$p_a$ & $2.56\cdot 10^{-3}$ & $6.67\cdot 10^{-7}$ & $2.76\cdot 10^{-30}$ & $1.35\cdot 10^{-113}$ \\
\bottomrule
\end{tabular}
\end{table}

\begin{figure}[t]
\centering
\includegraphics[width=337.5pt]{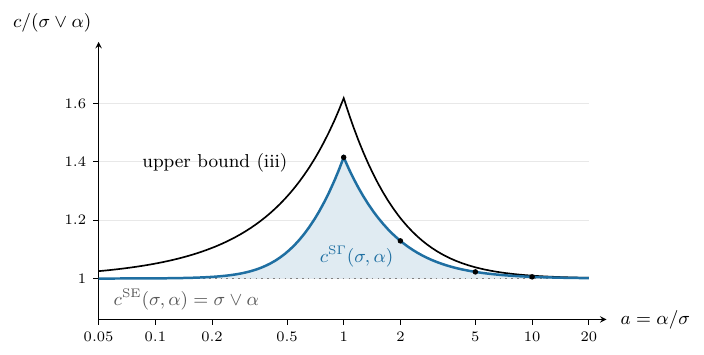}
\caption{The least multiple $\cG(\sigma,\alpha)$ of Theorem~\ref{thm:SG} and its upper bound in Theorem~\ref{thm:SG}(iii), divided by $\sigma \vee \alpha = \cE(\sigma,\alpha)$, as functions of $a = \alpha/\sigma$, with the values of Table~\ref{tab:SG} as dots. The shaded gap is the cost of the larger envelope: it vanishes as $a \to 0$, where the class decreases to the sub-Gaussian class, and as $a \to \infty$, where the exponential regime of Bernstein's tail bound covers every deviation, and it is largest at $a = 1$.}
\label{fig:scalar}
\end{figure}

Theorems~\ref{thm:SE} and~\ref{thm:SG} answer \Qref{2}, and their proofs show where the sub-Gamma constant separates from the sub-exponential one. For an element $X$ of variance $\sigma^2$, the expansions at $\lambda = 0$
\begin{align*}
    \log\E\left[e^{\lambda X}\right] = \frac{\sigma^2\lambda^2}{2} + \frac{\E[X^3]}{6}\,\lambda^3 + O(\lambda^4),
    \qquad
    \frac{\sigma^2\lambda^2}{2(1 - \alpha|\lambda|)} = \frac{\sigma^2\lambda^2}{2} + \frac{\sigma^2\alpha}{2}\,|\lambda|^3 + O(\lambda^4),
\end{align*}
taken for both signs of $\lambda$, show that the Gaussian envelope forces $\E[X^3] = 0$, while the sub-Gamma envelope allows $\E[X^3] > 0$. With $\sigma = 1$, the sub-Gamma envelope permits mildly asymmetric two-point laws of variance one whose ratio in Proposition~\ref{prop:reduction} is larger than one (proof of Theorem~\ref{thm:SG}). By scaling, this gives $\cG(\sigma,\alpha) > \sigma$. The sub-Gamma envelope also diverges at the endpoints of the window, which leaves room for larger upper atoms and gives $\cG(\sigma,\alpha) > \alpha$. At small levels $p$, with $\ell \coloneqq \log(1/p)$, the largest upper atoms compare with the tail mean of $\alpha\Lap$ as follows, for $\sigma = \alpha$ (Lemma~\ref{lem:endpoint}(i) for the first term, and the proof of Theorem~\ref{thm:SG} for the third):
\begin{align}
    \underbrace{\alpha\left(\ell + \log\left(e^{1/2} - 1\right) + o(1)\right)}_{\text{sub-exponential}}
    \quad < \quad
    \underbrace{\alpha\left(\ell + 1 - \log 2\right)}_{\text{tail mean of } \alpha\Lap}
    \quad < \quad
    \underbrace{\alpha\left(\ell + 1\right)\vphantom{\big(}}_{\text{sub-Gamma, at least}}\, ,
    \label{eq:threeatoms}
\end{align}
where $\log(e^{1/2} - 1) \approx -0.433$ and $1 - \log 2 \approx 0.307$. By \eqref{eq:threeatoms}, at small levels, the sub-exponential atoms stay below the tail mean of $\alpha\Lap$, while the sub-Gamma atoms exceed it by at least $\alpha\log 2$ and push the constant above $\alpha$.

In the sub-Gamma class, the constant is attained. The upper limits of the ratio of Proposition~\ref{prop:reduction} as $p \to 1/2$ and as $p \to 0$ are at most $1$ and at most $a$, while the ratio exceeds $1 \vee a$ at some level; since the ratio is upper semicontinuous (Lemma~\ref{lem:fixedp}), its supremum is a maximum. The proof of Theorem~\ref{thm:SG} has four parts. The formula and the upper bounds come from Proposition~\ref{prop:reduction}, the formula for the largest spread (Lemma~\ref{lem:fixedp}) and the window inclusion \eqref{eq:window}. The constant exceeds $\sigma$ because the sub-Gamma envelope allows a two-point law of variance $\sigma^2$ with a positive third cumulant, and it exceeds $\alpha$ because the envelope allows a rare upper atom beyond the tail mean of $\alpha\Lap$ (see \eqref{eq:threeatoms}). The fourth part, that the constant is attained, is the argument above, based on the upper semicontinuous ratio. The proof is in Appendix~\ref{app:SG}.

The next proposition gives the cost of the larger envelope at both ends of the range of $a$, and a bound at every level that is sharp for large $a$.

\begin{prop}
\label{prop:ends}
Let $\tau_0 \coloneqq 1/(2\log(1 + e/2)) \approx 0.583$.
\begin{enumerate}
\item[(i)] As $a \to 0$,
\begin{align*}
    \cG(1,a) = 1 + \frac{9}{16}\,a^3 + O(a^4) .
\end{align*}
\item[(ii)] As $a \to \infty$,
\begin{align*}
    a\left(\cG(1,a) - a\right) \longrightarrow \tau_0 .
\end{align*}
\item[(iii)] For every $a > 0$ and every $p \in (0,1/2]$, the ratio of Proposition~\ref{prop:reduction} satisfies
\begin{align*}
    \frac{(1-p)\,\nuG(p,a)}{\eL(p)} \le \frac{a + \sqrt{a^2 + 4\tau_0}}{2}\cdot\frac{1 + 2p/(e\,\eL(p))}{1 - 2p/e} .
\end{align*}
\end{enumerate}
\end{prop}

Proposition~\ref{prop:ends} shows that, by scaling, the cost $\cG(\sigma,\alpha) - \sigma \vee \alpha$ is $(9/16)\,\alpha^3/\sigma^2$ to first order when $\alpha/\sigma$ is small, of third order in the scale parameter, and $\tau_0\,\sigma^2/\alpha$ to first order when $\alpha/\sigma$ is large. The first cost comes from the third cumulant: the lower bound uses the two-point laws of variance one of the proof of Theorem~\ref{thm:SG}, and the upper bound uses Theorem~\ref{thm:SG}(i) at a frequency $s$ of order $a$. The second comes from the divergence of the envelope at the endpoint of the window.

Part (iii) is Theorem~\ref{thm:SG}(i) at the frequency $s$ where $g_p(s) = \log(1 + e/2)$. Its first factor is the positive root of $b^2 = ab + \tau_0$, the equation of Theorem~\ref{thm:SG}(iii) with $\tau_0$ in place of $1$, and its second factor tends to one as $p \to 0$. The matching lower bound uses two-point laws of level $e^{-L_\star a^2}$, with $L_\star \approx 2.557$ the constant defined in Appendix~\ref{app:SG}. The first factor of (iii) alone is not an upper bound on $\cG(1,a)$, since it is smaller at $a = 1$; numerically, the two differ by less than $2\cdot 10^{-6}$ for $a = 2$, $5$ and $10$ (Table~\ref{tab:SG}). The proof, in Appendix~\ref{app:SG}, also locates the extremal levels: $1 - 2p_a = O(a)$ as $a \to 0$, and $\log(1/p_a) > (1 - \log 2)^2a^2/2$ for every $a > 0$.

For instance, Theorem~\ref{thm:SG} applies to every centered Gamma variable: with shape $k$ and scale $\gamma$, it belongs to $\SG(k\gamma^2, \gamma)$ and is therefore dominated by $\cG(\gamma\sqrt{k}, \gamma)\,\Lap$. A centered chi-squared variable with $k$ degrees of freedom belongs to $\SG(2k, 2)$, and Theorem~\ref{thm:SG} gives the multiple $\cG(\sqrt{2k}, 2)$, which lies between $\sqrt{2k} \vee 2$ and $1 + \sqrt{1 + 2k} \le \sqrt{2k} + 2$.

\section{Sums of independent variables}
\label{sec:sums}

Bernstein's inequality is a statement about sums, and the comparison passes to them. This section transfers the comparison to sums, deduces a concentration bound for convex Lipschitz functions, and measures what comparing summand by summand loses. Both classes are stable under independent sums, with the variance factors added and the largest scale kept.

\begin{lem}
\label{lem:sumclass}
Let $X_1, \dots, X_n$ be independent with $X_i \in \SG(\sigma_i^2, \alpha_i)$, and let $w_1, \dots, w_n$ be real numbers, not all zero. Then the weighted sum satisfies
\begin{align*}
    \sum_{i=1}^n w_i X_i \in \SG\Bigl(\sum_{i=1}^n w_i^2\sigma_i^2,\ \max_{i} |w_i|\alpha_i\Bigr),
\end{align*}
and the same holds for the sub-exponential classes.
\end{lem}

Lemma~\ref{lem:sumclass} and Theorems~\ref{thm:SE} and~\ref{thm:SG} give, for averages of $n$ independent elements of a class,
{\air%
\begin{align}
    \frac{1}{n}\sum_{i=1}^n X_i \lecx \Bigl(\frac{\sigma}{\sqrt n} \vee \frac{\alpha}{n}\Bigr)\Lap
    \quad\text{in } \SE(\sigma^2,\alpha),
    \qquad
    \frac{1}{n}\sum_{i=1}^n X_i \lecx \cG\Bigl(\frac{\sigma}{\sqrt n}, \frac{\alpha}{n}\Bigr)\Lap
    \quad\text{in } \SG(\sigma^2,\alpha) .
    \label{eq:avg}
\end{align}}%
The window term $\alpha/n$ is the larger for small $n$, and the variance term $\sigma/\sqrt n$ takes over beyond $n = (\alpha/\sigma)^2$, which, up to the factor $t/2$, is where the two terms of Bernstein's bound \eqref{eq:HB} are equal, with $\sigma^2 = v$ and $\alpha = b/3$. These constants are sharp for the class given by Lemma~\ref{lem:sumclass}, but not for averages themselves when $n$ is large, as the next proposition shows.

\begin{prop}[Normalized sums]
\label{prop:clt}
Let $c_n(\sigma,\alpha)$ be the least $c$ such that $n^{-1/2}\sum_{i=1}^n X_i \lecx c\Lap$ for all independent $X_1, \dots, X_n$ in $\SG(\sigma^2,\alpha)$. Then $c_n(\sigma,\alpha) \le \cG(\sigma, \alpha/\sqrt n)$ and
\begin{align*}
    c_n(\sigma,\alpha) \longrightarrow \sigma\sqrt{\frac{2}{\pi}} \approx 0.798\,\sigma
    \qquad (n \to \infty) .
\end{align*}
The same limit holds in the sub-exponential class.
\end{prop}

Proposition~\ref{prop:clt} shows that, for normalized sums, the least multiple tends to the Gaussian constant $\sigma\sqrt{2/\pi}$. The bound $c_n(\sigma,\alpha) \le \cG(\sigma, \alpha/\sqrt n)$ holds because the normalized sum belongs to $\SG(\sigma^2, \alpha/\sqrt n)$ by Lemma~\ref{lem:sumclass}, and its right side tends to $\sigma$, so the limit is smaller. The element $\sigma\varepsilon$ and the function $|x|$ give the lower bound on the limit, and the upper bound follows from the central limit theorem in the $L^1$ distance between distribution functions \citep[Theorem 1.1]{goldstein2010bounds}, with third moments bounded on the class by Theorem~\ref{thm:SG}, and from $\sqrt{\pi/2}\,G \lecx \Lap$ (Proposition~\ref{prop:averages} below). In the limit, the comparison with $\Lap$ is thus the Gaussian one. The sub-exponential class is smaller and contains $\sigma\varepsilon$, so the same holds there. The proof is in Appendix~\ref{app:sums}.

The comparison also passes to sums with vector coefficients $\sum_i X_i x_i$ in $\R^d$. There the convex order uses convex functions $f \colon \R^d \to \R$, the stop-loss criterion \eqref{eq:slcrit} no longer characterizes it, and we use the coupling of Strassen's theorem. If independent real variables satisfy $X_i \lecx Z_i$, Strassen's theorem gives pairs $(X_i', Z_i')$ with $\E[Z_i' \mid X_i'] = X_i'$; taken independent, they give $\E[Z' \mid X'] = X'$ for the vectors $X'$ and $Z'$ of $\R^n$, and hence $X' \lecx Z'$ \citep[Theorems 7.A.1 and 7.A.8]{shaked2007stochastic}. A linear map keeps this identity, and the conditional Jensen inequality compares every convex function of the images. The next proposition therefore uses only the one-dimensional comparisons of Theorems~\ref{thm:SE} and~\ref{thm:SG}, and no comparison with a Gaussian vector.

\begin{prop}
\label{prop:sums}
Let $X_1, \dots, X_n$ be independent real variables with $X_i \lecx c_i\Lap$ for every $i$, let $\Lap_1, \dots, \Lap_n$ be independent copies of $\Lap$, and let $x_1, \dots, x_n \in \R^d$. Then, in the convex order of $\R^d$,
\begin{align*}
    \sum_{i=1}^n X_i\,x_i \lecx \sum_{i=1}^n c_i\,\Lap_i\,x_i ,
\end{align*}
and there is a coupling of the two sides in which the left side is the conditional expectation of the right side given $X_1, \dots, X_n$.
\end{prop}

Proposition~\ref{prop:sums} gives in particular, for every norm $\|\cdot\|$ on $\R^d$ and every convex nondecreasing $F \colon [0,\infty) \to \R$,
\begin{align}
    \E\Bigl[F\Bigl(\Bigl\|\sum_{i=1}^n X_i\,x_i\Bigr\|\Bigr)\Bigr]
    \le
    \E\Bigl[F\Bigl(\Bigl\|\sum_{i=1}^n c_i\sqrt{2W_i}\,G_i\,x_i\Bigr\|\Bigr)\Bigr],
    \label{eq:banach}
\end{align}
with independent standard exponential $W_i$ and standard Gaussian $G_i$, by the identity in law \eqref{eq:mixture}. Given the $W_i$, the right side is a Gaussian average with weights $c_i\sqrt{2W_i}$, to which the tools for Gaussian processes apply. For Rademacher variables, Theorem~\ref{thm:SE}(i) with $\sigma = \alpha = 1$ gives $c_i = 1$, and \eqref{eq:banach} is already known. It follows from the comparison of Rademacher and Gaussian averages of \citet[(4.8)]{ledoux1991probability}, which gives the right side with $\sqrt{\pi/2}$ in place of $\sqrt{2W_i}$, and from Jensen's inequality in the $W_i$, since $\E[\sqrt{2W_i}] = \sqrt{\pi/2}$ and, given the $W_i$, the expectation on the right side is a convex function of $(\sqrt{2W_1}, \dots, \sqrt{2W_n})$. The interest of \eqref{eq:banach} lies in the elements of the classes that are neither symmetric nor bounded. For centered variables, \eqref{eq:banach} needs no symmetrization step, whose factor $2$ appears in \citet[Lemma 6.3]{ledoux1991probability}.

Bernstein's bound \eqref{eq:bernstein} has two regimes, and so does the concentration of the Laplace vector. For $n$ independent standard Laplace coordinates, a function $h$ that is $\varrho_2$-Lipschitz for the Euclidean norm and $\varrho_1$-Lipschitz for the $\ell^1$ norm satisfies, by \citet[Section 2]{bobkov1997poincare},
\begin{align*}
    \Prob\bigl(h \ge \E[h] + t\bigr) \le \exp\Bigl(-\min\Bigl\{\frac{t^2}{16\varrho_2^2},\ \frac{t}{4\varrho_1}\Bigr\}\Bigr),
    \qquad t \ge 0,
\end{align*}
which is Gaussian for small $t$ and exponential for large $t$, whatever $n$. The comparison transfers this bound to the classes, for convex functions.

\begin{cor}
\label{cor:concentration}
Let $X = (X_1, \dots, X_n)$ have independent coordinates with $X_i \lecx c\Lap$ for every $i$, as holds with $c = \sigma \vee \alpha$ for $X_i \in \SE(\sigma^2,\alpha)$, and with $c = \cG(\sigma,\alpha)$ for $X_i \in \SG(\sigma^2,\alpha)$. Let $\Lap^{(n)}$ have $n$ independent standard Laplace coordinates. For every convex $f \colon \R^n \to \R$ that is $\varrho_2$-Lipschitz for the Euclidean norm and $\varrho_1$-Lipschitz for the $\ell^1$ norm, every $0 \le \lambda \le 1/(2c\varrho_1)$ and every $t \ge 0$,
\begin{align}
    \E\bigl[e^{\lambda f(X)}\bigr] &\le \exp\Bigl(\lambda\,\E\bigl[f(c\Lap^{(n)})\bigr] + 4c^2\varrho_2^2\lambda^2\Bigr),
    \label{eq:mgfconc}
    \\
    \Prob\Bigl(f(X) \ge \E\bigl[f(c\Lap^{(n)})\bigr] + t\Bigr) &\le \exp\Bigl(-\min\Bigl\{\frac{t^2}{16c^2\varrho_2^2},\ \frac{t}{4c\varrho_1}\Bigr\}\Bigr).
    \label{eq:concentration}
\end{align}
\end{cor}

Corollary~\ref{cor:concentration} shows that the same two constants describe the two regimes for every convex Lipschitz function, such as the norms in \eqref{eq:banach}: for independent coordinates, Bernstein's hypothesis keeps its two regimes in every dimension, at the level of the benchmark. Every function that is $\varrho$-Lipschitz for the Euclidean norm is $\varrho$-Lipschitz for the $\ell^1$ norm, so that one may take $\varrho_1 = \varrho_2$, but we keep both constants because they are the two parameters of Bernstein's bound. For a linear function $f(x) = \langle w, x\rangle$, they are $\varrho_2 = \|w\|_2$ and $\varrho_1 = \|w\|_\infty$, and \eqref{eq:concentration} is Gaussian at the scale $c\|w\|_2$ and exponential at the scale $c\|w\|_\infty$. This is the shape of Bernstein's bound, in which $\sigma\|w\|_2$ and $\alpha\max_i|w_i|$ play these roles (Lemma~\ref{lem:sumclass}), with the constant $c \ge \sigma \vee \alpha$ in place of $\sigma$ and $\alpha$.

This does not contradict Section~\ref{sec:vectors}, where the least multiple of the Laplace vector grows with the dimension for the classes defined through one-dimensional projections. A vector with independent coordinates in $\SG(\sigma^2,\alpha)$ belongs to such a class, by Lemma~\ref{lem:sumclass}, but the vectors that force the growth have dependent coordinates (Proposition~\ref{prop:dim}). For independent coordinates, the one-dimensional comparison holds in every dimension (Proposition~\ref{prop:sums}), and the Laplace vector passes its concentration to the class, as the Gaussian vector does for the sub-Gaussian class. If $\E[e^{\lambda X_i}] \le e^{\sigma^2\lambda^2/2}$ for every $\lambda$ and $i$, and $G$ is a standard Gaussian vector of $\R^n$, then, for every convex $f \colon \R^n \to \R$ that is $\varrho_2$-Lipschitz for the Euclidean norm and every $t \ge 0$,
\begin{align*}
    \Prob\Bigl(f(X) \ge \E\bigl[f(\sigma\sqrt{\pi/2}\,G)\bigr] + t\Bigr) \le \exp\Bigl(-\frac{t^2}{\pi\sigma^2\varrho_2^2}\Bigr),
\end{align*}
by the comparison $X_i \lecx \sigma\sqrt{\pi/2}\,G_i$ of \citet[Theorem 1]{zhang2026sharpmgf} and Gaussian concentration. The proof is in Appendix~\ref{app:sums}.

The bound is one-sided, and its center is the benchmark level $\E[f(c\Lap^{(n)})]$, which is at least $\E[f(X)]$. Dimension-free bounds around $\E[f(X)]$ are known for bounded coordinates \citep{talagrand1995concentration}, \citep[Theorem 6.10]{boucheron2013concentration}. Around the median, those for unbounded sub-Gaussian coordinates lose a logarithm of the dimension, which cannot be removed \citep{huang2023dimension}. The corollary differs: it has no such loss, because its center is the benchmark level and not the median, and its two regimes carry the factor $c$, so that for sums Bernstein's inequality is sharper.

For an average of independent elements of $\SG(\sigma^2,\alpha)$, the comparison can be made in two ways. Adding the parameters first gives \eqref{eq:avg}. Comparing each summand with $c\Lap$ first, where $c \coloneqq \cG(\sigma,\alpha)$, gives a comparison with an average of independent Laplace variables, which is not a Laplace variable: its moment generating function is $(1 - c^2\lambda^2/n^2)^{-n}$. The next proposition measures what is lost when this average is brought back to a single Laplace variable.

\begin{prop}
\label{prop:averages}
Let $b > 0$, $n \ge 1$, and let $X_1, \dots, X_n$ be independent with $X_i \lecx b\Lap$ for every $i$. Then their average satisfies
\begin{align*}
    \frac{1}{n}\sum_{i=1}^n X_i \lecx b\,\rho_n\Lap,
    \qquad
    \rho_n \coloneqq \frac{2\,\Gamma(n + 1/2)}{\sqrt{\pi}\,\Gamma(n+1)} = \frac{2}{4^n}\binom{2n}{n},
\end{align*}
and $\rho_n$ cannot be lowered when the $X_i$ are independent copies of $b\Lap$. Moreover, $1/\sqrt n \le \rho_n \le 2/\sqrt{\pi n}$.
\end{prop}

Proposition~\ref{prop:averages} shows that an average of $n$ independent copies of $b\Lap$ needs the multiple $b\rho_n$ of $\Lap$, of order $b/\sqrt n$. The proof, in Appendix~\ref{app:sums}, uses that a symmetric variable with a log-concave density is dominated by the multiple of $\Lap$ with the same first absolute moment, the symmetric form of a comparison with the exponential law of the same mean \citep[Theorem 3.A.55]{shaked2007stochastic}. Applied to $G$, which has a log-concave density and $\E[|G|] = \sqrt{2/\pi}$, it gives $\sqrt{\pi/2}\,G \lecx \Lap$.

Since $\SG(\sigma^2,\alpha)$ grows with $\alpha$ and $\cG$ is homogeneous, adding the parameters first gives
{\air%
\begin{align*}
    \cG\Bigl(\frac{\sigma}{\sqrt n}, \frac{\alpha}{n}\Bigr) \le \cG\Bigl(\frac{\sigma}{\sqrt n}, \frac{\alpha}{\sqrt n}\Bigr) = \frac{c}{\sqrt n} \le c\,\rho_n,
\end{align*}}%
while comparing each summand first gives $c\,\rho_n$, and no smaller multiple follows from $X_i \lecx c\Lap$ alone. For large $n$, the first multiple is about $\sigma/\sqrt n$, by Theorem~\ref{thm:SG}(iii), and the second about $2c/\sqrt{\pi n}$, larger by the factor $2c/(\sigma\sqrt\pi) \ge 2/\sqrt\pi$. Adding the parameters first also keeps the variance factor $\sigma^2/n$ and the scale $\alpha/n$ of the average, whereas the average of the Laplace variables only belongs to $\SG(2c^2/n, c/n)$, with $2c^2 > \sigma^2$ and $c > \alpha$. A comparison made summand by summand keeps the convex order but loses both the constant and these parameters.

\section{Random vectors}
\label{sec:vectors}

The theorem of \citet{vanhandel2025subgaussian} concerns random vectors, and \Qref{3} asks what remains of it under Bernstein's hypothesis. For $d \ge 1$, the classes $\SE_d(\sigma^2,\alpha)$ and $\SG_d(\sigma^2,\alpha)$ consist of the random vectors $X$ in $\R^d$ whose one-dimensional projections all belong to the corresponding class:
\begin{align*}
    \langle v, X \rangle \in \SE(\sigma^2,\alpha),
    \quad\text{respectively}\quad
    \langle v, X \rangle \in \SG(\sigma^2,\alpha),
    \qquad \text{for every unit vector } v \in \R^d .
\end{align*}
The benchmark is the rotation-invariant Laplace vector $\Lap_d \coloneqq \sqrt{2W}\,G_d$, with $G_d$ a standard Gaussian vector in $\R^d$ and $W$ an independent standard exponential variable, the symmetric multivariate Laplace law of \citet[Chapter 5]{kotz2001laplace}. By \eqref{eq:mixture}, each one-dimensional projection $\langle v, \Lap_d\rangle$ with $|v| = 1$ is a standard Laplace variable, and $\E[e^{\langle y, \Lap_d\rangle}] = 1/(1 - |y|^2)$ for $|y| < 1$. Conversely, a random vector whose projections on unit vectors all have the law of $c\Lap$ has the characteristic function of $c\Lap_d$, so $c\Lap_d$ is the only such law on $\R^d$.

Let $\cE_d(\sigma,\alpha)$ and $\cG_d(\sigma,\alpha)$, with values in $(0,\infty]$, be the least multiples of $\Lap_d$ that dominate the two classes in the convex order of $\R^d$, and let $M_n$ be the expected maximum of $n$ independent standard Gaussian variables. The next proposition gives lower bounds on these multiples: they do not decrease with the dimension, and they grow at least like $d^{1/4}$.

\begin{prop}
\label{prop:dim}
Let $\sigma, \alpha > 0$, $a \coloneqq \alpha/\sigma$, and let $c_d$ denote $\cE_d(\sigma,\alpha)$ or $\cG_d(\sigma,\alpha)$, and $\nu_{\Cl}$ the corresponding largest spread $\nuE$ or $\nuG$ of \eqref{eq:numax}.
\begin{enumerate}
\item[(i)] The sequence $c_d$ is nondecreasing in $d$, and $c_1$ is the one-dimensional constant.
\item[(ii)] For every $d \ge 1$,
\begin{align*}
    c_d \ge \sigma \left(\frac{d}{d+1}\right)^{3/2} \frac{\nu_{\Cl}\left(1/(d+1), a\right)}{\sqrt{\pi/2}\; M_{d+1}} .
\end{align*}
\item[(iii)] If $\alpha \le \sigma$, then
\begin{align*}
    \cE_2(\sigma,\alpha) \ge \frac{8}{9}\sqrt{2\log 2}\;\sigma > \cE_1(\sigma,\alpha),
    \qquad
    \frac{8}{9}\sqrt{2\log 2} \approx 1.047 .
\end{align*}
\item[(iv)] Let $b \coloneqq \sigma \wedge \alpha$ for $\SG_d(\sigma^2,\alpha)$ and $b \coloneqq \alpha\,(1 - e^{-\sigma^2/\alpha^2})^{1/2}$ for $\SE_d(\sigma^2,\alpha)$. With $m \coloneqq \lfloor\sqrt d\rfloor$, the least multiple satisfies
\begin{align*}
    c_d \ge \frac{b\,\sqrt m}{\sqrt{2\pi}},
    \qquad\text{so that}\qquad
    \liminf_{d \to \infty} \frac{c_d}{d^{1/4}} \ge \frac{b}{\sqrt{2\pi}} .
\end{align*}
\end{enumerate}
\end{prop}

Proposition~\ref{prop:dim} shows that the dimension has a cost. A lower bound needs one element of the class and one convex function. We follow the construction of \citet[Theorem 2]{zhang2026sharpmgf}: the uniform law on the vertices of a regular simplex with $d+1$ vertices, and the maximum of the linear forms given by the vertices. By an inequality of \citet[Lemma 10]{zhang2026sharpmgf} on exponential sums, among the projections of the uniform law on the vertices, the projection on the direction of a vertex has the largest moment generating function, at every positive frequency. This projection is a two-point law of level $1/(d+1)$, so the simplex can be stretched until this two-point law has the largest spread \eqref{eq:numax}. By \eqref{eq:mixture}, the expected maximum under $\Lap_d$ is $\sqrt{\pi/2}$ times the Gaussian one. The proof is in Appendix~\ref{app:vectors}.

Two things differ in our arguments from \citet{zhang2026sharpmgf}: the benchmark, and the largest spread of the vertex projection. As long as the window contains the Kearns--Saul frequency of the vertex projection, as in the plane for (iii), this spread is the one allowed on the whole line, and the bound in (ii) is the lower bound of \citet[Theorem 2 and Remark 3]{zhang2026sharpmgf} for the Gaussian benchmark, divided by $\E[\sqrt{2W}] = \sqrt{\pi/2}$. In high dimension, the spread grows like $\alpha\log d$ instead, by the second branch of Proposition~\ref{prop:numax} and \eqref{eq:inclusions}, while the expected Gaussian maximum grows like $\sqrt{2\log d}$: the bounds of \citet{zhang2026sharpmgf} stay bounded, as the theorem of \citet{vanhandel2025subgaussian} requires, while ours grow.

Part (iv) uses a second element, which gives a faster growth. Over $N$ unit directions, the expected maximum of the projections of $\Lap_d$ is at most of order $\sqrt{\log N}$, and that of an element of the class is at most of order $\log N$, by the maximal inequality for sub-Gamma variables \citep[Corollary 2.6]{boucheron2013concentration}. A lower bound built on $N$ well-spread directions therefore grows at most like $\sqrt{\log N}$, which is $\sqrt{\log d}$ when $N$ is a power of $d$, as for the simplex. A faster growth needs a number of directions that grows faster than every power of $d$, and an element of the class whose projection on the best of them is of order $\log N$. The sign matrices $\delta\delta'^\top$ give $4^m/2$ distinct directions in dimension $d = m^2$, a number exponential in $\sqrt d$.

View $\R^{m^2}$ as the space of $m \times m$ matrices, and let $\varepsilon$ and $\varepsilon'$ be independent vectors of $\R^m$ with independent Rademacher coordinates. The element is $X = b\,\varepsilon\varepsilon'^\top$, the Kronecker product of $b\varepsilon$ and $\varepsilon'$ written as a matrix, which we call the Kronecker element. The convex function is $F(x) = \max_{\delta,\delta'} \delta^\top x\,\delta'$, the maximum over sign vectors $\delta, \delta' \in \{-1,1\}^m$. Each projection of $X$ is a bilinear form $b\,\varepsilon^\top V\varepsilon'$, whose moment generating function is at most that of $b$ times a product of two independent standard Gaussian variables, so $X$ belongs to the class. The function $F$ collects the $m^2$ entries of $X$ in phase, $F(X) = b\,m^2$. For a Gaussian matrix, no choice of signs does better than $m$ times the operator norm, which is of order $\sqrt m$, so that the benchmark reaches only the order $m^{3/2}$ (Figure~\ref{fig:kron}). The ratio of the two is of order $\sqrt m$, that is, $d^{1/4}$. The bound (iv) holds for the sub-exponential class as well, so the difficulty does not come from the larger sub-Gamma envelope. Its constant $b$ is of order $\alpha$ when $\alpha \ll \sigma$, so that the bound exceeds $\sigma$ only for $d$ of order $(\sigma/\alpha)^4$ and beyond, in agreement with the dimension-free comparison of \citet{vanhandel2025subgaussian} in the sub-Gaussian limit.

\begin{figure}[t]
\centering
\includegraphics[width=291.4pt]{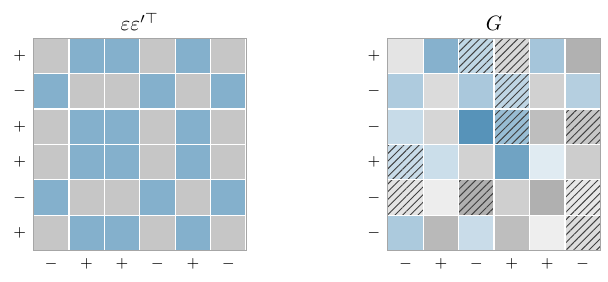}
\caption{Blue entries are positive and gray entries negative. Left, the sign matrix $\varepsilon\varepsilon'^\top$ of Proposition~\ref{prop:dim}(iv), with $\varepsilon$ and $\varepsilon'$ on the margins: the signs $\delta = \varepsilon$ and $\delta' = \varepsilon'$ put every entry $x_{ij}$ in phase, with the sign of $\delta_i\delta'_j$, and $\max_{\delta,\delta'}\delta^\top x\,\delta' = m^2$. Right, a Gaussian matrix with its maximizing signs on the margins: the entries crossed by lines stay out of phase, and the maximum is at most $m$ times the operator norm, of order $m^{3/2}$.}
\label{fig:kron}
\end{figure}

Our lower bounds grow because the rotation-invariant Laplace vector is a Gaussian vector with one random scale, so that, on a maximum of linear forms, it behaves like a Gaussian vector; by Theorem~\ref{thm:bench} below, every benchmark averaged over rotations behaves in the same way. The class, by contrast, contains vectors that, at each outcome, point along one of many linear forms: the simplex puts a rare atom of exponential size on each of its $d+1$ directions (Figure~\ref{fig:spikes}, right), and the Kronecker element points along one of the $4^m/2$ sign matrices, which $F$ sees with all its entries in phase (Figure~\ref{fig:kron}). The simplex gives the better bound in moderate dimension, and the Kronecker element in high dimension: for the sub-exponential class with $\alpha = \sigma$, the second exceeds the first from $d = 576$ (Figure~\ref{fig:rates}).

\begin{figure}[t]
\centering
\includegraphics[width=319.8pt]{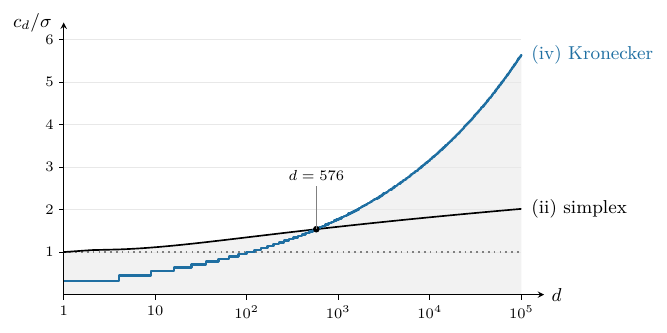}
\caption{The two lower bounds of Proposition~\ref{prop:dim} on $c_d$ for the sub-exponential class with $\alpha = \sigma$, in units of $\sigma$: the simplex bound (ii), which grows like $\sqrt{\log d}$, and the Kronecker bound (iv), a step function of $d$ through $m = \lfloor\sqrt d\rfloor$, which grows like $d^{1/4}$ and exceeds it from $d = 576$. The dotted line is the one-dimensional constant $c_1 = \sigma$, and $c_d$ lies above the shaded region; by Theorem~\ref{thm:bench}, the same bounds hold for every benchmark normalized by its mean absolute projection.}
\label{fig:rates}
\end{figure}

\paragraph{Other benchmarks.} The growth in Proposition~\ref{prop:dim} raises the question whether another benchmark could play, for Bernstein's hypothesis, the role of the Gaussian vector in the theorem of \citet{vanhandel2025subgaussian}. Let $\Cl_d$ be $\SE_d(\sigma^2,\alpha)$ or $\SG_d(\sigma^2,\alpha)$. For an integrable random vector $Z$ in $\R^d$, let
\begin{align}
    c_d(Z) \coloneqq \inf\bigl\{c > 0 \colon X \lecx cZ \ \text{ for every } X \in \Cl_d\bigr\},
    \qquad
    K(Z) \coloneqq \int_{\mathbb{S}^{d-1}} \E\left[|\langle v, Z\rangle|\right] \varsigma(\mathrm{d}v),
    \label{eq:Kbench}
\end{align}
where $\varsigma$ is the uniform probability on the unit sphere $\mathbb{S}^{d-1}$, the law of $G_d/|G_d|$ \citep[Exercise 3.3.7]{vershynin2018high}. By \eqref{eq:Kbench}, $c_d = c_d(\Lap_d)$ and $K(\Lap_d) = \E[|\Lap|] = 1$. We call $K(Z)$ the mean absolute projection of $Z$, and we say that $Z$ is normalized when $K(Z) = 1$, as $\Lap_d$ is. For $s > 0$, $c_d(sZ) = c_d(Z)/s$ and $K(sZ) = s\,K(Z)$, so the product $c_d(Z)\,K(Z)$ does not depend on the scale of $Z$, and the question is whether the constant
\begin{align}
    \bar c_d \coloneqq \inf\bigl\{c_d(Z)\,K(Z) \colon Z \text{ integrable},\ 0 < K(Z) < \infty\bigr\} \;\le\; c_d
    \label{eq:kappa}
\end{align}
stays bounded as $d \to \infty$.

Normalizations that bound the projections from above reduce to this one, since they bound $K(Z)$. If every projection of $Z$ belongs to $\SG(s^2,\beta)$ for some $\beta > 0$, as those of $G_d$ do with $s = 1$, then it has variance at most $s^2$, so $K(Z) \le s$ and $c_d(Z)\,s \ge \bar c_d$.

Two facts show that the lower bounds of Proposition~\ref{prop:dim} hold for $\bar c_d$. First, the classes are invariant under the group $O(d)$ of orthogonal matrices. Let $\bar U$ be a random element of $O(d)$, independent of $Z$, with the Haar probability, the law unchanged by multiplication by a fixed orthogonal matrix \citep[Section 1.2]{meckes2019random}. If $cZ$ dominates the class, then, for every fixed $U \in O(d)$ and every element $X$, it dominates the element $U^\top X$, so $cUZ$ dominates $X$; averaging over $U$, so does $c\bar U Z$. A rotation also leaves the mean absolute projection unchanged, so that
\begin{align*}
    c_d(\bar U Z) \le c_d(Z),
    \qquad
    K(\bar U Z) = K(Z),
\end{align*}
and $\bar U Z$ has the law of $|Z|\,\Theta$, with $\Theta \sim \varsigma$ independent of $|Z|$ (Appendix~\ref{app:vectors}). Second, the lower bounds of Proposition~\ref{prop:dim} test the comparison on maxima of linear forms (Figure~\ref{fig:spikes}, left), and on such functions a benchmark averaged over rotations acts as a multiple of $\Lap_d$. For every continuous, positively homogeneous $F$, such as a norm or a maximum of linear forms, only the mean radius matters:
\begin{align}
    \E\left[F(\bar U Z)\right] = \E[|Z|]\,\E\left[F(\Theta)\right] = K(Z)\,\E\left[F(\Lap_d)\right] .
    \label{eq:homog}
\end{align}
If $cZ$ dominates the class, so does $c\bar U Z$. For every such $F$ that is also convex, and every element $X$ of the class, \eqref{eq:homog} then bounds the expectation of $F(X)$ by a multiple of that of $F(\Lap_d)$:
\begin{align*}
    \E\left[F(X)\right] \le c\,\E\left[F(\bar U Z)\right] = c\,K(Z)\,\E\left[F(\Lap_d)\right] .
\end{align*}
On such functions, $Z$ does no better than the multiple $c\,K(Z)$ of $\Lap_d$. The lower bounds of Proposition~\ref{prop:dim} use such functions, so they hold for $c_d(Z)\,K(Z)$ for every $Z$, hence for $\bar c_d$. The next theorem states this consequence, together with the finiteness of $c_d$.

\begin{mainbox}
\begin{thm}[Finiteness for the Laplace vector and growth for every normalized benchmark]
\label{thm:bench}
Let $\sigma, \alpha > 0$ and $d \ge 1$.
\begin{enumerate}
\item[(i)] The least multiple of $\Lap_d$ is finite: $c_d \le d\,c_1$.
\item[(ii)] The constant $\bar c_d$ of \eqref{eq:kappa} satisfies the bounds (ii) and (iv) of Proposition~\ref{prop:dim}, with $\bar c_d$ in place of $c_d$. In particular, with $b$ as in Proposition~\ref{prop:dim}(iv), $\liminf_{d \to \infty} \bar c_d/d^{1/4} \ge b/\sqrt{2\pi}$.
\end{enumerate}
\end{thm}
\end{mainbox}

\begin{figure}[t]
\centering
\includegraphics[width=340.7pt]{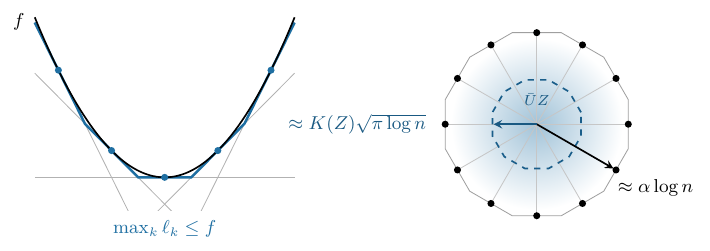}
\caption{Left: a convex function $f$ is the supremum of the affine functions $\ell_k$ below it, so the convex order is tested on maxima of affine functions, as in \citet{vanhandel2025subgaussian} and \citet{song2026short}; the lower bounds use maxima of linear forms. Right: an element of the class may put a rare atom at distance about $\alpha\log n$ in each of $n$ nearly orthogonal directions $v_k$, and $F(x) = \max_k \langle v_k, x\rangle$ takes this value at each atom (outer level set of $F$). A benchmark averaged over $O(d)$, $\bar U Z$ (shaded), need not have the law of $K(Z)\,\Lap_d$, but by \eqref{eq:homog} it has the same expectation of $F$, only about $K(Z)\sqrt{\pi\log n}$ (dashed level set).}
\label{fig:spikes}
\end{figure}

Theorem~\ref{thm:bench} shows that $\bar c_d \to \infty$, so that no normalized benchmark gives, under Bernstein's hypothesis, a comparison that does not depend on the dimension. The obstruction concerns convex functions that do not split along single projections. If $f(x) = \ell(x) + \sum_{k=1}^K \psi_k(\langle u_k, x\rangle)$, with $\ell$ affine, $u_k \in \R^d$ and $\psi_k$ convex, then each $\langle u_k, X\rangle$ is $|u_k|$ times an element of the one-dimensional class, and $\langle u_k, c_1\Lap_d\rangle$ has the law of $|u_k|\,c_1\Lap$, so that $\E[f(X)] \le \E[f(c_1\Lap_d)]$ in every dimension, as in the comparison for convex ridge functions of \citet[Theorem 22 and Remark 23]{davispower2026sharp}. The growth comes from the hypothesis, and not from the choice of $\Lap_d$, which, on norms and maxima of linear forms, does as well as any normalized benchmark.

The least multiple is nevertheless finite in every dimension, and the proof of Theorem~\ref{thm:bench}(i) is short. Let $X$ be in the class, $X_i \coloneqq \langle e_i, X\rangle$, and let $f$ be convex. Since $X$ is the average of the vectors $dX_ie_i$, convexity gives the first inequality below. Each $X_i$ belongs to the one-dimensional class, so $X_i \lecx c_1\Lap$, which gives the second, with $\Lap_{d,i}$ the $i$-th coordinate of $\Lap_d$. The other coordinates of $\Lap_d$ have conditional mean zero given $\Lap_{d,i}$, so $\E[\Lap_d \mid \Lap_{d,i}] = \Lap_{d,i}e_i$, and the conditional Jensen inequality gives the third:
\begin{align*}
    \E[f(X)] \le \frac1d\sum_{i=1}^d \E\bigl[f(dX_ie_i)\bigr] \le \frac1d\sum_{i=1}^d \E\bigl[f(dc_1\Lap_{d,i}e_i)\bigr] \le \E\bigl[f(dc_1\Lap_d)\bigr] .
\end{align*}
The least multiple $c_d$ is therefore finite, at least of order $d^{1/4}$ and at most $d\,c_1$.

\section{Discussion}
\label{sec:discussion}

After the theorem of \citet{vanhandel2025subgaussian}, the least multiple of $G$ that dominates sub-Gaussian variables was found under a tail bound \citep{davispower2026sharp} and under a bound on the moment generating function \citep{zhang2026sharpmgf}, and the least multiple of a Laplace variable under the exponential tail bound \citep[Theorem 16]{davispower2026sharp}. In dimension one, we have answered the same question for the two classes behind Bernstein's inequality, which, unlike the exponential tail bound, have both a Gaussian and an exponential regime: a multiple of the Laplace law dominates each class, and we have found the least one. The answer keeps the structure of the sub-Gaussian case. By the reductions of Section~\ref{sec:method}, the least multiple is a supremum over the level of a ratio for two-point laws, and it is reached, or approached, along two-point laws, which are averages of multiples of the benchmark, as in Strassen's theorem. The answer departs from the sub-Gaussian case in what the two features of the sub-Gamma bound cost. The bounded window costs nothing beyond the elementary lower bound, since $\cE(\sigma,\alpha) = \sigma \vee \alpha$. The larger envelope inside the window costs a strictly larger constant, attained by an asymmetric two-point law. The two classes satisfy the same Bernstein tail bound \eqref{eq:bernstein}, and the convex order separates them, through the third cumulant and the divergence of the envelope at the endpoints of the window (Section~\ref{sec:SG}).

In dimension $d$, the answer departs from the sub-Gaussian case more deeply. The least multiple of the rotation-invariant Laplace vector is finite, but it grows at least like $d^{1/4}$ for both classes, and the same lower bound holds for every normalized benchmark (Theorem~\ref{thm:bench}). Bernstein's hypothesis thus gives sharp comparisons in dimension one, but, unlike the sub-Gaussian bound, no comparison with a single normalized benchmark that does not depend on the dimension.

\paragraph{The least multiples \Qref{1}.} The sub-exponential constant is explicit. It is attained by $\sigma\varepsilon$ when $\alpha \le \sigma$, and, when $\alpha > \sigma$, only approached, along two-point laws with rare upper atoms (Section~\ref{sec:SE}). The sub-Gamma constant is always attained, but it is given only by the variational formula of Theorem~\ref{thm:SG}, and we do not know a closed form for it. Nor do we know whether the extremal level $p_a$ is unique. The two ends of the range of $a = \alpha/\sigma$ correspond to two different extremal laws. The extremal level tends to $1/2$ as $a \to 0$, and to $0$ as $a \to \infty$ (Proposition~\ref{prop:ends} and its proof). In the second case, the extremal upper atom is rare, which is the regime of Bennett's inequality \citep[Section 2.7]{boucheron2013concentration}.

\paragraph{The costs \Qref{2}.} In the sense of \Qref{2}, the cost of the larger envelope is the gap between $\cG(1,a)$ and $1 \vee a$, and Proposition~\ref{prop:ends} gives it to first order at both ends of the range. As $a \to 0$, it is $(9/16)\,a^3$, while Theorem~\ref{thm:SG}(iii) only bounds it by $a/2 + a^2/8$. It comes from the positive third cumulant that the sub-Gamma envelope allows (Section~\ref{sec:SG}): the lower bound uses the two-point law of variance one and level close to $1/2 - 3a/4$. As $a \to \infty$, the cost is $\tau_0/a$, while Theorem~\ref{thm:SG}(iii) bounds it by $1/a$. It comes from the divergence of the envelope at the endpoint of the window, and the lower bound uses rare upper atoms, of probability $e^{-L_\star a^2}$. Between the two ends, we know the bounds of Theorem~\ref{thm:SG}(iii) and the numerical values of Table~\ref{tab:SG}. We do not know whether $(a + \sqrt{a^2 + 4\tau_0})/2$, the first factor of Proposition~\ref{prop:ends}(iii), is a lower bound on $\cG(1,a)$ for every $a$, as the numerical values suggest.

These costs are measured against a benchmark with a single scale parameter, and the benchmark has a cost of its own: it fits only the exponential regime of Bernstein's tail bound. In the sub-Gaussian limit $\alpha \to 0$, the comparison with $\sigma\Lap$ is weaker than the comparison with $\sqrt{\pi/2}\,\sigma G$, since $\sqrt{\pi/2}\,G \lecx \Lap$ (Section~\ref{sec:sums}). The variable suggested by the bound itself does not remove this cost. The compound Poisson variable $\mathcal{P}$ of \eqref{eq:poisson} is an equality case of the sub-Gamma bound at every positive frequency, but it does not dominate $-\mathcal{P}$, which also belongs to the class, and it is not the worst element of the class either: numerically, for $\sigma = \alpha$, it needs only $c\Lap$ with $c \approx 1.140\,\sigma$, while the class needs $c \approx 1.415\,\sigma$. As in the sub-Gaussian case, where the equality case $G$ of the bound needs only the multiple $1$ of $G$ while the class needs $\sqrt{\pi/2}\,G$, the worst elements are two-point laws, and not the equality case of the bound.

\paragraph{Sums and random vectors \Qref{3}.} For sums, the comparison passes to weighted sums, and to sums $\sum_i X_i x_i$ in $\R^d$, with an explicit coupling (Proposition~\ref{prop:sums}). For independent elements, it gives a Bernstein-shaped concentration bound for convex Lipschitz functions that does not depend on the dimension (Corollary~\ref{cor:concentration}), with the constant $c \ge \sigma \vee \alpha$ in place of $\sigma$ and $\alpha$, and Proposition~\ref{prop:averages} measures what a comparison made summand by summand loses. The multiples in \eqref{eq:avg} are sharp for the class given by Lemma~\ref{lem:sumclass}, but not for averages of $n$ independent elements of a class: for normalized sums, the least multiple tends to $\sigma\sqrt{2/\pi}$, the constant of the Gaussian comparison (Proposition~\ref{prop:clt}). Since Bernstein's inequality concerns such averages, their least multiple at a fixed $n$ is the main open question for sums.

For random vectors, the obstruction concerns the classes defined through projections, whose extremal elements have dependent coordinates. The simplex bound of Proposition~\ref{prop:dim}(ii) grows like $\alpha\sqrt{\log d/\pi}$, and the Kronecker bound (iv) like $d^{1/4}$. Both hold for every normalized benchmark (Theorem~\ref{thm:bench}), because on maxima of linear forms every benchmark averaged over rotations acts only through its mean radius. For $\alpha \le \sigma$, the least multiple already exceeds the one-dimensional constant in the plane (Proposition~\ref{prop:dim}(iii)). The least multiple $c_d$ is finite, at least of order $d^{1/4}$ and at most $d\,c_1$ (Proposition~\ref{prop:dim} and Theorem~\ref{thm:bench}(i)), and its rate of growth is open.

Whether $\Lap_d$ is optimal among normalized benchmarks for every convex function, that is, whether $\bar c_d$ and $c_d$ have the same order, is also open.

In the sub-Gaussian case, the theorem of \citet{vanhandel2025subgaussian} gives a comparison that does not depend on the dimension, with the Gaussian vector. Under Bernstein's hypothesis, no normalized benchmark does so (Theorem~\ref{thm:bench}). For the same reason, the short proof of \citet{song2026short} does not extend to Bernstein's hypothesis. It passes to independent copies, under which the classes are stable by the argument of Lemma~\ref{lem:sumclass}, but it then compares expected suprema with a factor that does not depend on the dimension, and Theorem~\ref{thm:bench} shows that this comparison fails here. A benchmark of another form, for instance a Gaussian vector plus an independent multiple of $\Lap_d$, does not help either. By the definition \eqref{eq:kappa}, every integrable $Z$ that dominates the class satisfies $K(Z) \ge \bar c_d$: its mean absolute projection grows at least like $d^{1/4}$, whatever its form.

\section*{Acknowledgments}
The authors are grateful to Sam Power for helpful exchanges, and in particular for pointing them to the compound Poisson representation of the sub-Gamma envelope and to the note \citet{power2025subgamma}, which helped in this revision. The authors were supported by the French National Research Agency (ANR) through the project BACKUP (ANR-23-CE40-0018-01).

\section*{Statement on the use of AI tools}
The original ideas, in particular the identification of the Laplace law as the benchmark for the sub-Gamma class and the link with the Kearns--Saul inequality, are the authors' own and were known to them before any use of AI tools. AI tools found the results of Section~\ref{sec:vectors} that go beyond the simplex construction: the lower bound of Proposition~\ref{prop:dim}(iv), based on the Kronecker element, and, following the approaches of \citet{zhang2026sharpmgf} and \citet{song2026short}, the extension to every normalized benchmark (Theorem~\ref{thm:bench}) and the finiteness of $c_d$ (Theorem~\ref{thm:bench}(i)). AI tools also adapted the proof of the lower bound in Proposition~\ref{prop:dim}(ii) to every dimension, helped to sharpen the explicit upper bound of Theorem~\ref{thm:SG}(iii), helped to clarify the presentation of the two-point reduction and of the Strassen coupling, and were used to edit and proofread the text and to check its numerical values. The authors checked every result and take full responsibility for the content of the paper. The paper was written by the authors.

{\small\hbadness=1200
\setstretch{1.0}\setlength{\bibsep}{5.75pt plus 1pt}
\bibliographystyle{agsm}
\bibliography{refs}
\par}

\appendix
\setlength{\abovedisplayskip}{4.1pt plus 2pt minus 3pt}\setlength{\belowdisplayskip}{4.1pt plus 2pt minus 3pt}
\renewcommand{\theHsection}{app.\Alph{section}}

\section{Proofs for the reduction and the largest spreads}
\label{app:method}

\begin{proof}[Proof of Lemma~\ref{lem:compression}]
Let $\mu_t^+ \coloneqq \E[X \mid X > t]$ and $\mu_t^- \coloneqq \E[X \mid X \le t]$, so that $Y$ equals $\mu_t^+$ on the event $\{X > t\}$ and $\mu_t^-$ on its complement, and
\begin{align*}
    \mu_t^- \le t < \mu_t^+ .
\end{align*}
Since $Y$ is centered, its larger atom $\mu_t^+$ is positive and its smaller atom $\mu_t^-$ is negative, so $Y$ has the law of $Y_{p_t,\nu}$ with $\nu = \mu_t^+ - \mu_t^-$. The conditional Jensen inequality gives $\E[f(Y)] \le \E[f(X)]$ for every convex $f$. The function $x \mapsto (x - t)_+$ is affine on each side of $t$, so its conditional expectation on each of the two events is the hinge of the conditional mean:
\begin{align*}
    \E\left[(X - t)_+ \mid X > t\right] &= \mu_t^+ - t = (\mu_t^+ - t)_+,
    \\
    \E\left[(X - t)_+ \mid X \le t\right] &= 0 = (\mu_t^- - t)_+ .
\end{align*}
Taking expectations gives $\SL_X(t) = \SL_Y(t)$.
\end{proof}

\begin{proof}[Proof of Lemma~\ref{lem:classes}]
(i) Let $0 < \lambda_0 < 1/\alpha$. The defining bound makes $\E[e^{\pm\lambda_0 X}]$ finite, and $e^{\lambda_0|x|} \le e^{\lambda_0 x} + e^{-\lambda_0 x}$, so $X$ is integrable. For $0 < \lambda < 1/\alpha$, Jensen's inequality and the defining bound give
\begin{align*}
    \lambda\,\E[X] \le \log\E\left[e^{\lambda X}\right] \le \frac{\sigma^2\lambda^2}{2(1 - \alpha\lambda)},
\end{align*}
the right side being the larger of the two envelopes. Dividing by $\lambda$ and letting $\lambda \to 0$ gives $\E[X] \le 0$, and the same argument for $-\lambda$ gives $\E[X] \ge 0$.

(ii) Both envelopes are even in $\lambda$ and both windows are symmetric, which gives the claim on $-X$. Let $\Psi$ be the envelope of the class. Since $\Psi$ increases with $|\lambda|$, for $r \in [0,1]$ and $\lambda$ in the window,
\begin{align*}
    \E\left[e^{\lambda r X}\right] \le \exp\left(\Psi(r\lambda)\right) \le \exp\left(\Psi(\lambda)\right),
\end{align*}
where the first inequality is the bound for $X$ at the frequency $r\lambda$, which lies in the window. So $rX \in \Cl$. For the conditional expectation, the conditional Jensen inequality applied to $x \mapsto e^{\lambda x}$ gives, at every frequency $\lambda$ of the window,
\begin{align*}
    \E\left[e^{\lambda\E[X \mid \mathcal{A}]}\right] \le \E\left[e^{\lambda X}\right],
\end{align*}
so $\E[X \mid \mathcal{A}]$ satisfies the defining bound of $X$.

(iii) For the sub-Gamma class, the frequency $\lambda/\sigma$ lies in the window $|\lambda/\sigma| < 1/\alpha$ exactly when $|\lambda| < 1/a$, and then
\begin{align*}
    \E\left[e^{\lambda X/\sigma}\right] = \E\left[e^{(\lambda/\sigma)X}\right],
    \qquad
    \frac{\sigma^2(\lambda/\sigma)^2}{2\left(1 - \alpha|\lambda|/\sigma\right)} = \frac{\lambda^2}{2\left(1 - a|\lambda|\right)} ,
\end{align*}
so the bound for $X$ at $\lambda/\sigma$ is the bound for $X/\sigma$ at $\lambda$ in $\SG(1,a)$. The sub-exponential class is treated in the same way, with the envelope $\lambda^2/2$.
\end{proof}

\begin{proof}[Proof of Lemma~\ref{lem:criterion}]
We first show that the condition is sufficient, with the coupling \eqref{eq:coupling}. Let $\overline{c\Lap} \coloneqq \E[c\Lap \mid \one(\Lap > z_p)]$ be the compression of $c\Lap$ at level $p$. Since $\Prob(\Lap > z_p) = p$, the variable $\overline{c\Lap}$ takes the value
\begin{align*}
    \E\left[c\Lap \mid \Lap > z_p\right] = c\,\eL(p)
\end{align*}
with probability $p$, by \eqref{eq:laplace}, and, being centered, the value $-c\,p\,\eL(p)/(1-p)$ with probability $1-p$. So $\overline{c\Lap}$ has the law of $Y_{p,\nu_*}$, where
\begin{align*}
    (1-p)\,\nu_* = c\,\eL(p),
\end{align*}
and $\overline{c\Lap} \lecx c\Lap$ by the conditional Jensen inequality. This proves the case of equality.

Let now $(1-p)\nu \le c\,\eL(p)$, so that $\tau = \nu/\nu_* \in (0,1]$ and $Y_{p,\nu}$ has the law of $\tau \overline{c\Lap}$. With $\tilde\Lap$ and $\eta$ as in \eqref{eq:coupling}, so that $\Lap$, $\tilde\Lap$ and $\eta$ are independent, the variable $Z' = \eta\,c\Lap + (1-\eta)\,c\tilde\Lap$ has the law of $c\Lap$, since it equals $c\Lap$ or an independent copy of it. As $\eta$ and $\tilde\Lap$ are independent of $\Lap$, and $\E[\tilde\Lap] = 0$,
\begin{align*}
    X' = \E\left[Z' \mid \one\left(\Lap > z_p\right)\right]
    = \E[\eta]\;\E\left[c\Lap \mid \one\left(\Lap > z_p\right)\right] + c\,\E\left[(1-\eta)\tilde\Lap\right]
    = \tau \overline{c\Lap} .
\end{align*}
The variable $X'$ takes two distinct values, according to whether $\Lap > z_p$, so conditioning on $X'$ is conditioning on $\one(\Lap > z_p)$, and $\E[Z' \mid X'] = X'$. This is a coupling as in \eqref{eq:strassen}, and the conditional Jensen inequality gives $Y_{p,\nu} \lecx c\Lap$.

\medskip
Conversely, let $Y_{p,\nu} \lecx c\Lap$, and look at the threshold $c z_p$ of Figure~\ref{fig:stoploss}. Keeping only the upper atom of $Y_{p,\nu}$ in the stop-loss transform at $c z_p$ gives the first inequality below, the comparison of stop-loss transforms at $c z_p$ gives the second, and \eqref{eq:laplace} gives the equalities:
\begin{align*}
    p\left((1-p)\nu - c z_p\right)
    &\le \E\left[(Y_{p,\nu} - c z_p)_+\right]
    \\
    &\le \E\left[(c\Lap - c z_p)_+\right]
    = \frac{c}{2}\,e^{-z_p}
    = c\,p .
\end{align*}
Dividing by $p$ gives $(1-p)\nu \le c\,(1 + z_p) = c\,\eL(p)$.
\end{proof}

\begin{proof}[Proof of Proposition~\ref{prop:reduction}]
Let $\Cl$ be $\SE(1,a)$ or $\SG(1,a)$, $\nu_{\Cl}$ the corresponding largest spread, and
\begin{align*}
    c_* \coloneqq \sup_{0 < p \le 1/2} \frac{(1-p)\,\nu_{\Cl}(p,a)}{\eL(p)} .
\end{align*}

We first show that every dominating multiple is at least $c_*$. Let $c > 0$ with $X \lecx c\Lap$ for every $X \in \Cl$, and let $p \in (0,1/2]$ and $\nu > 0$ with $Y_{p,\nu} \in \Cl$. Lemma~\ref{lem:criterion} gives $(1-p)\nu \le c\,\eL(p)$, and taking the supremum over $\nu$ and $p$ gives $c \ge c_*$.

\medskip
We now show that the multiple $c_*$ dominates the class. Let $c_* < \infty$, $X \in \Cl$ and $t \in \R$; by Lemma~\ref{lem:classes}(i), $X$ is integrable and centered. For every centered integrable $Z$, the stop-loss transform is nonnegative, and Jensen's inequality gives $\SL_Z(t) \ge \E[Z] - t = -t$, so
\begin{align*}
    \SL_Z(t) \ge (-t)_+ .
\end{align*}
If $\Prob(X > t) = 0$, then $X \le t$ almost surely, so $t \ge \E[X] = 0$ and $\SL_X(t) = 0 = (-t)_+$. If $\Prob(X > t) = 1$, then $X > t$ almost surely, so $t < \E[X] = 0$ and $\SL_X(t) = \E[X] - t = -t = (-t)_+$. In both cases, the last display gives $\SL_X(t) = (-t)_+ \le \SL_{c_*\Lap}(t)$. Otherwise, Lemma~\ref{lem:compression} gives a two-point variable $Y$ with the law of $Y_{p_t,\nu}$ for some $\nu > 0$ and $\SL_Y(t) = \SL_X(t)$, and $Y \in \Cl$ by Lemma~\ref{lem:classes}(ii). If $p_t \le 1/2$, then
\begin{align*}
    (1-p_t)\,\nu \le (1 - p_t)\,\nu_{\Cl}(p_t, a) \le c_*\,\eL(p_t)
\end{align*}
by definition of $c_*$, and Lemma~\ref{lem:criterion} gives $Y \lecx c_*\Lap$. If $p_t > 1/2$, then $-Y$ has the law of $Y_{1-p_t,\nu}$ and belongs to $\Cl$, so the same argument gives $-Y \lecx c_*\Lap$; hence $Y \lecx -c_*\Lap$, which has the law of $c_*\Lap$. In either case,
\begin{align*}
    \SL_X(t) = \SL_Y(t) \le \SL_{c_*\Lap}(t) .
\end{align*}
Since $X$ and $c_*\Lap$ are centered, \eqref{eq:slcrit} gives $X \lecx c_*\Lap$.
\end{proof}

The following lemma reduces the bounds \eqref{eq:SE} and \eqref{eq:SG} on the moment generating function of a two-point law of level at most $1/2$ to positive frequencies.

\begin{lem}
\label{lem:sign}
Let $p \in (0,1/2]$ and $s > 0$. Then $g_p(-s) \le g_p(s)$, with equality only if $p = 1/2$. Consequently, a two-point variable $Y_{p,\nu}$ satisfies the bound \eqref{eq:SE} or \eqref{eq:SG} as soon as it satisfies it for $\lambda > 0$.
\end{lem}

\begin{proof}
Writing $2\sinh x = e^x - e^{-x}$,
\begin{align*}
    e^{g_p(s)} - e^{g_p(-s)}
    &= p\left(e^{(1-p)s} - e^{-(1-p)s}\right) - (1-p)\left(e^{ps} - e^{-ps}\right)
    \\
    &= 2p\sinh\left((1-p)s\right) - 2(1-p)\sinh(ps)
    \\
    &= 2p(1-p)\,s\left(\frac{\sinh\left((1-p)s\right)}{(1-p)s} - \frac{\sinh(ps)}{ps}\right) .
\end{align*}
The map $x \mapsto \sinh(x)/x = \sum_{k \ge 0} x^{2k}/(2k+1)!$ is strictly increasing on $(0,\infty)$, its Taylor coefficients being positive, and $(1-p)s \ge ps$, with equality only if $p = 1/2$. The consequence follows because both envelopes are even in $\lambda$.
\end{proof}

The probability $q_s$ of Section~\ref{sec:method} satisfies $\diff q_s/\diff s = q_s(1 - q_s) > 0$, so it increases strictly in $s$, from $q_0 = p$. Starting from $g_p'(s) = q_s - p$ and differentiating twice more gives
\begin{align}
    g_p' = q_s - p,
    \qquad
    g_p'' = q_s(1 - q_s),
    \qquad
    g_p''' = q_s(1 - q_s)(1 - 2q_s) .
    \label{eq:cumulants}
\end{align}

\begin{lem}[{\citealp[Theorem 1.1]{schlemm2016kearns}}]
\label{lem:unimodal}
Let $p \in (0,1/2]$ and $\psi_p(s) \coloneqq g_p(s)/s^2$ for $s > 0$. If $p < 1/2$, then $\psi_p$ is strictly increasing on $(0, 2\beta_p]$ and strictly decreasing on $[2\beta_p, \infty)$, and
\begin{align*}
    q_{2\beta_p} = 1 - p,
    \qquad
    g_p(2\beta_p) = (1 - 2p)\,\beta_p,
    \qquad
    \psi_p(2\beta_p) = \kappa_p .
\end{align*}
For $p = 1/2$, $\psi_p$ is strictly decreasing on $(0,\infty)$, with limit $1/8$ at the origin.
\end{lem}

\citet{schlemm2016kearns} proves that $\psi_p$, extended to the whole line, is strictly unimodal with mode $2\beta_p$. We include the short proof, because it shows where the exchange of the two probabilities is used.

\begin{proof}
We first compute the values at $2\beta_p$. Since $e^{2\beta_p} = ((1-p)/p)^2$,
\begin{align*}
    q_{2\beta_p} = \frac{p\,(1-p)^2/p^2}{1 - p + p\,(1-p)^2/p^2} = \frac{(1-p)^2/p}{(1-p)/p} = 1 - p,
\end{align*}
so the tilt at $2\beta_p$ exchanges the probabilities of the two atoms. For the value of $g_p$,
\begin{align*}
    e^{g_p(2\beta_p)}
    &= p\Bigl(\frac{1-p}{p}\Bigr)^{2(1-p)} + (1-p)\Bigl(\frac{1-p}{p}\Bigr)^{-2p}
    = \Bigl(\frac{1-p}{p}\Bigr)^{-2p}\Bigl(p\,\frac{(1-p)^2}{p^2} + 1 - p\Bigr)
    \\
    &= \Bigl(\frac{1-p}{p}\Bigr)^{-2p}\,\frac{1-p}{p}
    = \Bigl(\frac{1-p}{p}\Bigr)^{1-2p},
\end{align*}
that is, $g_p(2\beta_p) = (1-2p)\beta_p$, and $\psi_p(2\beta_p) = (1-2p)\beta_p/(4\beta_p^2) = \kappa_p$. By \eqref{eq:cumulants}, $g_p'(2\beta_p) = 1 - 2p$.

\medskip
We now study the sign of the derivative of $\psi_p$. Let $D(s) \coloneqq s\,g_p'(s) - 2g_p(s)$, so that
\begin{gather*}
    \psi_p'(s) = \frac{D(s)}{s^3},
    \qquad
    D(0) = D'(0) = 0,
    \qquad
    D''(s) = s\,g_p'''(s),
    \\
    D(2\beta_p) = 2\beta_p(1 - 2p) - 2(1-2p)\beta_p = 0 .
\end{gather*}
By \eqref{eq:cumulants} and $q_{\beta_p} = 1/2$, which holds since $p\,e^{\beta_p} = 1 - p$, $g_p'''$ is positive on $(0, \beta_p)$ and negative on $(\beta_p, \infty)$, and so is $D''$. Since $D'' > 0$ on $(0, \beta_p)$ and $D'(0) = 0$, $D' > 0$ and $D > 0$ on $(0, \beta_p]$. On $(\beta_p,\infty)$, $D'$ decreases strictly, so it vanishes at most once there. It does vanish at some $s_1 \in (\beta_p, 2\beta_p)$: otherwise $D' > 0$ on $(0, 2\beta_p)$, and $D(2\beta_p) > D(0) = 0$. So $D$ increases on $(0, s_1]$ and decreases on $[s_1, \infty)$. As $D(0) = 0$, $D > 0$ on $(0, s_1]$, and as $D(2\beta_p) = 0$ with $D$ decreasing on $[s_1, \infty)$, we get
\begin{align*}
    D > 0 \ \text{ on } (0, 2\beta_p),
    \qquad
    D < 0 \ \text{ on } (2\beta_p,\infty),
\end{align*}
which is the monotonicity of $\psi_p$. For $p = 1/2$, $\beta_p = 0$ and $D'' < 0$ on $(0,\infty)$, so $D < 0$ there, and $\psi_p(s) \to g_p''(0)/2 = 1/8$ as $s \to 0$.
\end{proof}

\begin{proof}[Proof of Proposition~\ref{prop:ks}]
Since $\xi - p$ has the law of $Y_{p,1}$, the least coefficient is
\begin{align*}
    A_p(a) = \sup_{0 < |\lambda| \le 1/a} \frac{g_p(\lambda)}{\lambda^2} = \sup_{0 < \lambda \le 1/a} \psi_p(\lambda),
\end{align*}
since, by Lemma~\ref{lem:sign}, the negative frequencies give smaller values, strictly for $p < 1/2$. For $p < 1/2$, Lemma~\ref{lem:unimodal} shows that the supremum on $(0, 1/a]$ is attained only at $2\beta_p \wedge 1/a$, which gives the two cases, with $\psi_p(1/a) = a^2 g_p(1/a)$. For $p = 1/2$, the supremum is the limit $1/8 = \kappa_{1/2}$ at the origin.
\end{proof}

\begin{proof}[Proof of Proposition~\ref{prop:numax}]
By Lemma~\ref{lem:sign}, which makes the bound strict at every negative frequency when $p < 1/2$, and the substitution $s = \lambda\nu$, $Y_{p,\nu}$ belongs to $\SE(1,a)$ if and only if
\begin{align}
    \psi_p(s) \le \frac{1}{2\nu^2}
    \qquad \text{for every } s \in \Bigl(0, \frac{\nu}{a}\Bigr] .
    \label{eq:feasE}
\end{align}
Let $\nu_{\mathrm{KS}} \coloneqq \sqrt{2\beta_p/(1-2p)}$, so that
\begin{align*}
    \kappa_p = \frac{1}{2\nu_{\mathrm{KS}}^2},
    \qquad
    \bar\lambda(p) = \frac{2\beta_p}{\nu_{\mathrm{KS}}} .
\end{align*}

On the first branch, $\bar\lambda(p) \le 1/a$, so that $2\beta_p \le \nu_{\mathrm{KS}}/a$. By Lemma~\ref{lem:unimodal}, $\psi_p \le \kappa_p = 1/(2\nu_{\mathrm{KS}}^2)$ on $(0,\infty)$, so $\nu_{\mathrm{KS}}$ satisfies \eqref{eq:feasE}. If $\nu > \nu_{\mathrm{KS}}$, the point $s = 2\beta_p$ lies in $(0, \nu/a]$, and
\begin{align*}
    \psi_p(2\beta_p) = \frac{1}{2\nu_{\mathrm{KS}}^2} > \frac{1}{2\nu^2},
\end{align*}
so \eqref{eq:feasE} fails. Hence $\nuE(p,a) = \nu_{\mathrm{KS}}$, and by Lemma~\ref{lem:unimodal}, equality in \eqref{eq:feasE} holds only at $s = 2\beta_p$, that is, at $\lambda = \bar\lambda(p)$.

\medskip
On the second branch, $\bar\lambda(p) > 1/a$, that is, $g_p(2\beta_p) = (1-2p)\beta_p > 1/(2a^2)$. Since $g_p$ increases strictly from $0$ on $[0,\infty)$, there is a unique $m \in (0, 2\beta_p)$ with $g_p(m) = 1/(2a^2)$, which is the $m$ of the statement. Let $\nu \coloneqq am$, so that $(0, \nu/a] = (0, m]$. On this interval $\psi_p$ increases strictly, so
\begin{align*}
    \psi_p(s) \le \psi_p(m) = \frac{g_p(m)}{m^2} = \frac{1}{2a^2m^2} = \frac{1}{2\nu^2}
    \qquad \text{for } 0 < s \le m,
\end{align*}
with equality only at $s = m$, that is, at $\lambda = 1/a$. If $\nu' > \nu$, then $m \in (0, \nu'/a]$ and $\psi_p(m) = 1/(2\nu^2) > 1/(2\nu'^2)$, so \eqref{eq:feasE} fails for $\nu'$. Hence $\nuE(p,a) = am$.

\medskip
At the level one half, by Lemma~\ref{lem:unimodal}, $\psi_{1/2} < 1/8$ on $(0,\infty)$ with limit $1/8$ at the origin, so \eqref{eq:feasE} holds if and only if $1/(2\nu^2) \ge 1/8$, that is, $\nu \le 2$.
\end{proof}

\section{Proofs for the sub-exponential class}
\label{app:SE}

The proof of Theorem~\ref{thm:SE} compares the two branches of Proposition~\ref{prop:numax} with the tail mean $\eL(p)$. We use two elementary inequalities, the behavior of the second branch, and an element of the class with a moment generating function that is infinite outside the window.

\begin{lem}
\label{lem:elementary}
For every $p \in (0, 1/2]$, with the left side of (a) read as its limit $1$ at $p = 1/2$,
\begin{align*}
    \text{(a)}\quad (1-p)\sqrt{\frac{2\beta_p}{1-2p}} \le \eL(p),
    \qquad
    \text{(b)}\quad (1-p)\,\beta_p \le \eL(p),
\end{align*}
and (a) is an equality only at $p = 1/2$.
\end{lem}

\begin{proof}
(b) Since $\beta_p \ge 0$ for $p \le 1/2$, $1 - p \le 1$ and $\log 2 \le 1$,
\begin{align*}
    (1-p)\,\beta_p \le \log\frac{1-p}{p} \le \log\frac{1}{p} = \log\frac{1}{2p} + \log 2 \le 1 + \log\frac{1}{2p} = \eL(p) .
\end{align*}

(a) Let $x \coloneqq 1 - 2p \in [0,1)$. Then
\begin{align*}
    1 - p = \frac{1+x}{2},
    \qquad
    \eL(p) = 1 - \log(1-x),
    \qquad
    \beta_p = \log\frac{1+x}{1-x} = 2\operatorname{artanh} x,
\end{align*}
so that $2\beta_p/(1-2p) = 4\operatorname{artanh}(x)/x$, and, both sides being nonnegative, (a) is equivalent to
\begin{align}
    (1+x)^2\,\frac{\operatorname{artanh} x}{x} \le \left(1 - \log(1-x)\right)^2 .
    \label{eq:seriesineq}
\end{align}
Both sides are power series in $x$ with nonnegative coefficients, convergent on $[0,1)$:
\begin{align*}
    \frac{\operatorname{artanh} x}{x} = \sum_{j \ge 0}\frac{x^{2j}}{2j+1},
    \qquad
    -\log(1-x) = \sum_{n \ge 1}\frac{x^n}{n} .
\end{align*}
Let $u_n$ and $v_n$ be the coefficients of $x^n$ on the left and on the right of \eqref{eq:seriesineq}. Multiplying the first series by $(1+x)^2 = 1 + 2x + x^2$,
\begin{gather*}
    u_0 = 1,
    \qquad
    u_1 = 2,
    \qquad
    u_n = \frac{2}{n} \ \ (n \ge 3 \text{ odd}),
    \\
    u_n = \frac{1}{n+1} + \frac{1}{n-1} = \frac{2n}{n^2-1} \ \ (n \ge 2 \text{ even}) .
\end{gather*}
On the right, with $H_{n-1} \coloneqq \sum_{j=1}^{n-1} 1/j$ and the partial fractions $1/(j(n-j)) = (1/j + 1/(n-j))/n$,
\begin{align*}
    v_0 = 1,
    \qquad
    v_1 = 2,
    \qquad
    v_n = \frac{2}{n} + \sum_{j=1}^{n-1}\frac{1}{j(n-j)} = \frac{2}{n} + \frac{2}{n}\,H_{n-1} = \frac{2}{n}\left(1 + H_{n-1}\right) \ \ (n \ge 2) .
\end{align*}
For odd $n \ge 3$, $u_n = 2/n < v_n$. For even $n \ge 2$, $u_n \le v_n$ reads
\begin{align*}
    \frac{2n}{n^2 - 1} \le \frac{2}{n}\left(1 + H_{n-1}\right)
    \quad\Longleftrightarrow\quad
    \frac{n^2}{n^2 - 1} \le 1 + H_{n-1}
    \quad\Longleftrightarrow\quad
    \frac{1}{n^2 - 1} \le H_{n-1},
\end{align*}
which holds since $H_{n-1} \ge 1$ and $n^2 - 1 \ge 3$. So $u_n \le v_n$ for every $n$, which proves \eqref{eq:seriesineq}. Since $u_2 = 4/3 < 2 = v_2$, the inequality is strict for every $x > 0$, that is, for $p < 1/2$.
\end{proof}

\begin{lem}
\label{lem:endpoint}
Let $a > 0$, and let $p \in (0,1/2)$ be on the second branch of Proposition~\ref{prop:numax}.
\begin{enumerate}
\item[(i)] As $p \to 0$, $(1-p)\,\nuE(p,a) = a(\log(1/p) + \log(e^{1/(2a^2)} - 1) + o(1))$.
\item[(ii)] If $a \ge 1$, then $(1-p)\,\nuE(p,a) < a\,\eL(p)$.
\end{enumerate}
\end{lem}

\begin{proof}
Let $m \coloneqq \nuE(p,a)/a$ and $w_p \coloneqq (1-p)e^{-pm}$. The equation $g_p(m) = 1/(2a^2)$ of Proposition~\ref{prop:numax} reads $p\,e^{(1-p)m} = e^{1/(2a^2)} - w_p$, and taking logarithms,
\begin{align}
    (1-p)\,m = \log\frac{1}{p} + \log\left(e^{1/(2a^2)} - w_p\right) .
    \label{eq:mexpand}
\end{align}

(i) Since $0 < m < 2\beta_p$,
\begin{align*}
    0 < p\,m < 2p\,\beta_p = 2p\log\frac{1-p}{p} \longrightarrow 0
    \qquad (p \to 0),
\end{align*}
so $w_p \to 1$, and \eqref{eq:mexpand} gives the expansion.

(ii) Since $\eL(p) = \log(e/(2p))$, the claim $(1-p)m < \eL(p)$ is equivalent, after exponentiating and multiplying by $p$, to $p\,e^{(1-p)m} < e/2$, that is, to
\begin{align}
    e^{1/(2a^2)} - w_p < \frac{e}{2} .
    \label{eq:master}
\end{align}
Since $m < 2\beta_p$ and $e^{-\beta_p} = p/(1-p)$,
\begin{align*}
    w_p > (1-p)\,e^{-2p\beta_p} = (1-p)\Bigl(\frac{p}{1-p}\Bigr)^{2p} = (1-p)^{1-2p}\,p^{2p} .
\end{align*}
We bound the logarithm of the right side term by term. The function $u \mapsto u\log u$ is at least $-1/e$ on $(0,1]$, so $2p\log p \ge -2/e$. The concave function $p \mapsto \log(1-p)$ lies above its chord between $0$ and $1/2$, so $\log(1-p) \ge -2p\log 2$ on $[0,1/2]$, and, with $p(1-2p) \le 1/8$,
\begin{align*}
    (1-2p)\log(1-p) \ge -2\log 2\;p(1-2p) \ge -\frac{\log 2}{4} .
\end{align*}
Hence, with $2/e < 0.7358$ and $(\log 2)/4 < 0.1733$,
\begin{align*}
    \log\left((1-p)^{1-2p}\,p^{2p}\right) \ge -\frac{2}{e} - \frac{\log 2}{4} > -0.91,
    \qquad\text{so}\qquad
    w_p > e^{-0.91} > 0.4,
\end{align*}
the last inequality because $\log(5/2) > 0.91$. As $a \ge 1$, $e^{1/(2a^2)} \le e^{1/2} < 1.65$, while $e/2 > 1.35$, so
\begin{align*}
    e^{1/(2a^2)} - w_p < 1.65 - 0.4 = 1.25 < 1.35 < \frac{e}{2},
\end{align*}
which is \eqref{eq:master}.
\end{proof}

\begin{lem}
\label{lem:edge}
For every $a > 0$, the class $\SE(1,a)$ contains a symmetric variable $X$ with $\E[e^{\lambda X}] = \infty$ for every $|\lambda| > 1/a$. Consequently, every integrable $Z$ that dominates $\SE(1,a)$ in the convex order satisfies $\E[e^{\lambda Z}] = \infty$ for $|\lambda| > 1/a$, and $c\Lap$ dominates $\SE(1,a)$ only if $c \ge a$.
\end{lem}

\begin{proof}
We build $X$ from a point mass at $0$ and atoms at $\pm n$, $n \ge 1$, whose probabilities decay like $n^{-4}e^{-n/a}$: the factor $e^{-n/a}$ makes $\E[e^{\lambda X}]$ infinite for $|\lambda| > 1/a$, and the factor $n^{-4}$ keeps it finite, and small enough, for $|\lambda| \le 1/a$. Let
\begin{align*}
    S_a \coloneqq \sum_{n \ge 1} n^{-4} e^{-n/a},
    \qquad
    r_n \coloneqq \frac{n^{-4} e^{-n/a}}{S_a} \quad (n \ge 1),
    \qquad
    \delta \coloneqq \frac{1}{2} \wedge \frac{1}{2C_a},
    \\
    C_a \coloneqq \sum_{n \ge 1} r_n\,n^2\cosh(n/a) = \frac{1}{2S_a}\sum_{n \ge 1}\frac{1 + e^{-2n/a}}{n^2} < \infty,
\end{align*}
so that $(r_n)_{n \ge 1}$ is a probability on the positive integers, and let $X$ have the symmetric law
\begin{align*}
    \Prob(X = 0) = 1 - 2\delta,
    \qquad
    \Prob(X = n) = \Prob(X = -n) = \delta\,r_n \quad (n \ge 1),
\end{align*}
which is a probability law since $\delta \le 1/2$, and $X$ is integrable. We first show that $X \in \SE(1,a)$. Comparing Taylor coefficients, $\cosh x - 1 \le (x^2/2)\cosh x$ for every real $x$, since $(2k)! \ge 2\,(2k-2)!$ for $k \ge 1$. For $|\lambda| \le 1/a$, we also have $\cosh(\lambda n) \le \cosh(n/a)$, and, with $\delta C_a \le 1/2$ and $1 + y \le e^y$,
\begin{align*}
    \E\left[e^{\lambda X}\right]
    &= 1 + 2\delta\sum_{n \ge 1} r_n\left(\cosh(\lambda n) - 1\right)
    \le 1 + \delta\lambda^2\sum_{n \ge 1} r_n\,n^2\cosh(n/a)
    \\
    &= 1 + \delta C_a\,\lambda^2
    \le 1 + \frac{\lambda^2}{2}
    \le e^{\lambda^2/2} .
\end{align*}
We now show that the moment generating function of $X$ is infinite outside the window. For $\lambda > 1/a$, keeping only the atoms at $+n$,
\begin{align*}
    \E\left[e^{\lambda X}\right] \ge \delta\sum_{n \ge 1} r_n\,e^{\lambda n} = \frac{\delta}{S_a}\sum_{n \ge 1} n^{-4}e^{(\lambda - 1/a)n} = \infty,
\end{align*}
and the same holds for $\lambda < -1/a$ by symmetry. If $Z$ dominates the class, the convex function $x \mapsto e^{\lambda x}$ gives $\E[e^{\lambda Z}] \ge \E[e^{\lambda X}] = \infty$ for every $|\lambda| > 1/a$. For $Z = c\Lap$, the moment generating function $\E[e^{\lambda c\Lap}] = 1/(1 - c^2\lambda^2)$ is finite for $|\lambda| < 1/c$, so no $\lambda$ satisfies $1/a < |\lambda| < 1/c$, that is, $c \ge a$.
\end{proof}

\begin{proof}[Proof of Theorem~\ref{thm:SE}]
By Lemma~\ref{lem:classes}(iii), we may take $\sigma = 1$ and $\alpha = a$. Let
\begin{align*}
    R(p) \coloneqq \frac{(1-p)\,\nuE(p,a)}{\eL(p)},
    \qquad
    \cE(1,a) = \sup_{0 < p \le 1/2} R(p),
\end{align*}
the identity being Proposition~\ref{prop:reduction}.

\medskip
We first prove the upper bound. At $p = 1/2$, $R(1/2) = (1/2)\cdot 2/1 = 1$ by Proposition~\ref{prop:numax}. Let $p < 1/2$.
\begin{itemize}
\item On the first branch of Proposition~\ref{prop:numax}, $R(p) \le 1$ by Lemma~\ref{lem:elementary}(a).
\item On the second branch, if $a \ge 1$, then $R(p) < a$ by Lemma~\ref{lem:endpoint}(ii).
\item On the second branch, if $a < 1$, the window $[-1/a, 1/a]$ contains $[-1, 1]$, so every element of $\SE(1,a)$ belongs to $\SE(1,1)$, and $\nuE(p,a) \le \nuE(p,1)$. The first two cases, applied with $a = 1$, give $(1-p)\,\nuE(p,1) \le \eL(p)$ on either branch, so $R(p) \le 1$.
\end{itemize}
Hence $\cE(1,a) \le 1 \vee a$, and Proposition~\ref{prop:reduction} gives $X \lecx \cE(1,a)\,\Lap$ for every element.

\medskip
We now prove the lower bound. The variable $\varepsilon$ belongs to $\SE(1,a)$, since, comparing the Taylor coefficients $1/(2k)!$ and $1/(2^k k!)$,
\begin{align*}
    \E\left[e^{\lambda\varepsilon}\right] = \cosh\lambda = \sum_{k \ge 0}\frac{\lambda^{2k}}{(2k)!} \le \sum_{k \ge 0}\frac{\lambda^{2k}}{2^k k!} = e^{\lambda^2/2} .
\end{align*}
If $\varepsilon \lecx c\Lap$, the convex function $|x|$ gives $1 \le c\,\E[|\Lap|] = c$. As $p \to 0$, $\bar\lambda(p) \to \infty$, so $p$ eventually lies on the second branch of Proposition~\ref{prop:numax}, and Lemma~\ref{lem:endpoint}(i), divided by $\eL(p) = \log(1/p) + 1 - \log 2$, gives $R(p) \to a$; Lemma~\ref{lem:edge} gives the same bound by a different route. Hence $\cE(1,a) = 1 \vee a$, which proves (i), and (ii) is the computation for $\varepsilon$ with $f(x) = |x|$.

\medskip
It remains to prove the strict comparison when $a > 1$. The bounds above give $R(p) < a$ for every $p \in (0,1/2]$, so every two-point element $Y_{p,\nu}$ satisfies $(1-p)\nu < a\,\eL(p)$. We first show that
\begin{align}
    \SL_X(t) < \SL_{a\Lap}(t)
    \qquad \text{for every } X \in \SE(1,a) \text{ and } t \in \R .
    \label{eq:strictsl}
\end{align}
For every $p \in (0,1/2]$, the stop-loss transform of $a\Lap$ is convex, with slope $-\Prob(a\Lap > a z_p) = -p$ at $a z_p$, so it lies above its tangent there:
\begin{align*}
    \SL_{a\Lap}(t) \ge \SL_{a\Lap}(a z_p) - p\,(t - a z_p) = a\,p\,\eL(p) - p\,t
    \qquad \text{for every } t \in \R .
\end{align*}
It also exceeds $(-t)_+$ at every $t$, since $\SL_{a\Lap}(t) - (-t)_+ = (a/2)\,e^{-|t|/a} > 0$. If $\Prob(X > t) \in \{0, 1\}$, then $\SL_X(t) = (-t)_+ < \SL_{a\Lap}(t)$. Otherwise, let $Y$ be the compression of Lemma~\ref{lem:compression}, with the law of $Y_{p,\nu}$ for $p = \Prob(X > t)$; it belongs to $\SE(1,a)$ by Lemma~\ref{lem:classes}(ii). If $p \le 1/2$, \eqref{eq:twopointsl} and the tangent give
\begin{align*}
    \SL_X(t) = \SL_Y(t) = \max\bigl\{(-t)_+,\ p\left((1-p)\nu - t\right)\bigr\} < \SL_{a\Lap}(t) .
\end{align*}
If $p > 1/2$, the same holds for $-Y$ at $-t$, and the identity $\SL_Z(t) = -t + \SL_{-Z}(-t)$, valid for every centered integrable $Z$, transfers it to $Y$ at $t$, $a\Lap$ being symmetric. This proves \eqref{eq:strictsl}.

Let now $f$ be convex, not affine, with $\E[f(a\Lap)] < \infty$, and let $\mu$ be the nonnegative measure with $\mu((s,t]) = f_+'(t) - f_+'(s)$ for $s < t$, which is not zero. Integrating by parts,
\begin{align*}
    f(x) = f(0) + f_+'(0)\,x + \int_{(0,\infty)} (x - t)_+\,\mu(\diff t) + \int_{(-\infty,0]} (t - x)_+\,\mu(\diff t) .
\end{align*}
For a centered integrable $Z$, $\E[(t - Z)_+] = \SL_Z(t) + t$, so, by Tonelli's theorem,
\begin{align*}
    \E[f(Z)] = f(0) + \int_{(0,\infty)} \SL_Z(t)\,\mu(\diff t) + \int_{(-\infty,0]} \left(\SL_Z(t) + t\right)\mu(\diff t)
    \ \in (-\infty, \infty] .
\end{align*}
Applied to $Z = X$ and $Z = a\Lap$, the integrands are nonnegative, since $\SL_Z(t) \ge (-t)_+$. By \eqref{eq:strictsl}, the integrands for $X$ are bounded by those for $a\Lap$, so $\E[f(X)]$ is finite, and
\begin{align*}
    \E[f(a\Lap)] - \E[f(X)] = \int_{\R} \left(\SL_{a\Lap}(t) - \SL_X(t)\right)\mu(\diff t) > 0,
\end{align*}
since the integrand is positive and $\mu \ne 0$. This proves (iii).
\end{proof}

\section{Proofs for the sub-Gamma class}
\label{app:SG}

Throughout, $\sigma = 1$, $a > 0$, $p \in (0,1/2]$, $\ell \coloneqq \log(1/p)$, so that $\eL(p) = \ell + h_0$ with $h_0 = 1 - \log 2$. By Lemma~\ref{lem:sign}, $Y_{p,\nu}$ belongs to $\SG(1,a)$ if and only if
\begin{align}
    g_p(\lambda\nu) \le \frac{\lambda^2}{2(1 - a\lambda)}
    \qquad \text{for every } \lambda \in \Bigl(0, \frac{1}{a}\Bigr) .
    \label{eq:feasG}
\end{align}

\begin{lem}
\label{lem:spreadG}
If $Y_{p,\nu} \in \SG(1,a)$, then
\begin{align*}
    \text{(a)}\quad p(1-p)\,\nu^2 \le 1,
    \qquad
    \text{(b)}\quad (1-p)\,\nu \le a\ell + \sqrt{2\ell} .
\end{align*}
\end{lem}

\begin{proof}
(a) As $\lambda \to 0$, $g_p(\lambda\nu)/(\lambda^2/2) \to g_p''(0)\,\nu^2 = p(1-p)\nu^2$, the variance of $Y_{p,\nu}$, while the envelope of \eqref{eq:feasG} divided by $\lambda^2/2$ tends to one.

(b) At $\lambda = \theta/a$ with $0 < \theta < 1$, keeping only the upper atom, \eqref{eq:feasG} gives
\begin{align*}
    -\ell + \frac{\theta}{a}(1-p)\nu = \log\left(p\,e^{\theta(1-p)\nu/a}\right) \le g_p\Bigl(\frac{\theta\nu}{a}\Bigr) \le \frac{\theta^2}{2a^2(1-\theta)},
\end{align*}
so that
\begin{align*}
    (1-p)\nu \le \frac{a\ell}{\theta} + \frac{\theta}{2a(1-\theta)} .
\end{align*}
With $r \coloneqq \theta/(1-\theta) \in (0,\infty)$, so that $1/\theta = 1 + 1/r$, the right side is
\begin{align*}
    a\ell + \frac{a\ell}{r} + \frac{r}{2a},
\end{align*}
and the choice $r = a\sqrt{2\ell}$ gives $a\ell + \sqrt{\ell/2} + \sqrt{\ell/2} = a\ell + \sqrt{2\ell}$.
\end{proof}

\begin{lem}
\label{lem:fixedp}
For every $p \in (0,1/2]$, the two-point variable $Y_{p,\nu}$ belongs to $\SG(1,a)$ if and only if $\nu \le \nuG(p,a)$, where
\begin{align*}
    \nuG(p,a) = \inf_{s > 0}\, \nu_+(s),
    \qquad
    \nu_+(s) \coloneqq \frac{s}{2}\biggl(a + \sqrt{a^2 + \frac{2}{g_p(s)}}\,\biggr) .
\end{align*}
Moreover, $\nuG(1/2, a) = 2$, and $p \mapsto \nuG(p,a)$ is upper semicontinuous on $(0,1/2)$.
\end{lem}

\begin{proof}
Let $\nu > 0$. With $s = \lambda\nu$, the condition \eqref{eq:feasG} reads
\begin{align*}
    g_p(s) \le \frac{s^2}{2\nu\left(\nu - as\right)}
    \qquad \text{for every } s \in \Bigl(0, \frac{\nu}{a}\Bigr) .
\end{align*}
For such $s$, we have $\nu - as > 0$ and $g_p(s) > 0$, and the inequality is equivalent to
\begin{align*}
    \nu^2 - as\,\nu - \frac{s^2}{2g_p(s)} \le 0 .
\end{align*}
The left side is a quadratic function of $\nu$ with a negative root and the positive root $\nu_+(s)$, so, as $\nu > 0$, the inequality holds if and only if $\nu \le \nu_+(s)$. For $s \ge \nu/a$, we have $\nu \le as < \nu_+(s)$ in any case. So $Y_{p,\nu}$ belongs to $\SG(1,a)$ if and only if $\nu \le \nu_+(s)$ for every $s > 0$. The function $\nu_+$ is continuous and positive on $(0,\infty)$. It tends to $1/\sqrt{p(1-p)}$ as $s \to 0$, since $g_p(s)/s^2 \to p(1-p)/2$, and to infinity as $s \to \infty$, since $\nu_+(s) \ge as$. Its infimum is therefore positive, the admissible spreads are those of $(0, \inf_{s > 0}\nu_+(s)]$, and their supremum \eqref{eq:numax} is this infimum, which proves the first claim.

At $p = 1/2$, Lemma~\ref{lem:spreadG}(a) gives $\nuG(1/2,a) \le 2$, and $\varepsilon = Y_{1/2,2}$, in $\SE(1,a)$ by the proof of Theorem~\ref{thm:SE}, hence in $\SG(1,a)$ by \eqref{eq:inclusions}, gives the converse. Finally, $(p, s) \mapsto g_p(s)$ is continuous, so $p \mapsto \nu_+(s)$ is continuous for every $s > 0$, and an infimum of continuous functions is upper semicontinuous, which is the last claim.
\end{proof}

\begin{proof}[Proof of Theorem~\ref{thm:SG}]
By Lemma~\ref{lem:classes}(iii), we may take $\sigma = 1$. Let
\begin{align*}
    R(p) \coloneqq \frac{(1-p)\,\nuG(p,a)}{\eL(p)},
    \qquad
    \cG(1,a) = \sup_{0 < p \le 1/2} R(p),
\end{align*}
the identity being Proposition~\ref{prop:reduction}.

\medskip
We first prove the upper bounds in (iii) and the formula in (i). Let $b > a$. For $|\lambda| \le 1/b$, we have $|\lambda| < 1/a$ and $1 - a|\lambda| \ge 1 - a/b > 0$, so every $X \in \SG(1,a)$ satisfies
\begin{align*}
    \E\left[e^{\lambda X}\right] \le \exp\Bigl(\frac{\lambda^2}{2(1 - a|\lambda|)}\Bigr) \le \exp\Bigl(\frac{b}{b - a}\cdot\frac{\lambda^2}{2}\Bigr),
    \qquad |\lambda| \le \frac1b .
\end{align*}
This is the inclusion \eqref{eq:window}, $\SG(1,a) \subseteq \SE(b/(b - a), b)$, and Theorem~\ref{thm:SE} gives
\begin{align*}
    \cG(1,a) \le \cE\Bigl(\sqrt{\frac{b}{b - a}},\, b\Bigr) = \sqrt{\frac{b}{b - a}} \vee b,
    \qquad b > a .
\end{align*}
Let $b_\star \coloneqq (a + \sqrt{a^2 + 4})/2$, the positive root of $b^2 = ab + 1$. Then $b_\star > a$ and $b_\star(b_\star - a) = 1$, so that $b_\star/(b_\star - a) = b_\star^2$, and the two terms of the maximum coincide at $b = b_\star$:
\begin{align*}
    \cG(1,a) \le b_\star = \frac{a + \sqrt{a^2 + 4}}{2} .
\end{align*}
Since $\sqrt{a^2 + 4} \le a + 2$, we have $b_\star \le 1 + a$, and since $b_\star = a + 1/b_\star$ with $b_\star > a$, we have $b_\star \le a + 1/a$. This proves the upper bounds in (iii). In particular, $\cG(1,a)$ is finite, so Proposition~\ref{prop:reduction} gives the domination in (i), and Lemma~\ref{lem:fixedp} gives the formula for $\nuG(p,a)$.

\medskip
We now show that the constant exceeds one. For $x \in (0,1)$, let $p \coloneqq (1-x)/2$ and $\nu \coloneqq 1/\sqrt{p(1-p)} = 2/\sqrt{1-x^2}$, so that $V \coloneqq Y_{p,\nu}$ has variance one, and let $\chi(\lambda) \coloneqq g_p(\lambda\nu)$. By \eqref{eq:cumulants},
\begin{align*}
    \chi'''(\lambda) = \nu^3\,\omega(q_{\lambda\nu}),
    \qquad
    \omega(q) \coloneqq q(1-q)(1-2q) .
\end{align*}
The function $\omega$ decreases on $[(3-\sqrt3)/6, 1/2]$, its derivative $6q^2 - 6q + 1$ being at most zero between its roots $(3 \pm \sqrt3)/6$, and $\omega$ is at most zero on $[1/2, 1)$; moreover $q_{\lambda\nu} \ge q_0 = p$ for $\lambda \ge 0$. So, if $x \le 1/\sqrt3$, that is, $p \ge (3 - \sqrt3)/6$, then for every $\lambda \ge 0$,
\begin{align*}
    \chi'''(\lambda) \le \chi'''(0) = \nu^3\,p(1-p)(1-2p) = \frac{1 - 2p}{\sqrt{p(1-p)}} = \frac{2x}{\sqrt{1-x^2}} .
\end{align*}
Since $\chi(0) = \chi'(0) = 0$ and $\chi''(0) = p(1-p)\nu^2 = 1$, Taylor's formula with integral remainder gives, for $0 < \lambda < 1/a$,
\begin{align*}
    \chi(\lambda) = \frac{\lambda^2}{2} + \frac{1}{2}\int_0^\lambda (\lambda - u)^2\,\chi'''(u)\,\diff u
    \le \frac{\lambda^2}{2} + \frac{x\,\lambda^3}{3\sqrt{1-x^2}}
    \le \frac{\lambda^2}{2}\left(1 + a\lambda\right) \le \frac{\lambda^2}{2(1 - a\lambda)},
\end{align*}
as soon as $2x/\sqrt{1-x^2} \le 3a$, the last inequality being $(1 + a\lambda)(1 - a\lambda) \le 1$. So, for every $x \in (0, 1/\sqrt3]$ with $2x/\sqrt{1-x^2} \le 3a$, $V \in \SG(1,a)$ by \eqref{eq:feasG}, and
\begin{align*}
    R(p) \ge \frac{(1-p)\,\nu}{\eL(p)} = \frac{A(x)}{B(x)},
    \qquad
    A(x) \coloneqq \sqrt{\frac{1+x}{1-x}},
    \qquad
    B(x) \coloneqq 1 - \log(1-x) .
\end{align*}
Here $A(0) = B(0) = 1$, and, since $A(x)^2 = (1+x)/(1-x) > (1+x)^2$, that is, $A(x) > 1 + x$,
\begin{align*}
    A'(x) = \frac{A(x)}{1-x^2} > \frac{1+x}{1-x^2} = \frac{1}{1-x} = B'(x) .
\end{align*}
So $A > B$ on $(0,1)$, and, since every small $x > 0$ is admissible, $\cG(1,a) > 1$.

\medskip
We now show that the constant exceeds $a$. Let $\nu \coloneqq a(\ell + 1)/(1-p)$, so that the upper atom $(1-p)\nu$ of $Y_{p,\nu}$ is $a(\ell+1)$. At $\lambda = \theta/a$, with $0 < \theta < 1$, let
\begin{align*}
    u \coloneqq \theta(\ell + 1),
    \qquad
    w \coloneqq \frac{\theta p(\ell+1)}{1-p},
    \qquad\text{so that}\qquad
    p\,u = (1-p)\,w,
\end{align*}
and the moment generating function of $Y_{p,\nu}$ at $\lambda$ is $p\,e^u + (1-p)e^{-w}$.

For $0 < \theta \le 1/2$, the linear terms cancel, and the bounds $e^u - 1 - u \le u^2e^u/2$, $e^{-w} - 1 + w \le w^2/2$ and $\log y \le y - 1$ give
\begin{align*}
    g_p(\lambda\nu)
    &\le p\left(e^u - 1 - u\right) + (1-p)\left(e^{-w} - 1 + w\right)
    \le \frac{p\,u^2e^u}{2} + \frac{(1-p)\,w^2}{2}
    \\
    &= \frac{\theta^2(\ell+1)^2}{2}\Bigl(p\,e^{\theta(\ell+1)} + \frac{p^2}{1-p}\Bigr)
    \le \frac{\theta^2(\ell+1)^2}{2}\Bigl(e^{(1-\ell)/2} + \frac{p^2}{1-p}\Bigr),
\end{align*}
since $p\,e^{\theta(\ell+1)} \le p\,e^{(\ell+1)/2} = e^{(1-\ell)/2}$. As $p \to 0$, $(\ell+1)^2(e^{(1-\ell)/2} + p^2/(1-p)) \to 0$, so for all small $p$ the right side is below $\theta^2/(2a^2) \le \theta^2/(2a^2(1-\theta))$.

For $1/2 \le \theta < 1$, let $\delta \coloneqq 1 - \theta \in (0, 1/2]$. The bounds $e^{-w} \le 1$ and $\log y \le y - 1$ give
\begin{align*}
    g_p(\lambda\nu) \le p\,e^{u} - p \le p\,e^{\theta(\ell+1)} = e^{\theta - \delta\ell} \le e^{1 - \delta\ell},
\end{align*}
while the envelope is
\begin{align*}
    \frac{\lambda^2}{2(1 - a\lambda)} = \frac{\theta^2}{2a^2\delta} \ge \frac{1}{8a^2\delta} .
\end{align*}
Since $\delta e^{-\delta\ell} \le 1/(e\ell)$, we have $\delta\,e^{1 - \delta\ell} \le 1/\ell < 1/(8a^2)$ as soon as $\ell > 8a^2$, and then $g_p(\lambda\nu) \le e^{1 - \delta\ell} < 1/(8a^2\delta)$, which is at most the envelope.

For all small $p$, \eqref{eq:feasG} therefore holds, $Y_{p,\nu} \in \SG(1,a)$, and
\begin{align*}
    R(p) \ge \frac{a(\ell + 1)}{\ell + h_0} > a,
\end{align*}
since $h_0 = 1 - \log 2 < 1$.

\medskip
We now show that the constant is attained. As $p \to 1/2$, Lemma~\ref{lem:spreadG}(a) gives $\nuG(p,a) \le 1/\sqrt{p(1-p)} \to 2$, so $\limsup R(p) \le 1$. As $p \to 0$, Lemma~\ref{lem:spreadG}(b) gives $\limsup R(p) \le a$. Since $\cG(1,a) > 1 \vee a$ and $R(1/2) = 1$, there are $0 < p_1 < p_2 < 1/2$ such that the supremum of $R$ is its supremum on $[p_1, p_2]$, on which $R$ is upper semicontinuous by Lemma~\ref{lem:fixedp}; it is therefore attained at some $p_a \in [p_1, p_2]$. Let $c \coloneqq \cG(1,a)$ and $Y \coloneqq Y_{p_a,\nu}$ with $\nu \coloneqq \nuG(p_a,a)$, an element of the class by Lemma~\ref{lem:fixedp}; since $R(p_a) = c$, $(1-p_a)\nu = c\,\eL(p_a)$. At $t \coloneqq c z_{p_a}$, which lies between $0$ and the upper atom, \eqref{eq:twopointsl} and \eqref{eq:laplace} give
\begin{align*}
    \E\left[(Y - t)_+\right] = p_a\left(c\,\eL(p_a) - c z_{p_a}\right) = c\,p_a = \frac{c}{2}\,e^{-z_{p_a}} = \E\left[(c\Lap - t)_+\right],
\end{align*}
which is (ii).

\medskip
Finally, (iii) gives
\begin{align*}
    1 \vee a < \cG(1,a) \le (1 + a) \wedge \Bigl(a + \frac1a\Bigr),
\end{align*}
so that $\cG(1,a)/(1 \vee a) \to 1$ as $a \to 0$ and as $a \to \infty$.
Multiplying by $\sigma$ gives the general case.
\end{proof}

The proofs of the three parts of Proposition~\ref{prop:ends} use the formula of Lemma~\ref{lem:fixedp} at well-chosen frequencies. With $R(p)$ as in the proof of Theorem~\ref{thm:SG}, so that $\cG(1,a)$ is the supremum of $R$ over $(0,1/2]$, let
\begin{align*}
    G_0 \coloneqq \log\Bigl(1 + \frac{e}{2}\Bigr),
    \qquad
    \tau_0 = \frac{1}{2G_0},
    \qquad
    B_\infty(a) \coloneqq \frac{a + \sqrt{a^2 + 4\tau_0}}{2} = \frac{1}{2}\biggl(a + \sqrt{a^2 + \frac{2}{G_0}}\,\biggr) .
\end{align*}

\begin{proof}[Proof of Proposition~\ref{prop:ends}(iii)]
Let $p \in (0,1/2]$. Since $g_p$ increases strictly from $0$ to $\infty$ on $[0,\infty)$, there is a unique $s > 0$ with $g_p(s) = G_0$, and Lemma~\ref{lem:fixedp} at this $s$ gives
\begin{align*}
    \nuG(p,a) \le \frac{s}{2}\biggl(a + \sqrt{a^2 + \frac{2}{G_0}}\,\biggr) = s\,B_\infty(a),
    \qquad\text{so that}\qquad
    R(p) \le B_\infty(a)\,\frac{u}{\eL(p)},
    \quad u \coloneqq (1-p)\,s .
\end{align*}
It remains to bound $u$. By \eqref{eq:gp}, and by $e^{-y} \ge 1 - y$ at $y = pu/(1-p)$,
\begin{align*}
    1 + \frac{e}{2} = e^{g_p(s)} = p\,e^{u} + (1-p)\,e^{-pu/(1-p)} \ge p\,e^u + 1 - p - p\,u,
\end{align*}
so that
\begin{align*}
    p\,e^u \le \frac{e}{2} + p\,(1+u) = \frac{e}{2}\Bigl(1 + \frac{2p}{e}\,(1+u)\Bigr) .
\end{align*}
Taking logarithms, with $\log(1+y) \le y$ and $\eL(p) = \log(e/(2p))$,
\begin{align*}
    u \le \eL(p) + \frac{2p}{e}\,(1+u),
    \qquad\text{that is}\qquad
    u \le \frac{\eL(p) + 2p/e}{1 - 2p/e},
\end{align*}
since $2p/e < 1$. Dividing by $\eL(p)$ gives (iii).
\end{proof}

\begin{proof}[Proof of Proposition~\ref{prop:ends}(ii)]
We first prove the upper bound. Let $p \in (0,1/2]$. If $\sqrt{2\ell} \le a h_0$, Lemma~\ref{lem:spreadG}(b) gives
\begin{align*}
    R(p) \le \frac{a\ell + \sqrt{2\ell}}{\ell + h_0} \le \frac{a\ell + a h_0}{\ell + h_0} = a .
\end{align*}
Since $\cG(1,a) > a$, the supremum of $R$ is not reached at such a level, so that $\log(1/p_a) > h_0^2a^2/2$ at every level $p_a$ where it is reached. Otherwise, $p < \bar p_a \coloneqq e^{-h_0^2a^2/2}$, and (iii), with $\eL(p) \ge 1$, gives
\begin{align*}
    R(p) \le B_\infty(a)\,\frac{1 + 2\bar p_a/e}{1 - 2\bar p_a/e} .
\end{align*}
Since $B_\infty(a) > a$, the last bound holds at every level, hence for $\cG(1,a)$. Moreover,
\begin{align*}
    B_\infty(a) - a = \frac{\sqrt{a^2 + 4\tau_0} - a}{2} = \frac{2\tau_0}{\sqrt{a^2 + 4\tau_0} + a} \le \frac{\tau_0}{a},
\end{align*}
so that
\begin{align*}
    a\left(\cG(1,a) - a\right)
    &\le a\left(B_\infty(a) - a\right) + a\,B_\infty(a)\,\frac{4\bar p_a/e}{1 - 2\bar p_a/e}
    \\
    &\le \tau_0 + a\,B_\infty(a)\,\frac{4\bar p_a/e}{1 - 2\bar p_a/e} .
\end{align*}
As $a \to \infty$, $a^2\,\bar p_a \to 0$, so the last term tends to zero, and the upper limit of $a(\cG(1,a) - a)$ is at most $\tau_0$.

\medskip
We now prove the lower bound. Let $\rho(\tau) \coloneqq \log(e^{1/(2\tau)} - 1)$ for $\tau > 0$. Then
\begin{align*}
    \rho'(\tau) = -\frac{1}{2\tau^2}\cdot\frac{e^{1/(2\tau)}}{e^{1/(2\tau)} - 1} = -F\Bigl(\frac{1}{2\tau}\Bigr),
    \qquad
    F(r) \coloneqq \frac{2r^2}{1 - e^{-r}} .
\end{align*}
The function $F$ increases on $(0,\infty)$, as the product of $r$ and $r/(1 - e^{-r})$, which are positive and increasing, the derivative of the second having the sign of $e^r - 1 - r$. Since $1/(2\tau)$ decreases in $\tau$, $\rho'$ increases, and $\rho$ is convex. At $\tau_0$, we have $1/(2\tau_0) = G_0$ and $e^{G_0} - 1 = e/2$, so that
\begin{align*}
    \rho(\tau_0) = \log\frac{e}{2} = h_0,
    \qquad
    -\rho'(\tau_0) = F(G_0) = 2G_0^2\Bigl(1 + \frac{2}{e}\Bigr) \eqqcolon L_\star \approx 2.557 .
\end{align*}
A convex function lies above its tangents, so $\rho(\tau) \ge h_0 - L_\star(\tau - \tau_0)$, that is,
\begin{align}
    \tau L_\star + \rho(\tau) \ge y_\star \coloneqq \tau_0 L_\star + h_0
    \qquad \text{for every } \tau > 0 .
    \label{eq:saddle}
\end{align}
Fix $\varepsilon_0 \in (0, y_\star)$ and let $y \coloneqq y_\star - \varepsilon_0$. For $a \ge 1$, let
\begin{align*}
    \ell \coloneqq L_\star a^2,
    \qquad
    p \coloneqq e^{-\ell} \le \frac12,
    \qquad
    \nu \coloneqq \frac{a(\ell + y)}{1-p},
\end{align*}
so that the upper atom of $Y_{p,\nu}$ is $a(\ell + y)$. We show that $Y_{p,\nu}$ belongs to $\SG(1,a)$ when $a$ is large. At $\lambda = \theta/a$, with $0 < \theta < 1$, let $u \coloneqq \theta(\ell + y)$ and $w \coloneqq pu/(1-p)$, as in the proof of Theorem~\ref{thm:SG}. Since $p\,u = (1-p)\,w$, $e^{-w} - 1 + w \le w^2/2$ and $1 - p \ge 1/2$,
\begin{align}
    e^{g_p(\lambda\nu)} - 1
    &= p\left(e^u - 1 - u\right) + (1-p)\left(e^{-w} - 1 + w\right)
    \notag
    \\
    &\le p\left(e^u - 1 - u\right) + \frac{p^2u^2}{2(1-p)}
    \le p\left(e^u - 1 - u\right) + p^2u^2 .
    \label{eq:mgfminus}
\end{align}
Let $E \coloneqq \theta^2/(2a^2(1-\theta))$ be the envelope at $\lambda$. Since $E \le e^E - 1$, the bound \eqref{eq:feasG} holds at $\lambda$ as soon as the right side of \eqref{eq:mgfminus} is at most $E$, or at most $e^E - 1$.

For $0 < \theta \le 1/2$, the bound $e^u - 1 - u \le u^2e^u/2$ and $p\,e^u \le p\,e^{(\ell + y)/2} = e^{(y - \ell)/2}$ give
\begin{align*}
    e^{g_p(\lambda\nu)} - 1 \le \frac{\theta^2(\ell + y)^2}{2}\left(e^{(y-\ell)/2} + 2p^2\right) .
\end{align*}
As $a \to \infty$, $a^2(\ell + y)^2(e^{(y-\ell)/2} + 2e^{-2\ell}) \to 0$, since $\ell = L_\star a^2$. For large $a$, the right side is therefore at most $\theta^2/(2a^2) \le E$.

For $1/2 \le \theta < 1$, write $1 - \theta = \tau/a^2$, with $\tau \in (0, a^2/2]$. Then $E = \theta^2/(2\tau)$, and, since $\theta \ge 1/2$ and $\theta^2 \ge 1 - 2\tau/a^2$,
\begin{align*}
    E \ge \frac{1}{8\tau},
    \qquad
    E \ge \frac{1}{2\tau} - \frac{1}{a^2} .
\end{align*}
Since $(1-\theta)\,\ell = \tau L_\star$ and $\theta y \le y$,
\begin{align*}
    p\,e^u = \exp\left(-(1-\theta)\,\ell + \theta y\right) \le e^{y - \tau L_\star} ,
\end{align*}
and $\eta \coloneqq p^2(\ell + y)^2$ bounds the last term of \eqref{eq:mgfminus}. Let $T \ge 1$ be such that $16\,\tau\,e^{y_\star - \tau L_\star} \le 1$ for every $\tau \ge T$, and let $a$ be large enough that $\eta \le 1/(8a^2)$. If $\tau \ge T$, then, since $\tau \le a^2/2$,
\begin{align*}
    e^{g_p(\lambda\nu)} - 1 \le e^{y - \tau L_\star} + \eta \le \frac{1}{16\tau} + \frac{1}{8a^2} \le \frac{1}{8\tau} \le E .
\end{align*}
If $\tau < T$, the inequality \eqref{eq:saddle} gives $e^{y - \tau L_\star} = e^{-\varepsilon_0}\,e^{y_\star - \tau L_\star} \le e^{-\varepsilon_0}(e^{1/(2\tau)} - 1)$, while
\begin{align*}
    e^E - 1 \ge e^{-1/a^2}e^{1/(2\tau)} - 1 = e^{-1/a^2}\left(e^{1/(2\tau)} - 1\right) - \left(1 - e^{-1/a^2}\right) .
\end{align*}
With $e^{1/(2\tau)} - 1 \ge 1/(2\tau) > 1/(2T)$ and $1 - e^{-1/a^2} \le 1/a^2$, for $a^2 \ge 1/\varepsilon_0$,
\begin{align*}
    \left(e^E - 1\right) - \left(e^{g_p(\lambda\nu)} - 1\right)
    &\ge \left(e^{-1/a^2} - e^{-\varepsilon_0}\right)\left(e^{1/(2\tau)} - 1\right) - \left(1 - e^{-1/a^2}\right) - \eta
    \\
    &\ge \frac{e^{-1/a^2} - e^{-\varepsilon_0}}{2T} - \frac{1}{a^2} - \frac{1}{8a^2},
\end{align*}
which is positive for large $a$. So, for large $a$, \eqref{eq:feasG} holds at every frequency, $Y_{p,\nu} \in \SG(1,a)$, and, since $y - h_0 = \tau_0L_\star - \varepsilon_0$,
\begin{align*}
    \cG(1,a) \ge R(p) \ge \frac{a(\ell + y)}{\ell + h_0} = a + \frac{a\left(\tau_0 L_\star - \varepsilon_0\right)}{L_\star a^2 + h_0} .
\end{align*}
The lower limit of $a(\cG(1,a) - a)$ is therefore at least $\tau_0 - \varepsilon_0/L_\star$, for every $\varepsilon_0 \in (0, y_\star)$, which proves (ii).
\end{proof}

\begin{proof}[Proof of Proposition~\ref{prop:ends}(i)]
Let $x \coloneqq 1 - 2p \in [0,1)$ and $v \coloneqq p(1-p) = (1-x^2)/4$. With $A$ and $B$ as in the proof of Theorem~\ref{thm:SG}, we have $(1-p)/\sqrt v = \sqrt{(1-p)/p} = A(x)$ and $\eL(p) = B(x)$, so that
\begin{align}
    R(p) = \frac{A(x)}{B(x)}\,\sqrt{v}\,\nuG(p,a) .
    \label{eq:Rsplit}
\end{align}
The functions $A(x) = (1+x)(1-x^2)^{-1/2}$ and $B(x) = 1 - \log(1-x)$ are analytic on $(-1,1)$, with
\begin{align*}
    A(x) = 1 + x + \frac{x^2}{2} + \frac{x^3}{2} + O(x^4),
    \qquad
    B(x) = 1 + x + \frac{x^2}{2} + \frac{x^3}{3} + O(x^4),
\end{align*}
so that
\begin{align}
    \frac{A(x)}{B(x)} = 1 + \frac{x^3}{6} + O(x^4)
    \qquad (x \to 0) .
    \label{eq:ABexp}
\end{align}
Moreover, $A/B$ is nondecreasing on $[0,1)$: since $A' = A/(1-x^2)$ and $B' = 1/(1-x)$,
\begin{align*}
    \Bigl(\frac{A}{B}\Bigr)' = \frac{A\left(B - 1 - x\right)}{(1-x^2)\,B^2} \ge 0,
    \qquad\text{because}\qquad
    B(x) - 1 - x = -\log(1-x) - x \ge 0 .
\end{align*}

We first prove the lower bound. Let $x_a \coloneqq 3a/\sqrt{4 + 9a^2}$, so that $2x_a/\sqrt{1 - x_a^2} = 3a$, and $x_a \le 1/\sqrt3$ when $a^2 \le 2/9$. By the proof of Theorem~\ref{thm:SG}, the level $(1 - x_a)/2$ then satisfies $R \ge A(x_a)/B(x_a)$. Since $x_a = 3a/2 + O(a^3)$, the expansion \eqref{eq:ABexp} gives
\begin{align*}
    \cG(1,a) \ge \frac{A(x_a)}{B(x_a)} = 1 + \frac{x_a^3}{6} + O(a^4) = 1 + \frac{9}{16}\,a^3 + O(a^4) .
\end{align*}

\medskip
We now prove the upper bound, for $a \le 1/2$. We treat first the levels far from $1/2$, then the levels with $x \le 3a/2$, and finally the levels in between.

Let first $\ell \ge 8$. Lemma~\ref{lem:spreadG}(b) gives
\begin{align*}
    R(p) \le \frac{a\ell + \sqrt{2\ell}}{\ell + h_0} \le a + \sqrt{\frac{2}{\ell}} \le \frac12 + \frac12 = 1 .
\end{align*}
Let now $\ell < 8$ and $p < 1/2$, so that $0 < x < x_1 \coloneqq 1 - 2e^{-8}$. Let $\nu_{\mathrm{KS}} \coloneqq \sqrt{2\beta_p/(1-2p)}$ and $\bar\lambda(p) = \sqrt{2\beta_p(1-2p)}$, as in Proposition~\ref{prop:numax}. By Lemma~\ref{lem:unimodal}, $g_p(2\beta_p) = (1-2p)\,\beta_p$, so that $2/g_p(2\beta_p) = \nu_{\mathrm{KS}}^2/\beta_p^2$, and Lemma~\ref{lem:fixedp} at $s = 2\beta_p$ gives
\begin{align*}
    \nuG(p,a) \le a\beta_p + \sqrt{a^2\beta_p^2 + \nu_{\mathrm{KS}}^2}
    = \nu_{\mathrm{KS}}\left(r + \sqrt{1 + r^2}\right)
    \le \nu_{\mathrm{KS}}\,e^{r},
    \qquad
    r \coloneqq \frac{a\beta_p}{\nu_{\mathrm{KS}}} = \frac{a\,\bar\lambda(p)}{2} .
\end{align*}
The last inequality holds because $e^r \ge r$ and $(e^r - r)^2 - (1 + r^2) = 2e^r(\sinh r - r) \ge 0$. Let $R_{\mathrm{KS}}(p) \coloneqq (1-p)\,\nu_{\mathrm{KS}}/\eL(p)$, which is less than one by Lemma~\ref{lem:elementary}(a). Then
\begin{align*}
    R(p) \le R_{\mathrm{KS}}(p)\,e^{a\bar\lambda(p)/2},
    \qquad\text{so that}\qquad
    R(p) \le 1
    \quad\text{as soon as}\quad
    a \le m(x) \coloneqq \frac{2\log\left(1/R_{\mathrm{KS}}(p)\right)}{\bar\lambda(p)} .
\end{align*}
The function $m$ is continuous and positive on $(0,1)$. By the proof of Lemma~\ref{lem:elementary}, $R_{\mathrm{KS}}(p)^2$ is the ratio of the two sides of \eqref{eq:seriesineq}, whose expansions are $1 + 2x + 4x^2/3 + O(x^3)$ and $1 + 2x + 2x^2 + O(x^3)$. Hence, as $x \to 0$,
\begin{align*}
    \log\frac{1}{R_{\mathrm{KS}}(p)} = \frac{x^2}{3} + O(x^3),
    \qquad
    \bar\lambda(p)^2 = 4x\operatorname{artanh} x = 4x^2\left(1 + O(x^2)\right),
    \qquad
    \frac{m(x)}{x} \longrightarrow \frac13 .
\end{align*}
So $m(x)/x$ is bounded from below on $(0, x_1]$ by some $\mu > 0$. Let $K \coloneqq 2 \vee 1/\mu$. For $x \in [Ka, x_1]$, we get $a \le \mu x \le m(x)$, hence $R(p) \le 1$. Together with the case $\ell \ge 8$, this gives
\begin{align}
    R(p) \le 1
    \qquad \text{at every level with } 1 - 2p \ge Ka .
    \label{eq:farlevels}
\end{align}
Since $\cG(1,a) > 1$, every level $p_a$ at which the supremum of $R$ is reached satisfies $1 - 2p_a < Ka$.

Let now $x \le 3a/2$. Lemma~\ref{lem:spreadG}(a) gives $\sqrt{v}\,\nuG(p,a) \le 1$. By \eqref{eq:Rsplit}, the monotonicity of $A/B$ and \eqref{eq:ABexp},
\begin{align*}
    R(p) \le \frac{A(x)}{B(x)} \le \frac{A(3a/2)}{B(3a/2)} = 1 + \frac{9}{16}\,a^3 + O(a^4) .
\end{align*}

\medskip
Let finally $3a/2 < x < Ka$. In what follows, $C$ denotes a constant that depends only on $K$, whose value may change from line to line, and $a$ is small enough that $Ka \le 1/8$. We first expand $g_p$ at the origin. By \eqref{eq:cumulants} and $\diff q_s/\diff s = q_s(1 - q_s)$, every derivative of $g_p$ of order at least two is a polynomial in $q_s$; with $w_s \coloneqq q_s(1-q_s)$,
\begin{align*}
    g_p^{(4)} = w_s\left(1 - 6w_s\right),
    \qquad
    g_p^{(5)} = w_s\left(1 - 12w_s\right)\left(1 - 2q_s\right),
\end{align*}
and, since $q_s \in (0,1)$, the derivative $g_p^{(6)}$ is bounded by a numerical constant. At $s = 0$, we have $q_0 = p$, $w_0 = v$ and $1 - 2q_0 = x$, so Taylor's formula gives, for every $s > 0$,
\begin{align*}
    \Bigl|\,g_p(s) - \frac{v s^2}{2} - \frac{v x s^3}{6} - \frac{v(1-6v)s^4}{24} - \frac{v x(1 - 12v)s^5}{120}\Bigr| \le C s^6 .
\end{align*}
We divide by $v s^2/2$. Since $v \ge 3/16$ for $x \le 1/2$, $1 - 6v = -1/2 + 3x^2/2 \ge -1/2$, $1 - 12v \ge -2$, and both $x s^3$ and $s^4$ are at most $(x^2 + s^2)^2$,
\begin{align}
    \frac{2g_p(s)}{v s^2} \ge 1 + \frac{x s}{3} - \frac{s^2}{24} - C\left(x^2 + s^2\right)^2 .
    \label{eq:Qexp}
\end{align}
Let $z \coloneqq x - 3a/2 \in (0, Ka)$, and take $s \coloneqq 4z$ in Lemma~\ref{lem:fixedp}. As $\sqrt v \le 1/2$,
\begin{align*}
    \sqrt{v}\,\nuG(p,a)
    \le \frac{\sqrt{v}\,s}{2}\biggl(a + \sqrt{a^2 + \frac{2}{g_p(s)}}\,\biggr)
    \le \frac{as}{4} + \sqrt{\frac{a^2s^2}{16} + \frac{v s^2}{2g_p(s)}} .
\end{align*}
Here $x$ and $s$ are at most $4Ka \le 1/2$, so $|xs/3 - s^2/24| \le Ca^2$ and $(x^2 + s^2)^2 \le Ca^4$. By \eqref{eq:Qexp}, and by $1/(1+y) \le 1 - y + 2y^2$ for $|y| \le 1/2$, for $a$ small enough,
\begin{align*}
    \frac{v s^2}{2g_p(s)} \le 1 - \frac{x s}{3} + \frac{s^2}{24} + Ca^4 .
\end{align*}
With $a^2s^2 \le Ca^4$ and $\sqrt{1+y} \le 1 + y/2$, this gives
\begin{align*}
    \sqrt{v}\,\nuG(p,a) \le 1 + \frac{as}{4} - \frac{x s}{6} + \frac{s^2}{48} + Ca^4 = 1 - \frac{z^2}{3} + Ca^4,
\end{align*}
since $as/4 - xs/6 = -(s/12)(2x - 3a) = -2z^2/3$ and $s^2/48 = z^2/3$. By \eqref{eq:Rsplit} and \eqref{eq:ABexp},
\begin{align*}
    R(p) \le \Bigl(1 + \frac{x^3}{6} + Ca^4\Bigr)\Bigl(1 - \frac{z^2}{3} + Ca^4\Bigr) \le 1 + \frac{x^3}{6} - \frac{z^2}{3} + Ca^4 .
\end{align*}
With $x = 3a/2 + z$,
\begin{align*}
    \frac{x^3}{6} = \frac{9a^3}{16} + \frac{9a^2z}{8} + \frac{3az^2}{4} + \frac{z^3}{6},
\end{align*}
and, since $3a/4 + z/6 \le 1/12$,
\begin{align*}
    \frac{x^3}{6} - \frac{z^2}{3}
    = \frac{9a^3}{16} + \frac{9a^2z}{8} - z^2\Bigl(\frac13 - \frac{3a}{4} - \frac{z}{6}\Bigr)
    \le \frac{9a^3}{16} + \frac{9a^2z}{8} - \frac{z^2}{4}
    \le \frac{9a^3}{16} + \frac{81a^4}{64},
\end{align*}
the last step being the maximum over $z$, reached at $z = 9a^2/4$. So $R(p) \le 1 + (9/16)\,a^3 + Ca^4$ on this range. With the range $x \le 3a/2$ and \eqref{eq:farlevels}, this bound holds at every level, so that $\cG(1,a) \le 1 + (9/16)\,a^3 + Ca^4$, and, with the lower bound, this proves (i).
\end{proof}

\section{Proofs for sums of independent variables}
\label{app:sums}

\begin{proof}[Proof of Lemma~\ref{lem:sumclass}]
Let $A \coloneqq \max_i |w_i|\alpha_i$ and $|\lambda| < 1/A$. For every $i$ with $w_i \ne 0$, $|w_i\lambda|\,\alpha_i \le A|\lambda| < 1$, so the bound \eqref{eq:SG} for $X_i$ applies at the frequency $w_i\lambda$, and
\begin{align*}
    \log\E\left[e^{\lambda w_i X_i}\right] \le \frac{w_i^2\sigma_i^2\lambda^2}{2\left(1 - \alpha_i|w_i\lambda|\right)} \le \frac{w_i^2\sigma_i^2\lambda^2}{2\left(1 - A|\lambda|\right)} .
\end{align*}
By independence, the moment generating function of the sum is the product of those of the summands, and adding the bounds gives the claim. For the sub-exponential classes, the same computation with $|\lambda| \le 1/A$ and the envelope $w_i^2\sigma_i^2\lambda^2/2$ applies.
\end{proof}

\begin{proof}[Proof of Proposition~\ref{prop:sums}]
By Strassen's theorem, there are, for each $i$, a pair $(X_i', Z_i')$ with $X_i' \sim X_i$, $Z_i' \sim c_i\Lap$ and $\E[Z_i' \mid X_i'] = X_i'$, and we take the $n$ pairs independent. Then $\sum_i X_i' x_i$ has the law of $\sum_i X_i x_i$, the variable $\sum_i Z_i' x_i$ has the law of $\sum_i c_i\Lap_i x_i$, and, by independence of the pairs,
\begin{align*}
    \E\Bigl[\sum_{i=1}^n Z_i'\,x_i \Bigm| X_1', \dots, X_n'\Bigr] = \sum_{i=1}^n \E\left[Z_i' \mid X_i'\right] x_i = \sum_{i=1}^n X_i'\,x_i .
\end{align*}
For every convex $F \colon \R^d \to \R$, the conditional Jensen inequality gives
\begin{align*}
    \E\Bigl[F\Bigl(\sum_{i=1}^n X_i'\,x_i\Bigr)\Bigr]
    = \E\Bigl[F\Bigl(\E\Bigl[\sum_{i=1}^n Z_i'\,x_i \Bigm| X_1', \dots, X_n'\Bigr]\Bigr)\Bigr]
    \le \E\Bigl[F\Bigl(\sum_{i=1}^n Z_i'\,x_i\Bigr)\Bigr],
\end{align*}
which is the comparison.
\end{proof}

\begin{proof}[Proof of Corollary~\ref{cor:concentration}]
The proof moves the exponential moment of $f(X)$ to the Laplace vector by the comparison, bounds it there by the concentration inequality of \citet{bobkov1997poincare}, and applies Markov's inequality.

By Proposition~\ref{prop:sums}, applied with the coordinate vectors of $\R^n$, the vector $X$ is dominated in the convex order of $\R^n$ by $c\Lap^{(n)}$. For $\lambda \ge 0$, the function $e^{\lambda f}$ is convex, as the composition of the convex function $f$ with the convex nondecreasing function $u \mapsto e^{\lambda u}$. The comparison therefore gives
\begin{align}
    \E\bigl[e^{\lambda f(X)}\bigr] \le \E\bigl[e^{\lambda f(c\Lap^{(n)})}\bigr] .
    \label{eq:stepone}
\end{align}

The right side is controlled by the concentration of the Laplace vector. Let $g$ be a Lipschitz function on $\R^n$ whose partial derivatives satisfy, almost everywhere,
\begin{align*}
    |\partial_ig| \le 1 \quad\text{for every } i,
    \qquad
    \sum_{i=1}^n(\partial_ig)^2 \le a^2 .
\end{align*}
Then, for $0 \le \theta \le 1/2$,
\begin{align}
    \E\bigl[e^{\theta(g - \E g)(\Lap^{(n)})}\bigr] \le e^{4\theta^2a^2} .
    \label{eq:BL}
\end{align}
This is the bound of \citet[Section 2]{bobkov1997poincare}, and we recall its proof. The standard Laplace law satisfies the modified logarithmic Sobolev inequality
\begin{align*}
    \mathrm{Ent}\bigl(e^{u}\bigr) \le \frac{2}{1-\kappa}\,\E\bigl[u'^2e^{u}\bigr]
    \qquad\text{for every Lipschitz } u \text{ with } |u'| \le \kappa < 1,
\end{align*}
where $\mathrm{Ent}(Y) = \E[Y\log Y] - \E[Y]\log\E[Y]$. For independent coordinates, the entropy of a function is at most the sum, over the coordinates, of the expected entropies in each coordinate. Taking $u = \theta g$ in each coordinate with $\kappa = 1/2$, which is allowed because every partial derivative of $\theta g$ is at most $1/2$ in absolute value, gives
\begin{align*}
    \mathrm{Ent}\bigl(e^{\theta g}\bigr)
    \le 4\theta^2\,\E\Bigl[\sum_{i=1}^n(\partial_ig)^2\,e^{\theta g}\Bigr]
    \le 4\theta^2a^2\,\E\bigl[e^{\theta g}\bigr],
\end{align*}
so that the derivative of $\theta \mapsto \theta^{-1}\log\E[e^{\theta g}]$ satisfies
\begin{align*}
    \frac{\mathrm{Ent}\bigl(e^{\theta g}\bigr)}{\theta^2\,\E\bigl[e^{\theta g}\bigr]} \le 4a^2,
\end{align*}
and the limit of this function at $\theta = 0$ is $\E[g]$. Integrating from $0$ to $\theta$ gives $\theta^{-1}\log\E[e^{\theta g}] \le \E[g] + 4a^2\theta$, which is \eqref{eq:BL}.

To use \eqref{eq:BL}, let $h(z) = f(cz)$. This function is $c\varrho_2$-Lipschitz for the Euclidean norm and $c\varrho_1$-Lipschitz for the $\ell^1$ norm. The function $g = h/(c\varrho_1)$ then satisfies the two conditions above with $a = \varrho_2/\varrho_1$, and we take $\theta = c\varrho_1\lambda$, which is at most $1/2$ when $\lambda \le 1/(2c\varrho_1)$. Since $4\theta^2a^2 = 4c^2\varrho_2^2\lambda^2$, we obtain
\begin{align}
    \E\bigl[e^{\lambda(h - \E h)(\Lap^{(n)})}\bigr] \le e^{4c^2\varrho_2^2\lambda^2},
    \qquad 0 \le \lambda \le \frac{1}{2c\varrho_1} .
    \label{eq:steptwo}
\end{align}
The function $h$ need not be bounded. We apply \eqref{eq:BL} to the truncations $\max\{-m, \min\{h, m\}\}$, which have the same Lipschitz constants, and let $m$ tend to infinity. Dominated convergence applies to the moment generating function and to the mean, because
\begin{align*}
    |h(z)| \le |h(0)| + c\varrho_1\sum_{i=1}^n|z_i|,
    \qquad
    \E\bigl[e^{\mu|\Lap|}\bigr] = \frac{1}{1-\mu} < \infty
    \quad\text{for } \mu = \lambda c\varrho_1 \le \frac12 .
\end{align*}

Since $h(\Lap^{(n)}) = f(c\Lap^{(n)})$, \eqref{eq:stepone} and \eqref{eq:steptwo} give the first inequality, \eqref{eq:mgfconc}. Markov's inequality then gives, for $0 \le \lambda \le 1/(2c\varrho_1)$,
\begin{align*}
    \Prob\Bigl(f(X) \ge \E\bigl[f(c\Lap^{(n)})\bigr] + t\Bigr) \le \exp\bigl(-\lambda t + 4c^2\varrho_2^2\lambda^2\bigr) .
\end{align*}
The exponent is minimal at $\lambda = t/(8c^2\varrho_2^2)$, which lies in the window if and only if $t \le t_0 = 4c\varrho_2^2/\varrho_1$, and it then equals $-t^2/(16c^2\varrho_2^2)$. Beyond $t_0$, we take the endpoint $\lambda = 1/(2c\varrho_1)$, and the exponent satisfies
\begin{align*}
    -\frac{t}{2c\varrho_1} + \frac{\varrho_2^2}{\varrho_1^2} \le -\frac{t}{4c\varrho_1},
\end{align*}
since $\varrho_2^2/\varrho_1^2 < t/(4c\varrho_1)$ when $t > t_0$. In both cases, the bound is the one of \eqref{eq:concentration}, because $t^2/(16c^2\varrho_2^2) \le t/(4c\varrho_1)$ if and only if $t \le t_0$.

For the sub-Gaussian bound, the hypothesis gives $X_i \lecx \sigma\sqrt{\pi/2}\,G_i$ for every $i$ \citep[Theorem 1]{zhang2026sharpmgf}. The comparison in the convex order of $\R^n$ follows as above, since the proof of Proposition~\ref{prop:sums} uses only the one-dimensional comparisons, and it gives $X \lecx \sigma\sqrt{\pi/2}\,G$. Gaussian concentration \citep[Theorem 5.5]{boucheron2013concentration} replaces \eqref{eq:BL}: for every $\ell$-Lipschitz function $h$ and every $\lambda$, $\E[e^{\lambda(h - \E h)(G)}] \le e^{\lambda^2\ell^2/2}$. With $h(z) = f(\sigma\sqrt{\pi/2}\,z)$, for which $\ell^2 = \pi\sigma^2\varrho_2^2/2$, this gives, for $\lambda \ge 0$,
\begin{align*}
    \E\bigl[e^{\lambda f(X)}\bigr] \le \exp\Bigl(\lambda\,\E\bigl[f(\sigma\sqrt{\pi/2}\,G)\bigr] + \frac{\pi\sigma^2\varrho_2^2\lambda^2}{4}\Bigr),
\end{align*}
and Markov's inequality with $\lambda = 2t/(\pi\sigma^2\varrho_2^2)$ gives the bound $\exp(-t^2/(\pi\sigma^2\varrho_2^2))$.
\end{proof}

\begin{proof}[Proof of Proposition~\ref{prop:averages}]
By Proposition~\ref{prop:sums} with $d = 1$, the average is dominated by $b\bar\Lap_n$, where $\bar\Lap_n \coloneqq n^{-1}\sum_i \Lap_i$. The variable $\bar\Lap_n$ is symmetric, and its density is log-concave, as a convolution of log-concave densities: for log-concave densities $f$ and $g$, the function $(x, y) \mapsto f(x - y)\,g(y)$ is log-concave on $\R^2$, and so is its marginal in $x$ \citep[Theorem 6]{prekopa1973logarithmic}.

\medskip
We first show that a symmetric variable with a log-concave density is dominated by the multiple of $\Lap$ with the same first absolute moment. Let $Z$ be symmetric with a log-concave density $f_Z$, and $\mu_Z \coloneqq \E[|Z|]$. The function $t \mapsto \Prob(Z > t)$ is log-concave, as the marginal in $t$ of the log-concave function $(x,t) \mapsto \one(x > t)\,f_Z(x)$ \citep[Theorem 6]{prekopa1973logarithmic}, and so is its integral
\begin{align*}
    \SL_Z(t) = \int_t^\infty \Prob(Z > s)\,\diff s
\end{align*}
\citep[Theorem 3]{bagnoli2005logconcave}. By symmetry, $\SL_Z(0) = \E[Z_+] = \mu_Z/2$, and the derivative of $\SL_Z$ at $0$ is $-\Prob(Z > 0) = -1/2$. The concave function $\log\SL_Z$ lies below its tangent at $0$, whose slope is $(-1/2)/(\mu_Z/2) = -1/\mu_Z$:
\begin{align*}
    \log\SL_Z(t) \le \log\frac{\mu_Z}{2} - \frac{t}{\mu_Z},
    \qquad\text{that is}\qquad
    \SL_Z(t) \le \frac{\mu_Z}{2}\,e^{-t/\mu_Z} = \SL_{\mu_Z\Lap}(t)
    \qquad (t \ge 0) .
\end{align*}
For $t < 0$, the identity $\SL_V(t) = -t + \SL_V(-t)$, valid for $V = Z$ and $V = \mu_Z\Lap$, both symmetric, transfers the bound. By \eqref{eq:slcrit}, $Z \lecx \mu_Z\Lap$.

\medskip
We now compute the constant. By \eqref{eq:mixture} with independent pairs $(W_i, G_i)$, the sum $\sum_i \sqrt{2W_i}\,G_i$ has, given the $W_i$, a centered Gaussian law with variance $2T$, where $T \coloneqq \sum_i W_i$ has the Gamma law with shape $n$ and scale one. With $\E[|G|] = \sqrt{2/\pi}$ and $\E[\sqrt T] = \Gamma(n + 1/2)/\Gamma(n)$,
\begin{align*}
    \E\bigl[|b\bar\Lap_n|\bigr] = \frac{b}{n}\,\sqrt{2}\;\E\bigl[\sqrt{T}\bigr]\,\E[|G|] = \frac{b}{n}\,\sqrt2\,\frac{\Gamma(n + 1/2)}{\Gamma(n)}\,\sqrt{\frac{2}{\pi}} = \frac{2b\,\Gamma(n + 1/2)}{\sqrt\pi\,\Gamma(n+1)} = b\,\rho_n,
\end{align*}
and the binomial form follows from $\Gamma(n + 1/2) = \sqrt{\pi}\,(2n)!/(4^n n!)$. The comparison $Z \lecx \mu_Z\Lap$ proved above, applied to $Z = b\bar\Lap_n$, for which $\mu_Z = b\rho_n$, gives $b\bar\Lap_n \lecx b\rho_n\Lap$, so the average is dominated by $b\rho_n\Lap$. The function $|x|$ shows that $b\,\rho_n$ cannot be lowered for $X_i = b\Lap_i$.

\medskip
Finally, we bound $\rho_n$. By Jensen's inequality, $\E[\sqrt T] \le \sqrt{\E[T]} = \sqrt n$, which gives $\rho_n \le 2/\sqrt{\pi n}$. For the lower bound, $\rho_1 = 1$, and, since $\Gamma(n + 3/2) = (n + 1/2)\Gamma(n + 1/2)$ and $(n + 1/2)^2 \ge n(n+1)$,
\begin{align*}
    \frac{\sqrt{n+1}\,\rho_{n+1}}{\sqrt{n}\,\rho_n} = \frac{\sqrt{n+1}}{\sqrt n}\cdot\frac{n + 1/2}{n+1} = \frac{n + 1/2}{\sqrt{n(n+1)}} \ge 1,
\end{align*}
so $\sqrt n\,\rho_n \ge 1$ for every $n$.
\end{proof}

\begin{proof}[Proof of Proposition~\ref{prop:clt}]
By Lemma~\ref{lem:classes}(iii), we may take $\sigma = 1$. Let $c_1 \coloneqq \cG(1,a)$, let $X_1, \dots, X_n$ be independent elements of $\SG(1,a)$, and let
\begin{align*}
    W \coloneqq \frac{1}{\sqrt n}\sum_{i=1}^n X_i,
    \qquad
    v \coloneqq \E\left[W^2\right] \le 1,
\end{align*}
each $X_i$ having variance at most one, by the expansion of \eqref{eq:SG} at $\lambda = 0$. By Theorem~\ref{thm:SG}, $\E[|X_i|^3] \le \Gamma(4)\,c_1^3 = 6c_1^3$. Let $v > 0$. The bound of \citet[Theorem 1.1]{goldstein2010bounds} on the $L^1$ distance between the distribution function of the sum $W/\sqrt v$, of variance one, and $\Phi$ gives, after the change of variable $u \mapsto u/\sqrt v$,
\begin{align*}
    \int_{\R}\Bigl|\Prob\left(W \le u\right) - \Phi\Bigl(\frac{u}{\sqrt v}\Bigr)\Bigr|\,\diff u
    \le \sqrt{v}\;\frac{\sum_{i=1}^n \E\left[|X_i|^3\right]}{(nv)^{3/2}}
    \le \frac{6c_1^3}{v\sqrt n} .
\end{align*}
Since $\SL_W(t) = \int_t^\infty \Prob(W > u)\,\diff u$, and likewise for $\sqrt v\,G$,
\begin{align}
    \left|\SL_W(t) - \SL_{\sqrt v\,G}(t)\right| \le \frac{6c_1^3}{v\sqrt n}
    \qquad \text{for every } t \in \R .
    \label{eq:cltsl}
\end{align}

We first prove the lower bound. For $X_i = \varepsilon_i$, independent Rademacher variables, $v = 1$ and $W$ is symmetric, so $\E[|W|] = 2\,\SL_W(0)$, and \eqref{eq:cltsl} at $t = 0$ gives
\begin{align*}
    \E\left[|W|\right] \ge \E\left[|G|\right] - \frac{12c_1^3}{\sqrt n} = \sqrt{\frac{2}{\pi}} - \frac{12c_1^3}{\sqrt n} .
\end{align*}
If $W \lecx c\Lap$, the function $|x|$ gives $\E[|W|] \le c$, so the lower limit of $c_n(1,a)$ is at least $\sqrt{2/\pi}$.

\medskip
We now prove the upper bound. Fix $\delta > 0$, and let
\begin{align*}
    c \coloneqq \sqrt{\frac{2}{\pi}} + \delta,
    \qquad
    T \coloneqq \frac{4}{c},
    \qquad
    v_0 \coloneqq c^2e^{-2T/c},
    \qquad
    \gamma \coloneqq \min_{0 \le t \le T}\left(\SL_{c\Lap}(t) - \SL_{\sqrt{2/\pi}\,\Lap}(t)\right) .
\end{align*}
By \eqref{eq:laplace}, $\SL_{c\Lap}(t) = (c/2)\,e^{-t/c}$ increases strictly in $c$ for every $t \ge 0$, so $\gamma > 0$. Let $n$ be large enough that $6c_1^3/(v_0\sqrt n) \le \gamma$ and $\sqrt n \ge 4a/c$. We show that $\SL_W(t) \le \SL_{c\Lap}(t)$ for every $t \ge 0$.

Let $0 \le t \le T$. If $v \ge v_0$, then \eqref{eq:cltsl}, the comparisons $\sqrt v\,G \lecx G \lecx \sqrt{2/\pi}\,\Lap$, the first for a contraction and the second from the proof of Proposition~\ref{prop:averages}, and the definition of $\gamma$ give
\begin{align*}
    \SL_W(t) \le \SL_{\sqrt v\,G}(t) + \gamma \le \SL_{\sqrt{2/\pi}\,\Lap}(t) + \gamma \le \SL_{c\Lap}(t) .
\end{align*}
If $v < v_0$, then, since $\SL_W$ is nonincreasing and $W$ is centered,
\begin{align*}
    \SL_W(t) \le \SL_W(0) = \frac{\E\left[|W|\right]}{2} \le \frac{\sqrt v}{2} \le \frac{c}{2}\,e^{-T/c} = \SL_{c\Lap}(T) \le \SL_{c\Lap}(t) .
\end{align*}

Let now $t \ge T$. By Lemma~\ref{lem:sumclass}, $W \in \SG(1, a/\sqrt n)$. At the frequency $\lambda \coloneqq 2/c$, we have $a\lambda/\sqrt n \le 1/2$, so the Chernoff bound gives $\Prob(W > u) \le e^{-\lambda u + \lambda^2}$ for every $u$, and
\begin{align*}
    \SL_W(t) = \int_t^\infty \Prob(W > u)\,\diff u \le \frac{e^{\lambda^2 - \lambda t}}{\lambda} = \frac{c}{2}\,e^{4/c^2 - 2t/c} \le \frac{c}{2}\,e^{-t/c} = \SL_{c\Lap}(t),
\end{align*}
the last inequality because $t \ge 4/c$.

For $t < 0$, the same bounds hold for $-W$, which is the normalized sum of the elements $-X_i$, and the identity $\SL_W(t) = -t + \SL_{-W}(-t)$ transfers them to $W$, $c\Lap$ being symmetric. By \eqref{eq:slcrit}, $W \lecx c\Lap$ for every large $n$, so the upper limit of $c_n(1,a)$ is at most $\sqrt{2/\pi} + \delta$, for every $\delta > 0$. The bound $c_n(\sigma,\alpha) \le \cG(\sigma,\alpha/\sqrt n)$ is Lemma~\ref{lem:sumclass} with the definition of $\cG$. The class $\SE(\sigma^2,\alpha)$ is contained in $\SG(\sigma^2,\alpha)$ and contains $\sigma\varepsilon$, so the two limits also hold for it.
\end{proof}

\section{Proofs for random vectors}
\label{app:vectors}

The proof of Proposition~\ref{prop:dim} uses a regular simplex. Let $n \coloneqq d + 1$, and identify $\R^d$ with the hyperplane $H \coloneqq \{z \in \R^n \colon \sum_i z_i = 0\}$. With $e_1, \dots, e_n$ the canonical basis of $\R^n$ and $\mathbf{1} \coloneqq e_1 + \dots + e_n$, the vectors $u_i \coloneqq \sqrt{n}\,(e_i - \mathbf{1}/n)$, $i = 1, \dots, n$, are the vertices of a regular simplex, with
\begin{align*}
    \sum_{i=1}^n u_i = 0,
    \qquad
    |u_i|^2 = d,
    \qquad
    \langle u_i, u_j\rangle = -1 \ (i \ne j),
    \qquad
    \langle u_i, v\rangle = \sqrt{n}\,v_i \ (v \in H) .
\end{align*}
For a unit vector $v \in H$, the numbers $y_i \coloneqq \langle u_i, v\rangle$ therefore satisfy
\begin{align*}
    \sum_{i=1}^n y_i = 0,
    \qquad
    \sum_{i=1}^n y_i^2 = n\,|v|^2 = n .
\end{align*}
By \citet[Lemma 10]{zhang2026sharpmgf}, for such numbers and every $s \ge 0$,
\begin{align}
    \frac{1}{n}\sum_{i=1}^n e^{s y_i} \le \frac{1}{n}\,e^{s\sqrt d} + \frac{d}{n}\,e^{-s/\sqrt d},
    \label{eq:zhangmax}
\end{align}
with equality when $v = u_j/\sqrt d$ for some $j$. For this $v$, the numbers $y_i$ are $\sqrt d$ once and $-1/\sqrt d$ otherwise, so the projection on $v$ of the uniform law on the vertices is the two-point law $Y_{1/n,\, n/\sqrt d}$.

\begin{proof}[Proof of Proposition~\ref{prop:dim}]
By Lemma~\ref{lem:classes}(iii), applied to every projection, we may take $\sigma = 1$.

(i) The argument is that of \citet[Lemma 13]{zhang2026sharpmgf}. Let $k < d$ and let $X$ belong to the class in $\R^k$. For a unit vector $v = (v_1, v_2)$ of $\R^d$ with $v_1 \in \R^k$ and $v_1 \ne 0$,
\begin{align*}
    \langle v, (X, 0)\rangle = |v_1|\,\Bigl\langle \frac{v_1}{|v_1|}, X\Bigr\rangle ,
\end{align*}
a projection of $X$ multiplied by $|v_1| \in (0,1]$, which stays in the class by Lemma~\ref{lem:classes}(ii); if $v_1 = 0$, the projection is zero. So $(X, 0)$ belongs to the class in $\R^d$. For a convex function $F$ on $\R^k$, the function $\tilde F(x, y) \coloneqq F(x)$ is convex on $\R^d$, and the first $k$ coordinates of $\Lap_d$ form $\Lap_k$, so that
\begin{align*}
    \E[F(X)] = \E\bigl[\tilde F(X, 0)\bigr] \le \E\bigl[\tilde F(c_d\Lap_d)\bigr] = \E\bigl[F(c_d\Lap_k)\bigr] .
\end{align*}
Hence $c_d \ge c_k$. For $d = 1$, the unit vectors are $\pm 1$, so the classes are the one-dimensional ones by Lemma~\ref{lem:classes}(ii), and $\Lap_1 \sim \Lap$ by \eqref{eq:mixture}.

(ii) Let $p \coloneqq 1/n$, let $\nu > 0$ with $Y_{p,\nu}$ in the scalar class, and let $X$ be uniform on $\{r u_1, \dots, r u_n\}$, with $r \coloneqq \nu\sqrt d/n$. For a unit vector $v$ and $\lambda \in \R$, \eqref{eq:zhangmax} applied to $\pm v$, according to the sign of $\lambda$, with $s = r|\lambda|$, gives
\begin{align*}
    \E\left[e^{\lambda\langle v, X\rangle}\right] \le \frac{1}{n}\,e^{|\lambda| r\sqrt d} + \frac{d}{n}\,e^{-|\lambda| r/\sqrt d} = \E\left[e^{|\lambda| Y_{p,\nu}}\right],
\end{align*}
since $r\,n/\sqrt d = \nu$. Both envelopes depend on $|\lambda|$ only, so $X$ belongs to the vector class. Let
\begin{align*}
    F(x) \coloneqq \max_{1 \le i \le n} \langle u_i, x\rangle,
    \qquad\text{so that}\qquad
    F(r u_j) = r\,|u_j|^2 = r\,d
    \quad\text{and}\quad
    \E[F(X)] = r\,d .
\end{align*}
On the benchmark side, write $G_d = Q\tilde G$, with $\tilde G$ a standard Gaussian vector in $\R^n$ and $Q$ the orthogonal projection on $H$, so that $\langle u_i, Q\tilde G\rangle = \sqrt{n}\,(\tilde G_i - \bar G)$, with $\bar G$ the mean of the $\tilde G_i$. With $\E[\sqrt{2W}] = \sqrt{2}\,\Gamma(3/2) = \sqrt{\pi/2}$, the independence of $W$ and $\tilde G$, and $\E[\bar G] = 0$,
\begin{align*}
    \E[F(c\Lap_d)]
    = c\,\E\bigl[\sqrt{2W}\bigr]\,\E\Bigl[\max_i \sqrt{n}\,(\tilde G_i - \bar G)\Bigr]
    = c\,\sqrt{\frac{\pi}{2}}\,\sqrt{n}\,M_n .
\end{align*}
If $X \lecx c\Lap_d$, then $\E[F(X)] \le \E[F(c\Lap_d)]$, and, since $r\,d = \nu\,d^{3/2}/n$,
\begin{align*}
    \frac{\nu\,d^{3/2}}{n} \le c\,\sqrt{\frac{\pi}{2}}\,\sqrt n\,M_n,
    \qquad\text{that is}\qquad
    c \ge \Bigl(\frac{d}{n}\Bigr)^{3/2}\frac{\nu}{\sqrt{\pi/2}\;M_n} .
\end{align*}
The supremum over $\nu$ gives (ii).

(iii) Let $d = 2$ and $a \le 1$. At $p = 1/3$,
\begin{align*}
    \beta_p = \log 2,
    \qquad
    \bar\lambda(1/3) = \sqrt{2\log 2/3} < 1 \le \frac1a,
    \qquad\text{so}\qquad
    \nuE(1/3, a) = \sqrt{6\log 2}
\end{align*}
by Proposition~\ref{prop:numax}. Moreover, integrating by parts with $\varphi'(x) = -x\varphi(x)$, and using $\Phi(x) + \Phi(-x) = 1$,
\begin{align*}
    M_3 = 3\int_\R x\,\varphi(x)\,\Phi(x)^2\,\diff x = 6\int_\R \varphi(x)^2\,\Phi(x)\,\diff x = 3\int_\R \varphi(x)^2\,\diff x = \frac{3}{2\sqrt\pi} .
\end{align*}
Inserted in (ii), with $\sqrt{\pi/2}\,M_3 = 3/(2\sqrt2)$ and $(2/3)^{3/2} = 2\sqrt2/(3\sqrt3)$, these values give
\begin{align*}
    \cE_2(1,a) \ge \frac{2\sqrt2}{3\sqrt3}\cdot\sqrt{6\log 2}\cdot\frac{2\sqrt2}{3} = \frac{8}{9}\sqrt{2\log 2} > 1 = \cE_1(1,a),
\end{align*}
the last equality by Theorem~\ref{thm:SE}.

(iv) By (i), it suffices to treat $d = m^2$. We identify $\R^d$ with the space of $m \times m$ matrices, with $\langle x, y\rangle = \sum_{j,l} x_{jl}\,y_{jl}$, so that $|V|$ is the Frobenius norm. Let $\varepsilon$ and $\varepsilon'$ be independent vectors of $\R^m$ with independent Rademacher coordinates, and let
\begin{align*}
    X \coloneqq b\,\varepsilon\varepsilon'^\top,
    \qquad
    b = 1 \wedge a \ \text{ for } \SG_d(1,a),
    \qquad
    b = a\,(1 - e^{-1/a^2})^{1/2} \ \text{ for } \SE_d(1,a) .
\end{align*}
We first bound the moment generating function of a projection. Let $|V| = 1$ and $b|\lambda| < 1$. Two facts are used: for $w, y \in \R^m$ and a standard Gaussian vector $g$ of $\R^m$, independent of $\varepsilon$,
\begin{align*}
    \E\left[e^{\langle w, \varepsilon'\rangle}\right] = \prod_l \cosh w_l \le e^{|w|^2/2},
    \qquad
    e^{|y|^2/2} = \E\left[e^{\langle y, g\rangle}\right],
\end{align*}
the first as in the proof of Theorem~\ref{thm:SE}. Since $\langle V, X\rangle = b\,\langle V^\top\varepsilon, \varepsilon'\rangle$, the first fact given $\varepsilon$, the second, and the first again, now for $\varepsilon$ given $g$, give
\begin{align*}
    \E\left[e^{\lambda\langle V, X\rangle}\right]
    &\le \E\left[e^{\lambda^2 b^2 |V^\top\varepsilon|^2/2}\right]
    = \E\left[e^{\lambda b\,\langle \varepsilon, Vg\rangle}\right]
    \le \E\left[e^{\lambda^2 b^2 |Vg|^2/2}\right]
    = \prod_l \left(1 - \lambda^2 b^2 \mu_l\right)^{-1/2},
\end{align*}
where the $\mu_l \ge 0$ are the eigenvalues of $V^\top V$, since $|Vg|^2$ has the law of $\sum_l \mu_l g_l^2$. As $\sum_l \mu_l = |V|^2 = 1$, and $-\log(1 - \mu y) \le -\mu\log(1 - y)$ for $\mu \in [0,1]$ by convexity,
\begin{align*}
    \prod_l \left(1 - \lambda^2 b^2 \mu_l\right)^{-1/2}
    = \exp\Bigl(-\frac12\sum_l \log\bigl(1 - \mu_l b^2\lambda^2\bigr)\Bigr)
    \le \left(1 - b^2\lambda^2\right)^{-1/2} .
\end{align*}
We now check that $X$ belongs to the class. In the sub-Gamma case, $b = 1 \wedge a$, and, since $-\log(1 - y) \le y/(1 - y)$, $b \le 1$ and $b^2\lambda^2 \le b|\lambda| \le a|\lambda|$, for $|\lambda| < 1/a$,
\begin{align*}
    -\frac12\log\left(1 - b^2\lambda^2\right) \le \frac{b^2\lambda^2}{2\,(1 - b^2\lambda^2)} \le \frac{\lambda^2}{2\,(1 - a|\lambda|)} .
\end{align*}
In the sub-exponential case, the function $u \mapsto -\log(1 - b^2 u) - u$ is convex and vanishes at $0$, and, by the choice of $b$, it vanishes at $u = 1/a^2$, so
\begin{align*}
    -\frac12\log\left(1 - b^2\lambda^2\right) \le \frac{\lambda^2}{2}
    \qquad\text{for } |\lambda| \le \frac1a .
\end{align*}
Finally, let $F(x) \coloneqq \max_{\delta,\delta' \in \{-1,1\}^m} \delta^\top x\,\delta'$. With $\delta = \varepsilon$ and $\delta' = \varepsilon'$,
\begin{align*}
    F(X) = b\,|\varepsilon|^2\,|\varepsilon'|^2 = b\,m^2,
\end{align*}
as in Figure~\ref{fig:kron}, left.
On the benchmark side, $\Lap_d = \sqrt{2W}\,G$, with $G$ an $m \times m$ matrix with independent standard Gaussian entries, and $\delta^\top G\,\delta' \le |\delta|\,|\delta'|\,\|G\| = m\,\|G\|$, where $\|G\|$ is the operator norm. With $\E[\|G\|] \le 2\sqrt m$ \citep[Theorem 7.3.1]{vershynin2018high},
\begin{align*}
    \E[F(c\Lap_d)] \le c\,\E\bigl[\sqrt{2W}\bigr]\,m\,\E\bigl[\|G\|\bigr] \le c\,\sqrt{\frac{\pi}{2}}\;2m^{3/2} = c\,\sqrt{2\pi}\,m^{3/2} .
\end{align*}
If $X \lecx c\Lap_d$, then $b\,m^2 \le c\,\sqrt{2\pi}\,m^{3/2}$, which is (iv).
\end{proof}

\begin{proof}[Proof of Theorem~\ref{thm:bench}]
Part (i) is proved in Section~\ref{sec:vectors}, after the statement; we prove (ii). The Haar probability on $O(d)$ is the unique probability measure $\mu$ on the Borel sets of $O(d)$ that is invariant under multiplication on both sides: $\mu(VA) = \mu(AV) = \mu(A)$ for every $V \in O(d)$ and every Borel set $A \subseteq O(d)$ \citep[Section 1.2]{meckes2019random}. Throughout, $\bar U \sim \mu$ is independent of $Z$, and $\Theta \sim \varsigma$.

The bounds (ii) and (iv) of Proposition~\ref{prop:dim} both come from an element $X$ of the class and a convex, positively homogeneous function $F$ with $|F(x)| \le C|x|$, through the inequality $\E[F(X)] \le c\,\E[F(\Lap_d)]$: the maximum of the vertex forms in (ii), and, in (iv), the function $F$ of the first $m^2$ coordinates, with $X$ completed by zeros as in the proof of (i). It is therefore enough to show that, for every such $X$ and $F$,
\begin{align*}
    \E[F(X)] \le c\,K(Z)\,\E[F(\Lap_d)]
    \qquad\text{whenever } cZ \text{ dominates the class} .
\end{align*}

We first average over the rotations (Figure~\ref{fig:spikes}, right). Let $U \in O(d)$. The vector $U^\top X$ belongs to the class, since $\langle v, U^\top X\rangle = \langle Uv, X\rangle$ and $Uv$ is a unit vector whenever $v$ is one, and $F \circ U$ is convex. Since $cZ$ dominates $U^\top X$, and $F$ is positively homogeneous,
\begin{align*}
    \E[F(X)] = \E\bigl[(F \circ U)(U^\top X)\bigr] \le \E\bigl[(F \circ U)(cZ)\bigr] = c\,\E\bigl[F(UZ)\bigr] .
\end{align*}
The map $(U, z) \mapsto F(Uz)$ is continuous, hence measurable, and $|F(Uz)| \le C|Uz| = C|z|$, which is integrable under the law of $Z$. Fubini's theorem therefore applies to the independent pair $(\bar U, Z)$, and integrating the last display against $\mu$ gives
\begin{align*}
    \E[F(X)] \le c\int_{O(d)} \E\bigl[F(UZ)\bigr]\,\mu(\mathrm{d}U) = c\,\E\bigl[F(\bar U Z)\bigr] .
\end{align*}

The same averaging gives $c_d(\bar UZ) \le c_d(Z)$, as stated in Section~\ref{sec:vectors}. Let $f$ be convex, and let $\ell \le f$ be affine. The argument above gives $\E[f(X)] \le \E[f(cUZ)]$ for every $U \in O(d)$. Since $cZ$ dominates a centered variable, it is centered, so that $\E[\ell(cUZ)] = \ell(0) = \E[\ell(X)]$. Tonelli's theorem, applied to the nonnegative function $f - \ell$, then gives
\begin{align*}
    \E\bigl[(f-\ell)(X)\bigr] \le \int_{O(d)} \E\bigl[(f-\ell)(cUZ)\bigr]\,\mu(\mathrm{d}U) = \E\bigl[(f-\ell)(c\bar UZ)\bigr] ,
\end{align*}
that is, $\E[f(X)] \le \E[f(c\bar UZ)]$.

We then compute $\E[F(\bar UZ)]$. Given $Z = z \ne 0$, let $u \coloneqq z/|z|$ and choose $V \in O(d)$ with $Ve_1 = u$. By invariance under multiplication on the right, $\bar UV$ has the law of $\bar U$, so $\bar Uu = \bar UVe_1$ has the law of $\bar Ue_1$, the first column of $\bar U$, which is uniform on $\mathbb{S}^{d-1}$ \citep[Section 1.2]{meckes2019random}. As $\bar U$ is independent of $Z$, the vector $\bar UZ$ has the law of $|Z|\,\Theta$, with $\Theta$ independent of $Z$, both sides vanishing when $Z = 0$. As $F$ is positively homogeneous and $\Theta$ is independent of $Z$, the first identity below holds. For the second, $\int |\langle v, u\rangle|\,\varsigma(\mathrm{d}v) = \E[|\Theta_1|]$ for every unit vector $u$, by the rotation invariance of $\varsigma$, so that, by \eqref{eq:Kbench},
\begin{align*}
    \E\bigl[F(\bar UZ)\bigr] = \E[|Z|]\,\E[F(\Theta)],
    \qquad
    K(Z) = \E\Bigl[|Z|\int \Bigl|\Bigl\langle v, \frac{Z}{|Z|}\Bigr\rangle\Bigr|\,\varsigma(\mathrm{d}v)\Bigr] = \E[|Z|]\,\E[|\Theta_1|] .
\end{align*}
The vector $\Lap_d$ is invariant under rotations, so these identities also hold with $\Lap_d$ in place of $Z$, and, since $K(\Lap_d) = 1$,
\begin{align*}
    \E\bigl[F(\Lap_d)\bigr] = \E[|\Lap_d|]\,\E[F(\Theta)],
    \qquad
    1 = \E[|\Lap_d|]\,\E[|\Theta_1|] .
\end{align*}
Combining the four identities,
\begin{align*}
    \E\bigl[F(\bar UZ)\bigr] = \frac{K(Z)}{\E[|\Theta_1|]}\,\E[F(\Theta)] = K(Z)\,\E[|\Lap_d|]\,\E[F(\Theta)] = K(Z)\,\E\bigl[F(\Lap_d)\bigr],
\end{align*}
which is \eqref{eq:homog}, and the claimed inequality follows. Letting $c \downarrow c_d(Z)$ and taking the infimum over $Z$ gives the theorem.
\end{proof}

\end{document}